\documentclass{imsart}

\usepackage[a4paper]{geometry}

\usepackage{etoolbox}

\makeatletter
\patchcmd{\@makecaption}
  {\parbox}
  {\advance\@tempdima-\fontdimen2} 
  {}{}
\makeatother

\usepackage{float}
\usepackage{amsmath}
\usepackage{graphicx,psfrag,epsf}
\usepackage{enumerate}
\usepackage{natbib}
\usepackage{url} 
\usepackage{tikz}
\usepackage{subcaption}
\usepackage{pgfplots}
\pgfplotsset{compat=1.18} 

\usepackage{cancel}  
\usepackage{color,ulem}

\usepackage[utf8]{inputenc}
\usetikzlibrary{patterns}
\usetikzlibrary{shapes}
\usepackage[linesnumbered,ruled]{algorithm2e}
\usepackage{bm}
\usepackage{dsfont}
\usepackage{amsmath,amssymb}
\usepackage{amsfonts}
\usepackage{amsthm}
\usepackage{xcolor}
\usepackage{multirow}

\usepackage{hyperref}

\newcommand{\add}[1]{\textcolor{black}{#1}}
\newcommand{\delete}[1]{}
\newcommand{\deleteeq}[1]{}
\newcommand{\deletecite}[1]{}

\newtheorem{theorem}{Theorem}[section]
\newtheorem{lemma}[theorem]{Lemma}
\newtheorem{proposition}[theorem]{Proposition}
\newtheorem{corollary}[theorem]{Corollary}
\newtheorem{remark}{Remark}[section]
\newtheorem{example}{Example}[section]
\newtheorem{assumption}{Assumption}

\newtheorem{definition}{Definition}

\newcommand{\mP}{\ensuremath{\mathbb P}}

\newcommand{\ind}[1]{{\mathbf{1}\{#1\}}}
\newcommand{\range}[1]{[#1]}
\newcommand{\R}{\mathbb{R}}
\newcommand{\E}{\ensuremath{\mathbb E}}
\renewcommand{\P}{\ensuremath{\mathbb P}}
\newcommand{\cH}{{\mathcal{H}}}
\newcommand{\wt}[1]{{\widetilde{#1}}}
\renewcommand{\l}{\ell}
\newcommand{\wh}{\widehat}

\newcommand{\FCR}{\mathrm{FCR}}

\newcommand{\FDR}{\mathrm{FDR}}
\newcommand{\FDP}{\mathrm{FDP}}

\newcommand{\FCP}{\mathrm{FCP}}

\newcommand{\BH}{\mathrm{BH}}
\newcommand{\Pow}{\mathrm{Pow}}

\newcommand{\pfull}[2]{\overline{p}^{(#1)}_{#2}}
\newcommand{\pfullbf}{\mathbf{\overline{p}}}
\newcommand{\Wfull}{\overline{W}}
\newcommand{\RinfoSP}{\mathcal R_{\alpha}^{{\tiny \textnormal{InfoSP}}}}
\newcommand{\RinfoSCOP}{\mathcal R_{\alpha}^{{\tiny \textnormal{InfoSCOP}}}}

\newcommand{\pcond}[2]{\tilde{p}^{(#1)}_{#2}}
\newcommand{\pcondbf}{\mathbf{\wt{p}}}
\newcommand{\Wcond}{\wt{W}}

\newcommand{\pzhao}[2]{\breve{p}^{(#1)}_{#2}}
\newcommand{\pzhaobf}{\mathbf{\breve{p}}}

\newcommand{\ncal}{n_{\tiny\mbox{cal}}}
\newcommand{\ntest}{n_{\tiny\mbox{test}}}

\newcommand{\dcal}{\mathcal{D}_{{\tiny \mbox{cal}}}}
\newcommand{\dtest}{\mathcal{D}_{{\tiny \mbox{test}}}}

\newcommand{\dtestX}{\mathcal{D}^X_{{\tiny \mbox{test}}}}

\newcommand{\dtrain}{\mathcal{D}_{{\tiny \mbox{train}}}}

\newcommand{\q}[2]{\pcond{#1}{#2,{\tiny \mbox{adapt}}}}

\newcommand{\Score}{S}
\newcommand{\ZZ}{\mathcal{Y}}

\begin{document}


\begin{frontmatter}
\title{Selecting informative conformal prediction sets with false coverage rate control}
\runtitle{Selecting informative conformal prediction sets}

\begin{aug}

\author[A]{ \fnms{Ulysse}~\snm{Gazin}\ead[label=e1]{ugazin@lpsm.paris}}

\author[B]{ \fnms{Ruth}~\snm{Heller}\ead[label=e2]{ruheller@gmail.com}}

\author[C]{ \fnms{Ariane}~\snm{Marandon}\ead[label=e3]{amarandon-carlhian@turing.ac.uk}}

\author[D]{ \fnms{Etienne}~\snm{Roquain}\ead[label=e4]{etienne.roquain@upmc.fr}}

\address[A]{Université Paris Cité and Sorbonne Université, CNRS, Laboratoire de Probabilités, Statistique et Modélisation, F-75013 Paris, France\printead[presep={,\ }]{e1}
}
\address[B]{Department of Statistics and Operations Research, Tel-Aviv University, Tel-Aviv, Israel\printead[presep={,\ }]{e2}
}
\address[C]{The Alan Turing Institute, London, United Kingdom\printead[presep={,\ }]{e3}
}
\address[D]{Sorbonne Université and Université Paris Cité, CNRS, Laboratoire de Probabilités, Statistique et Modélisation, F-75005 Paris, France\printead[presep={,\ }]{e4}
}
\runauthor{Gazin, Heller, Marandon, Roquain}
\end{aug}

\begin{abstract}
In supervised learning, including regression and classification, conformal methods provide  prediction sets for the outcome/label with finite sample coverage for any machine learning predictor. We consider here the case where such prediction sets come after a selection process. The selection process requires that the selected prediction sets be `informative' in a well defined sense. We consider both the classification and regression settings where the analyst may consider as informative only the sample with prediction  sets  small enough, excluding null values, or obeying other appropriate `monotone' constraints. We develop a unified framework for building such informative conformal prediction sets while controlling the false coverage rate (FCR) on the selected sample. While conformal prediction sets after selection have been the focus of much recent literature in the field, the new introduced procedures, called \texttt{InfoSP} and \texttt{InfoSCOP}, are to our knowledge the first ones providing 
FCR control for 
informative prediction sets.
We show the usefulness of our resulting procedures on real and simulated data.
\end{abstract}


\end{frontmatter}

\section{Introduction}\label{sec-intro}

 In modern data analysis, machine learning algorithms are often used to make predictions and a main challenge is to measure the  uncertainty of such methods.
Conformal inference offers an elegant solution to this problem, by providing prediction sets that provably cover the true value with high probability, for any sample size, any predictive algorithm, and any distribution of the data \cite{vovk2005algorithmic}. 
We consider the following classic split/inductive conformal prediction setting \cite{papadopoulos2002inductive, vovk2005algorithmic,lei2014conformal}.
Let $(X,Y)\in \mathcal X \times \mathcal{Y}$ be a random vector with unknown distribution
$P_{XY}$. Typically,  $\mathcal X$ is a subset of $\R^d$ (real valued covariates), and  $\mathcal{Y} =\range{K}$ (classification  among $K\geq 2$ classes) or $\mathcal{Y} = \R$ (regression, with a real valued outcome)\footnote{Our theory can be applied for more general observation spaces (e.g., $\mathcal{Y} = \R^{d'}$) but we consider the most common settings $\mathcal{Y} \in \{ \range{K}, \R\} $ for   simplicity.}.
In this setting, there are two independent samples of points $(X_i,Y_i)$: the calibration sample $\{(X_j,Y_j),j\in \range{n} \}$; and the test sample $\{(X_{n+i},Y_{n+i}), i \in \range{m}\}$. \add{We also have a prediction machine for $Y_j$ given $X_j$, built from training data that is independent of the calibration and test data.}

While all measurements are observed in the calibration sample, only the $X_i$'s are observed in the test sample and the aim is to provide  prediction sets for  the unobserved  $(Y_{n+i},i\in \range{m})$. We denote the prediction set for $Y_{n+i}$ by $\mathcal C_{n+i}$. It 
is a subset of $\range{K}$ for classification and a subset (often an interval) of $\R$ for regression. 
The classic conformal prediction set for $Y_{n+i}$,  denoted by $\mathcal C^{\alpha}_{n+i}$, is a function of $\{(X_j,Y_j),j\in \range{n} \}, X_{n+i}$, and $\alpha$ \citep{sadinle2019least, lei2018distribution} and has the following coverage guarantee, assuming that $(X_{n+i}, Y_{n+i})$ and  $(X_j,Y_j),j\in \range{n}$, are exchangeable pairs of observations from any distribution $P_{XY}$: 
\begin{equation}\label{eq-CCbasic}
\forall \alpha \in (0,1), \:\:\P\left(Y_{n+i}\in \mathcal C^{\alpha}_{n+i} \right)\geq 1-\alpha. 
\end{equation}

In practice, the size of the test sample $m$  is often large, encompassing hundreds or thousands of unlabeled examples. Inferring on all of them is unnecessary or inefficient in many applications \citep{Jin2023selection, bao2024selective}. For example, in classification, if  each image belongs to one of  $\range{K}$ categories, the analyst is not interested in the examples where $\mathcal C^{\alpha}_{n+i} = \range{K}$. Thus, it is natural to assume that  a subset of individuals will be selected. However, reporting their  prediction sets constructed to have at least   $1-\alpha$  confidence is problematic, since conditional on being selected, the coverage may be much smaller \citep{benjamini2005false, benjamini2014selective}, see also Figure~\ref{fig-classifintro} below.

Our focus is on the common setting where the analyst is only interested in reporting ``interesting'' or ``informative'' prediction sets, i.e., that cover only part of the $\mathcal Y$ space in a well defined sense to the analyst, which definition depends on the specific context. 
For concreteness, we start by providing typical examples of what can be the pre-specified collection $\mathcal{I}$ of informative subsets of $\ZZ$.

\begin{example}[Informative prediction sets in regression, $\mathcal{Y}=\R$]\label{ex:regression}
\begin{enumerate}  
    \item Intervals excluding a range of values: $\mathcal{I}=\{\mbox{$I$ interval of } \R\::\: I\cap \mathcal{Y}_0=\emptyset \}$ for some subset of null values $\mathcal{Y}_0\subset \mathcal{Y}$ that are considered as uninteresting for the user. The choice $\mathcal{Y}_0=(-\infty,y_0]$ is related to the selection proposed in \cite{Jin2023selection}, for which a ``normal'' value for the outcome is a value below $y_0$. 
   \item Length-restricted intervals: $\mathcal{I}=\{\mbox{$I$ interval of } \R\::\: |I|\leq 2\lambda_0\}$ for some $\lambda_0>0$, which are useful for  only reporting prediction intervals that are accurate enough. 
\end{enumerate}
\end{example}

\begin{example}[Informative prediction sets in classification, $\mathcal{Y}=\range{K}$]\label{ex:classif}
\begin{enumerate}
\item Excluding one class: $\mathcal{I}=\{C\subset \range{K}\::\: y_0\notin C \}$ for some null class $y_0\in \range{K}$. It is suitable when the user does not want to report prediction sets for individuals that are in class $y_0$. 
 This can be extended to excluding several classes: $\mathcal{I}=\{C\subset \range{K}\::\: C\cap \mathcal{Y}_0=\emptyset \}$ for some label set $\mathcal{Y}_0\subset \range{K}$.
    \item Non-trivial classification: $\mathcal{I}=\{C\subset \range{K}\::\: |C|\leq K-1\}$. It is appropriate when the analyst wants a label set that is  minimally informative. More generally, at most $k_0$-sized classification can be considered with $\mathcal{I}=\{C\subset \range{K}\::\: |C|\leq k_0\}$.
 \end{enumerate}
\end{example}

 A common target error guarantee is that the inference on  at most $\alpha$ examples among the selected is expected to be false. This is a classical error criterion in the selective inference literature, \citep{benjamini2005false, benjamini2014selective,weinstein20online}. It has been used, e.g.,  for novelty detection \citep{bates2023testing, marandon2024adaptive}, for classification \citep{rava2021burden,zhao2023controlling, Jin2023selection}, for regression \citep{bao2024selective}, and for unsupervised clustering \citep{mary2022error,marandon2022false}. 
 For selecting  prediction sets, the target error guarantee is thus that at most $\alpha$ examples among the selected 
 are expected to have prediction sets that do not cover their true outcome value. The false coverage proportion (FCP) for the procedure $\mathcal{R}=(\mathcal{C}_{n+i})_{i\in \mathcal{S}}$ is defined as the proportion of non-covered examples in the selected set $\mathcal{S}$:
\begin{align}
\FCP(\mathcal{R},Y)=\frac{\sum_{i\in \mathcal{S}} \ind{Y_{n+i} \notin \mathcal{C}_{n+i}} }{1\vee |\mathcal{S}|},\label{equFSP}
\end{align}
and the target error is simply its expectation,  which we refer to as the false coverage rate (FCR) as in \cite{bao2024selective}: 
\begin{align}
\FCR(\mathcal{R})&= \E[\FCP(\mathcal{R},Y)].
\label{equFSR}
\end{align}
In this work, we consider two popular models for generating the samples\footnote{In both models, the independence assumption can be relaxed.  In the iid model, it is enough that all $n+m$ samples are exchangeable. In the class-conditional model, it is enough that the subset of $\range{n+m}$ of samples from the same class is exchangeable, for all classes in $\range{K}$. See \S~\ref{sec:proofFCRcontrol} for more details.}. First, 
the {\it iid  model} (both for regression and classification): the variables $(X_i,Y_i)\sim P_{XY}$, $i\in \range{n+m}$, are all independent and identically distributed (iid). This is the standard assumption classically used for conformal prediction \citep{vovk2005algorithmic}. The parameter of the model is in that case $P_{XY}$.
Second, the {\it class-conditional model} (only the classification setting): all the labels $(Y_i, i\in \range{n+m})$ are deterministic and the covariates $(X_i, i\in \range{n+m})$ are mutually independent with $X_j\sim P_{X|Y=Y_j}$.  It relaxes the exchangeable model assumptions of iid model, by only requiring that the distribution within each class is the same for the test and calibration sample \citep{sadinle2019least, ding2023class}. The target FCR is then conditional on the labels.  The parameters of the model are then given by $(P_{X|Y=k})_{k\in \range{K}}$ and $(Y_i, i\in \range{n+m})$.

We now briefly summarize the contributions of our work.
We first introduce a new method, called \texttt{InfoSP} (Informative selective prediction sets), that selects only informative prediction sets with a level $\alpha$ FCR guarantee on the selected (\S~\ref{sec:Ralphagen}).  Formally, this means that we achieve both $\FCR(\mathcal{R})\leq \alpha$ and $\forall i\in \mathcal{S}$, $\mathcal{C}_{n+i}\in \mathcal{I}$, for $\mathcal{I}$ being the collection of informative subsets. 
The FCR control of \texttt{InfoSP} is established both in the iid model and class-conditional model (see Theorem~\ref{thm-gen-basic}).

We introduce a second procedure, called \texttt{InfoSCOP} (Informative selective conditional prediction sets), that has the same theoretical properties as \texttt{InfoSP} in the iid model (\S~\ref{subsec-preprocess}) and that relies on an initial selection step that is aimed at eliminating (at least some of) the examples for which informative prediction sets cannot be constructed. Further selection then takes place in order to ensure that all reported prediction sets are informative. While the pre-processing step is inspired by \cite{bao2024selective}, their theoretical framework  precludes this type of selection (see \S~\ref{subsec-prevworks} for more details).

\add{Third, our main theoretical FCR control guarantees come from a single general theorem in \S~\ref{sec:genBYresult} that includes general classes of $p$-values and accommodates any {\it concordant} selection rule \citep{benjamini2005false, benjamini2014selective}.}
Importantly, our novel theory supports  conformal $p$-values, and selecting only informative prediction sets, as specific examples.

We optimize the analysis pipeline for common informative selection rules in \S~\ref{sec:appliRegression} (regression) and \S~\ref{sec:classif}-\S~\ref{sec:appli} (classification), while providing additional theoretical results and appropriate numerical experiments. \add{The interest of the method is also demonstrated in a specific  application in the field of molecular biology in \S~\ref{subsec-yeast}, for which the practitioner aims at predicting gene expression from promoter sequences. While the analyst is interested in selecting promoters that cause low-expression, the produced prediction two-sided interval allows to establish that the expression is above zero and to evaluate the plausible extent to which it exceeds zero. }
 Finally, an application to directional FDR control is also provided in \S~\ref{sec:directionalFDR}.

To provide an intuition for our approach, \texttt{InfoSP} is illustrated  in Figure~\ref{fig-classifintro} for the classification case (for $K=3$ classes).
Left panels display a naive method reporting the marginal classical conformal prediction sets $\mathcal C^{\alpha}_{n+i}$ in \eqref{eq-CCbasic} for all those that are informative, that is, such that $\mathcal C^{\alpha}_{n+i}\in \mathcal{I}$, without further correction. 
Since no error can occur when the prediction set is trivial, the selection always inflates the FCP values. The new procedure \texttt{InfoSP} is displayed in the right panels: prediction sets are made slightly larger to maintain a correct FCP value while the selection (red boxes) guarantees that only informative (i.e.,  non-trivial) prediction sets are selected. This example is only for one data generation: it is presented here only for illustrative propose and more accurate in-expectation FCR values are given in \S~\ref{sec:appli}.
The regression case is illustrated in Figures~\ref{fig-regressexnonullab} and \ref{fig-regressLength} (\S~\ref{sec:appliRegression}), for which the second procedure \texttt{InfoSCOP} is also displayed. It is worth to note that in some situations, the latter may even results in prediction sets that are smaller (!) than those of the naive method.

\vspace{-5mm}

\subsection{Relation to previous work}

There are interesting connections between our approach and previous work of the literature: selecting confidence interval that satisfies specific notions of informativeness \cite{weinstein20,weinstein20online}; multi-class classification \cite{zhao2023controlling}; and very recent works on selective conformal inference \cite{bao2024selective,jin2024confidence}. Due to the limited space, the details are in\S~\ref{subsec-prevworks} of the SM.

We also underline that our second procedure \texttt{InfoSCOP} relies on the approach of splitting the calibration sample, which has already been used in conformal literature in different contexts. The idea is to enable an additional data-driven tuning of the method by only paying the price of splitting the calibration sample. For instance, it has been provided in \cite{marandon2024adaptive} for the task of building adaptive scores. In the present aim of controlling the FCR on a data-driven selection, we formulate a general statement in the SM, see Lemma~\ref{lemselective}. It applies to any type of data-driven selection and error rate.

\begin{figure}[h!]
\begin{center}
\vspace{-2mm}
\begin{tabular}{cc}
Classical conformal with naive selection & \texttt{InfoSP} \\
(FCP = 3/20 = 0.15) & (FCP = 1/10 = 0.10) \\
\includegraphics[width=0.5\textwidth, height = 0.3\textheight]{"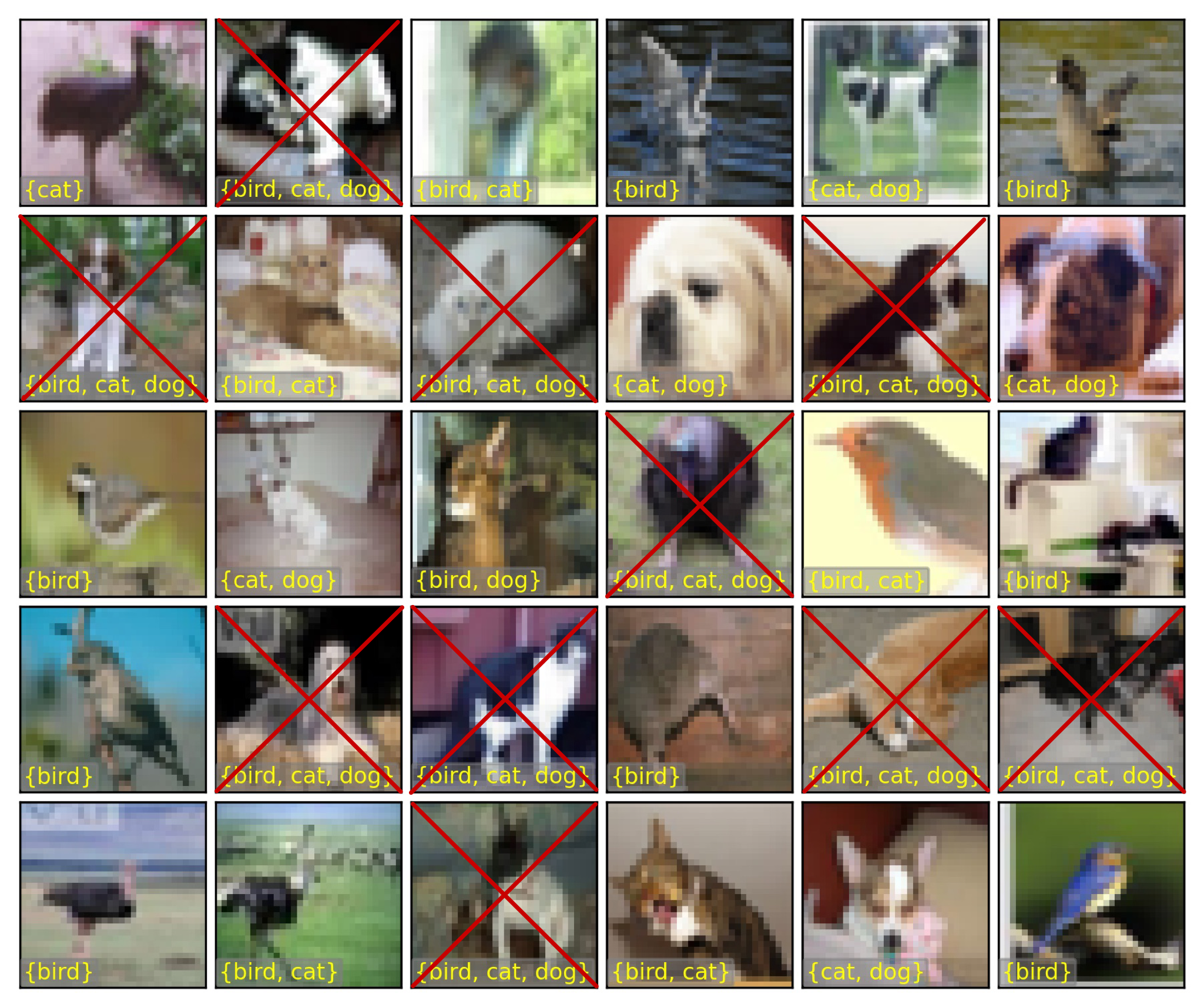"}&\hspace{-5mm}
\includegraphics[width=0.5\textwidth,  height = 0.3\textheight]{"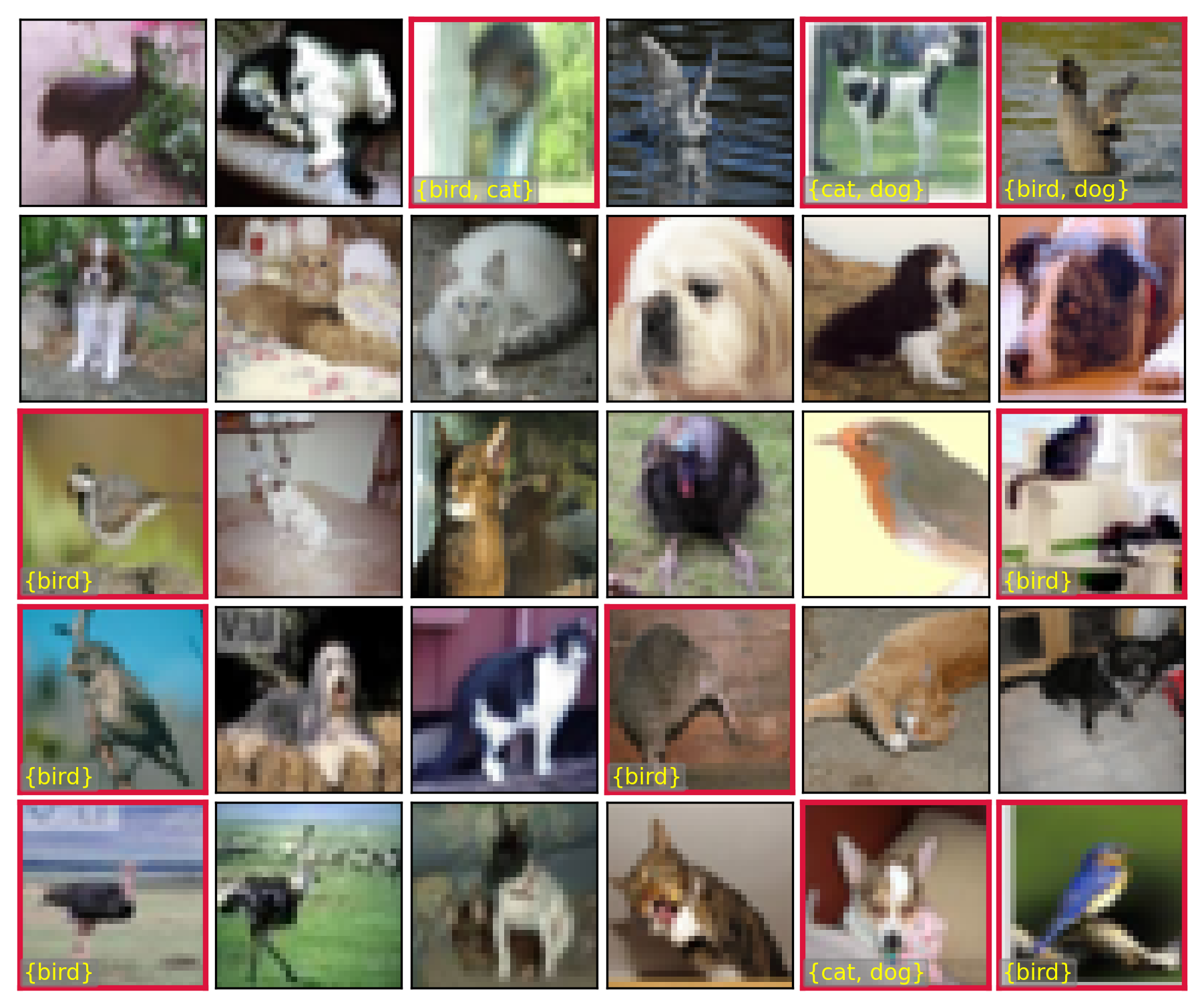"}
\end{tabular}
\vspace{-7mm}    
\end{center}
\caption{
Informative prediction sets in classification for CIFAR-10 dataset, restricted to the $K=3$ classes ``bird", ``cat", and ``dog" classes (iid setting). Informative prediction subsets are those of size at most $K-1=2$ (i.e., non-trivial, Example~\ref{ex:classif} item 1). Selection by \texttt{InfoSP} are framed in red (right panel). $\alpha=10\%$. See \S~\ref{sec:appli} for more details. 
\label{fig-classifintro}}
\vspace{-6mm}
\end{figure}

\section{Preliminaries}

\subsection{Notation}
Expectations and probabilities are denoted for the iid model  with  $\E_{(X,Y)\sim P_{XY}}(\cdot)$ and $\P_{(X,Y)\sim P_{XY}}(\cdot)$, and for the class-conditional model with $\E_{X\sim P_{X|Y}}(\cdot)$ and $\P_{X\sim P_{X|Y}}(\cdot)$, respectively. If the data generation process is clear from the context, or if the expression is relevant for both  models, then the subscript is omitted. 
In addition, $A\subset B$ means that the set $A$ is included in the set $B$ with a possible equality.
For any subset $C\subset \R$, $|C|$ denotes the cardinality of $C$ if the set $C$ is finite, and the Lebesgue measure of $C$ if $C$ is an interval. Finally, for two samples $\mathcal{D}_1$ and $\mathcal{D}_2$, $\mathcal{D}_1\cup \mathcal{D}_2$ denotes the new sample formed by concatenating the elements of $\mathcal{D}_1$ and $\mathcal{D}_2$.

\subsection{Conformal prediction sets}

The classical conformal prediction set for $Y_{n+i}$ is given by
\begin{align}\label{PSuncorrected}
\mathcal{C}^{\alpha}_{n+i}(\mathbf{p})=\{y\in \mathcal{Y}\::\: {p}^{(y)}_{i}> \alpha \},\:\:\: i \in\range{m},
\end{align}
where $\mathbf{p}=(p^{(y)}_i, y\in \mathcal{Y}, i \in\range{m})$ is a collection of conformal $p$-values satisfying the following super-uniform guarantee:
\begin{equation}\label{superunif}
    \P(p^{(Y_{n+i})}_i\leq \alpha)\leq \alpha, \:\:\:i \in\range{m},
\end{equation}
where the above probability is computed either in the iid ($Y_{n+i}$ random) or conditional ($Y_{n+i}$ fixed) model. 
Super-uniformity \eqref{superunif} implies directly that $\mathcal{C}^{\alpha}_{n+i}(\mathbf{p})$ in \eqref{PSuncorrected} provides $1-\alpha$ coverage for $Y_{n+i}$, that is,
\begin{equation}\label{equ-marginalcov}
\P(Y_{n+i}\in \mathcal{C}^{\alpha}_{n+i}(\mathbf{p}))\geq 1-\alpha, \:\:\:i\in \range{m}, \:\:\:\alpha\in (0,1), 
\end{equation}
which is generally referred  to as marginal coverage.

The $p$-value family is built from the calibration and test sample by using non-conformity score functions $S_{y}: x\in \R^d\mapsto \R$, $y\in \mathcal{Y}$,  measuring the inadequacy between $y$ and the prediction at point $x$. 
Importantly, we follow a split/inductive conformal approach, where the score functions have been computed from an independent training data sample so that they can be considered as fixed here\footnote{Note that we can relax slightly this assumption: our theory also allows this function to depend on the calibration plus test samples in an exchangeable way, see Assumption~\ref{as:AllAdaptive}.} (and all probabilities/expectations are taken conditional on that training sample).
\begin{assumption}\label{as:noties}
The score functions $\Score_{y}(\cdot)$, $y\in\ZZ$, have been computed from an independent training  sample 
 and the computed scores  $(S_{Y_i}(X_i), i\in \range{n+m} )$ have no ties almost surely. When $\mathcal{Y}=\R$, the score function is regular in the following sense: for every $x\in\R^d$, the function $y\in\R\mapsto S_y(x)\in\R$ is right continuous with left limits.
\end{assumption}

Assumption~\ref{as:noties} is a very mild assumption, which is typically satisfied. 
{For instance, for the regression case, a classical choice is the locally weighted residual function $S_y(x)=|y-\mu(x)|/\sigma(x)$ where $\mu(x)\in \mathcal{Y}$ is the predicted outcome at point $x$ and $\sigma(x)$ is the predicted standard deviation of $Y$ given $X=x$ \citep{lei2018distribution}. Another common example is the quantile-based score function $S_y(x)=\max(q_{\beta_0}(x)-y,y-q_{\beta_1}(x))$ 
where $q_\beta(x)$ is the predicted $\beta$-quantile of $Y$ given $X=x$, which corresponds to the so-called conformalized quantile regression \citep{romano2019conformalized,sesia2021conformal}  (for some prespecified $0<\beta_0<\beta_1<1$). More generally,  we refer the reader to \cite{gupta2022nested} for a general framework giving rise to a large class of score functions.} In the classification case, the typical  score is the residual function $S_y(x)=1-\pi_y(x)$ where $\pi_y(x)$ is an estimator of the probability to generate label $y$ at point $x$.  

Formally, the  $p$-value family is given as follows: 
\begin{itemize}
    \item full-calibrated $p$-values: $\pfullbf=(\pfull{y}{i}, i\in \range{m}, y\in \mathcal{Y})$, both for the regression and classification cases, with  
\begin{equation}\label{standardpvalue}
\pfull{y}{i}= \frac{1}{n+1}\Big(1+\sum_{j=1}^n \ind{S_{Y_j}(X_j)\geq S_{y}(X_{n+i})} \Big),\:\: i\in \range{m}, y\in \mathcal{Y}.
\end{equation} 
    \item class-calibrated $p$-values: $\pcondbf=(\pcond{y}{i}, i\in \range{m}, y\in \mathcal{Y})$, only for the classification case $\mathcal{Y}=\range{K}$, with
\begin{equation}\label{confpvalues0}
\pcond{y}{i}= \frac{1}{|\mathcal{D}^{(y)}_{{\tiny \mbox{cal}}}|+1}\Big(1+\sum_{j\in \mathcal{D}^{(y)}_{{\tiny \mbox{cal}}}} \ind{S_{y}(X_j)\geq S_{y}(X_{n+i})} \Big),\:\: i\in \range{m}, y\in \mathcal{Y}, 
\end{equation}
where $\mathcal{D}^{(y)}_{{\tiny \mbox{cal}}}= \{j \in \range{n}\::\: Y_j= y\}$ corresponds to the elements of the calibration sample that have label $y\in \ZZ$.    
\end{itemize}

Both $p$-values $\pfull{y}{i}$ and $\pcond{y}{i}$ are computed by examining how extreme the score $S_{y}(X_{n+i})$ is among the scores of the true labels in the calibration sample.
Full-calibrated $p$-values and class-calibrated $p$-values satisfy the super-uniformity \eqref{superunif}  in the iid model and class-conditional model, respectively, by using an exchangeability argument, see \cite{vovk2005algorithmic, RW2005, bates2023testing}. 
This means that prediction set 
$\mathcal{C}^{\alpha}_{n+i}(\mathbf{p})$ in \eqref{PSuncorrected} satisfies the marginal coverage \eqref{equ-marginalcov} in each context.
As a result, the 
false coverage rate \eqref{equFSR} for the full selection $\mathcal{S}=\range{m}$ is controlled at level $\alpha$, that is,
\begin{align}\label{globalFCRcontrol}
\E\left[\frac{\sum_{i\in \range{m}} \ind{Y_{n+i} \notin \mathcal{C}^{\alpha}_{n+i}(\mathbf{p})} }{m}\right] = m^{-1} \sum_{i\in \range{m}} \P(Y_{n+i} \notin \mathcal{C}^{\alpha}_{n+i}(\mathbf{p}))\leq \alpha.
\end{align}

\begin{remark}\label{rem:caldepth}
    The prediction set $\mathcal{C}^{\alpha}_{n+i}(\mathbf{p})$ in \eqref{PSuncorrected} can be described as  a score level set, with a threshold depending on the calibration scores. Formally, we have
    \begin{align}\label{predictionsetthreshold}
\mathcal{C}^{\alpha}_{n+i}(\mathbf{p})=\{y\in \ZZ\::\: p^{(y)}_{i} > \alpha\} = \{y\in \ZZ\::\: S_{y}(X_{n+i}) \leq  \hat{s}_\alpha\} 
\end{align}
where the score threshold $\hat{s}_\alpha$ is $S_{(\lceil (1-\alpha)(\ncal+1)\rceil)}$ with $S_{(1)}\leq \dots\leq S_{(\ncal)}$ being the ordered calibration scores (and with the convention $S_{(\ncal+1)}=+\infty$). For full-calibrated $p$-values, the $\ncal=n$ calibration scores are $S_{Y_{j}}(X_{j})$, $j\in \range{n}$. For class-conditional $p$-values, the $\ncal=|\mathcal{D}^{(y)}_{{\tiny \mbox{cal}}}|$ calibration scores are $S_{y}(X_{j})$, $j\in \mathcal{D}^{(y)}_{{\tiny \mbox{cal}}}$ ($\hat{s}_\alpha$ depends on $y$). 
\end{remark}

\subsection{$\mathcal{I}$-adjusted $p$-values}

Our theory relies on the following  assumption.
\begin{assumption}\label{assI}
The subset collection $\mathcal{I}$ is monotone in the following sense: for the considered $p$-value collection $\mathbf{p}$ (either $\pfullbf$ or $\pcondbf$), we have
    \begin{itemize}
        \item[(i)] If a prediction set is informative, then all the prediction sets it contains are also informative, that is,
for all $\mathcal{C},\mathcal{C}'$ (subsets of $\range{K}$ for classification, intervals of $\R$ for regression) with $\mathcal{C}'\subset\mathcal{C}$, $\mathcal{C}\in\mathcal{I}$ implies $\mathcal{C}'\in\mathcal{I}$.
        \item[(ii)] Almost surely, the function $\alpha\in (0,1]\mapsto \ind{\mathcal{C}^{\alpha}_{n+i}(\mathbf{p})\in \mathcal{I}} \in \{0,1\}$ is right-continuous.
   \item[(iii)] (for regression) For all $\alpha\in(0,1)$, almost surely, $\mathcal{C}^{\alpha}_{n+i}(\mathbf{p})$ is an interval of $\R$. 
          \end{itemize}
\end{assumption}
Note that Assumption~\ref{assI} implies that $\alpha\in (0,1]\mapsto \ind{\mathcal{C}^{\alpha}_{n+i}(\mathbf{p})\in \mathcal{I}} \in \{0,1\}$ is both right-continuous and nondecreasing: if $\alpha\leq \alpha'$, 
it follows that  $\mathcal{C}^{\alpha'}_{n+i}(\mathbf{p})\subset \mathcal{C}^{\alpha}_{n+i}(\mathbf{p})$ 
 from \eqref{PSuncorrected} and thus $\mathcal{C}^{\alpha}_{n+i}(\mathbf{p})\in \mathcal{I}$ implies that $\mathcal{C}^{\alpha'}_{n+i}(\mathbf{p})\in \mathcal{I}$ by Assumption~\ref{assI} (i) (iii). 

As a result, we can define the {\it $\mathcal{I}$-adjusted $p$-value vector} by $\mathbf{q}=(q_i)_{i\in \range{m}}$ where
\begin{equation}\label{qformula}
    q_i=\min\{\alpha\in (0,1]\::\: \mathcal{C}^{\alpha}_{n+i}(\mathbf{p})\in \mathcal{I}\},
\end{equation}
with by convention $q_i=1$ if the set is empty. 
Assumption~\ref{assI} can be easily checked for Examples~\ref{ex:regression} and~\ref{ex:classif}, with an explicit expression for $q_i$'s \add{(see Section~\ref{sec:qicomput} for the detailed derivations of these expressions)}.

\begin{example}[Example~\ref{ex:regression} continued]\label{ex:regression2}
For regression (see \S~\ref{sec:appliRegression} for more details):
$q_i=\sup_{y\in [a,b]} p^{(y)}_i$ for intervals  excluding $\mathcal{Y}_0=[a,b]$; $q_i= (n+1)^{-1}\big(1+\sum_{j=1}^n \ind{S_{Y_j}(X_j)>A}\big)$ for length-restricted intervals $I$ (with $|I|\leq 2\lambda_0$), where  $A=\lambda_0/\sigma(X_{n+i})$ for $S_y(x)=|y-\mu(x)|/\sigma(x)$ and $A=\lambda_0-(q_{\beta_1}(X_{n+i})-q_{\beta_0}(X_{n+i}) )/2$ for $S_y(x)=\max(q_{\beta_0}(x)-y,y-q_{\beta_1}(x))$.
\end{example}
\begin{example}[Example~\ref{ex:classif} continued]\label{ex:classif2}
For classification: $q_i=p^{(y_0)}_i$ for excluding one-class;  $q_i=\max_{y\in \mathcal{Y}_0} p^{(y)}_i$ for excluding several classes;  $q_i=\min_{y\in \range{K}} p^{(y)}_i$ for non-trivial classification;  $q_i=$  
the $(K-k_0)$-th smallest element in the set $\{p^{(y)}_i,y\in \range{K}\}$
for at most $k_0$-sized classification.
\end{example}

\begin{example}[Combining informative subset collections]\label{ex:combine}
Let $\mathcal{I}_1$ and $\mathcal{I}_2$  be two subset collections  that satisfy Assumption~\ref{assI} with adjusted $p$-values given by $(q_{1,i})_{i\in \range{m}}$ and $(q_{2,i})_{i\in \range{m}}$, respectively. Then we can easily check that the intersected collection
$
\mathcal{I} :=\{I_1\cap I_2, I_1\in \mathcal{I}_1, I_2\in \mathcal{I}_2\}
$
also satisfies Assumption~\ref{assI} with adjusted $p$-values given by $q_i=\max\{q_{1,i},q_{2,i}\}$, $i\in \range{m}$. This is useful to combine the constraints imposed by the informativeness. For instance, in the classification case (Examples~\ref{ex:classif} and \ref{ex:classif2}), we can declare  a subset as informative if it excludes a null class while it is of cardinality at most one (see \S~\ref{sec:directionalFDR} for a concrete application).
\end{example}

From its definition in  \eqref{qformula}, it follows that $q_i$ can be seen as a $p$-value to test whether  $Y_{n+i}$ lies in an informative set or not, that is, to test
\begin{equation}\label{nullhypopb}
\mbox{$H_{0,i}$: ``$Y_{n+i}\notin \cup_{C\in \mathcal{I}} C $'' versus $H_{1,i}:$ ``$Y_{n+i}\in \cup_{C\in \mathcal{I}} C $''.} 
\end{equation}
\add{More specifically, using the $p$-value $q_i$,   $H_{0,i}$ is rejected at level $\alpha$ whenever $\mathcal{C}^{\alpha}_{n+i}(\mathbf{p})\in \mathcal{I}$, and it satisfies a valid super-uniformity property, see Remark~\ref{rem:adjpvalues}.}

In cases where being informative is linked to  particular values in $\ZZ$, this testing problem takes an especially meaningful form. In classification, for excluding one class in classification:  $q_i=p^{(y_0)}_i$ tests $H_{0,i}$: ``$Y_{n+i}=y_0$'' versus $H_{1,i}:$ ``$Y_{n+i}\neq y_0$''. In regression, for excluding $\mathcal{Y}_0=[a,b]$: $q_i$ tests $H_{0,i}$: ``$Y_{n+i}\in [a,b]$'' versus $H_{1,i}:$ ``$Y_{n+i}\notin [a,b]$''.
Note that in case where informative sets are those with small size (e.g., non-trivial classification or length-restricted intervals), we have $\cup_{C\in \mathcal{I}} C = \ZZ$, so the null hypothesis is always false, and the testing problem \eqref{nullhypopb} is not interesting. \delete{However, in case the small size is just one of the criteria for being informative, as in Example~\ref{ex:combine}, then the testing problem \eqref{nullhypopb}  may still be meaningful. }

\begin{remark}
Assumption~\ref{assI} implies that the $\mathcal{I}$-adjusted $p$-value vector $\mathbf{q}$ is a nondecreasing function of the $p$-value collection, see Lemma~\ref{lem:monotonicityq}; Assumption~\ref{assI} (iii) is always satisfied up to taking as prediction sets $\mathcal{C}^{\alpha}_{n+i}(\mathbf{p})$ the convex hull of the set in the right-hand-side of \eqref{PSuncorrected}. It is also satisfied without modifying $\mathcal{C}^{\alpha}_{n+i}(\mathbf{p})$ for any score function such that all the sets $\{y\in \R\::\:S_y(x)\leq s\}$ are intervals, which is often the case, see Remark~\ref{rem:caldepth} and \S~\ref{sec:appliRegression}.
\end{remark}

\begin{remark}
\label{rem:asaf}
In a context of building online {\it confidence} intervals, \cite{weinstein20online} have proposed to report only  intervals that are ``good'' in the sense that they ``localize'' the signal. Formalizing their proposal in our regression setting and with our notation, this corresponds to consider the informative collection $\mathcal{I}=\{\mbox{$I$ interval of $\R$}\::\: I\subset \mathcal{C}_\l \mbox{ for one $\l\in \range{L}$}\},  $
where $\mathcal{C}_\l$, $\l\in \range{L}$, are pre-specified disjoint subsets of $\R$. This collection satisfies Assumption~\ref{assI} and the corresponding $\mathcal{I}$-adjusted $p$-values are given by 
$q_i=\min_{\ell\in\range{L}}\sup_{y\notin\mathcal{C}_\ell} p_i^{(y)}$. 
\end{remark}

\add{
\begin{remark}\label{rem:adjpvalues}
The following super-uniformity property holds for the $\mathcal{I}$-adjusted $p$-value $q_i$ \eqref{qformula} and the null $H_{0,i}$ \eqref{nullhypopb}:
$$
P(q_i\leq \alpha, \mbox{$H_{0,i}$ true})\leq \alpha,\:\:\mbox{ for all $\alpha\in (0,1)$.}
$$
To see this, note that when $H_{0,i}$ is true, we have $Y_{n+i}\in \cap_{C \in \mathcal{I}} C^c$. From the definition, $q_i\leq \alpha$ means  $\mathcal{C}^{\alpha}_{n+i}(\mathbf{p})\in \mathcal{I}$. Therefore, if both $q_i\leq \alpha$ and $H_{0,i}$ is true, we have $Y_{n+i}\notin \mathcal{C}^{\alpha}_{n+i}(\mathbf{p})$, i.e., that 
$p^{(Y_{n+i})}_i\leq \alpha$, which occurs with probability at most $\alpha$ by \eqref{superunif}.
\end{remark}
}

\subsection{Aim: informative selection with FCR guarantees}

 The general inferential task is to produce prediction sets for examples of interest in the test sample, that is, after selection. The selection process  is driven by the requirement that  the prediction sets be informative, and the requirement of a relevant error control. 

A {\it selective prediction set procedure} is a function of the observations $\{(X_j,Y_j),j\in \range{n} \}$, $\{X_{n+i}, i \in \range{m}\}$ of the form  $\mathcal{R}=(\mathcal{C}_{n+i})_{i\in \mathcal{S}}$ where $\mathcal{S}$ is the selected subset of $\range{m}$ for which prediction sets are constructed, and $\mathcal{C}_{n+i}\subset \mathcal{Y}$ is the prediction set for $Y_{n+i}, i\in \mathcal S$.

A selective prediction set procedure $\mathcal{R}=(\mathcal{C}_{n+i})_{i\in \mathcal{S}}$ {\it is said to be $\mathcal{I}$-informative (or informative)} if the selection $\mathcal{S}$ is a subset of $\range{m}$ that imposes that $\mathcal{C}_{n+i}$ is informative, that is, $\forall i\in \mathcal S$, $\mathcal{C}_{n+i}\in \mathcal{I}$.

The FCR in the iid model and class-conditional model, respectively, are 
\begin{align}
\FCR(\mathcal{R},P_{XY})&= \E_{(X,Y)\sim P_{XY}}[\FCP(\mathcal{R},Y)]
\label{equFSRmixturemodel};\\
\FCR(\mathcal{R},P_{X\mid Y},Y)&= \E_{X\sim P_{X\mid Y}}[\FCP(\mathcal{R},Y)],\label{equFSRcondmodel}
\end{align}
for the FCP in \eqref{equFSP}. The FCR expression in \eqref{equFSRcondmodel} for 
quantifying the  errors among the selected is classical in the selective inference literature: since $(Y_{n+i})_{i\in \range{m}}$ is fixed, this is the false coverage rate on the parameters \citep{benjamini2005false, benjamini2014selective}.  On the other hand, in  \eqref{equFSRmixturemodel}, 
the false coverage rate is on random outcomes (i.e., $(Y_{n+i})_{i\in \range{m}}$ in our setting) rather than on parameters, which is the usual setting in selective conformal inference (see references in \S~\ref{sec-intro}) and is related to the Bayes FDR criterion in the multiple testing literature, see, e.g., \cite{ETST2001}.

We will focus on finding selective prediction set procedures $\mathcal{R}=\mathcal{R}_\alpha$ with either of the two following controls:
\begin{align}
\sup_{P_{X,Y}}\{\FCR(\mathcal{R},P_{X,Y})\}\leq \alpha \:\: ;\label{iidcontrol}\\
\sup_{P_{X\mid Y},Y}\{\FCR(\mathcal{R},P_{X\mid Y},Y)\}\leq \alpha\:\:.\label{strongcontrol}
\end{align}
Obviously, the class-conditional control \eqref{strongcontrol} (considered only for classification) is stronger than the unconditional control \eqref{iidcontrol} (considered both for classification and regression). 

To balance with FCR control, we also consider  the {\it{resolution adjusted power}}:
\begin{align}
\Pow(\mathcal{R})&= \E\left[ \sum_{i\in \mathcal{S}} \frac{\ind{Y_{n+i}\in \mathcal{C}_{n+i}}}{|\mathcal{C}_{n+i}|} \right].\label{equPower}
\end{align}
Hence, for the same selection set, a decision with a smaller covering  decision set $\mathcal{C}_{n+i}$ yields  higher power. 
Our aim is to maximize the resolution adjusted power (i.e., informally, to select as many examples as possible that are informative, with as narrow as possible a prediction set for the selected examples), while controlling the FCR at a pre-specified level $\alpha$.

Finally, considering the multiple testing problem \eqref{nullhypopb}, that test if $Y_{n+i}$ lies in an informative set or not, we can also quantify the error amount of a given selection procedure $\mathcal{S}$ (by itself, without quantifying the non-coverage errors of the attached prediction sets), by letting \citep{BH1995}
\begin{align}
\FDP(\mathcal{S},Y) &= \frac{\sum_{i\in \mathcal{S}} \ind{Y_{n+i} \notin \cup_{C\in \mathcal{I}} C }}{1\vee |\mathcal{S}|}.\label{equFDP}
\end{align}
For instance, when one wants to exclude a given label set $\mathcal{Y}_0$ in classification/regression, we have $\FDP(\mathcal{S},Y) = (\sum_{i\in \mathcal{S}} \ind{Y_{n+i} \in \mathcal{Y}_0})/(1\vee |\mathcal{S}|)$.
The false discovery rates in the iid model and class-conditional model are the corresponding expectations
\begin{align}
\FDR(\mathcal{S},P_{XY})&= \E_{(X,Y)\sim P_{XY}}[\FDP(\mathcal{S},Y)]\label{equFDRmixturemodel};\\
\FDR(\mathcal{S},P_{X\mid Y},Y)&= \E_{X\sim P_{X\mid Y}}[\FDP(\mathcal{S},Y)]\label{equFDRcondmodel},
\end{align}
 respectively. The following lemma holds.

\begin{lemma} \label{lem:FDRsmallerFCR}
    For any selective prediction set procedure $\mathcal{R}=(\mathcal{C}_{n+i})_{i\in \mathcal{S}}$ that is $\mathcal{I}$-informative, we have
$ 
    \FDP(\mathcal{S},Y)\leq \FCP(\mathcal{R},Y).
       $
\end{lemma}

It comes directly from the fact that, for $i\in \mathcal{S}$,  $Y_{n+i}\in \mathcal{C}_{n+i}$ implies $Y_{n+i} \in \cup_{C\in \mathcal{I}} C $ because  $\mathcal{C}_{n+i}\in \mathcal{I}$. 
As a result, producing an informative selective prediction set procedure that controls the FCR at level $\alpha$  immediately ensures that the attached selection procedure controls the FDR at level $\alpha$ and thus comes with a relevant interpretation. 

\section{Main results}
\label{sec:BasicMethod}

\subsection{Informative selective prediction sets (\texttt{InfoSP})}\label{sec:Ralphagen}

In order to have a level $\alpha$ FCR guarantee on the selected examples, it is necessary to correct the threshold  $\alpha$ in  $\mathcal{C}^{\alpha}_{n+i}(\mathbf{p})$ \eqref{PSuncorrected}. A standard approach in the selective inference literature \citep{benjamini2005false, benjamini2014selective} is to use the reduced level $\alpha |\mathcal S|/m$ for a selection set $\mathcal S$. In order for the selective prediction set procedure to be $\mathcal I$ informative,  the selection rule needs to be carefully selected. 
We shall use the following basic observation: selection by a thresholding rule on the  family $\mathbf{q}=(q_i,i\in \range{m})$ given by \eqref{qformula} will result in selected examples that are $\mathcal I$ informative for prediction sets that are constructed at a level that is at least at the selection threshold, since $\mathcal{C}^{t}_{n+i} \in \mathcal{I}$ if and only if $q_i\leq t$ for all $t\in (0,1]$ by the definition of $q_i$. Combining the standard approach for FCR control with this basic observation, leads us to use the following selection thresholding rule which is necessarily $\mathcal I$-informative: all examples with $q_i$ at most $$\max \left \lbrace t: t\leq \alpha \frac{\left (\sum_{j=1}^m\ind{q_j\leq t} \right)}{m}\right\rbrace.$$ This is exactly the BH selection rule \citep{BH1995} on the adjusted $p$-value vector $\mathbf{q}=(q_i,i\in \range{m})$. 
In practice, let us recall that the BH procedure $\BH(\mathbf{q})$ can  be obtained as
    \begin{equation}\label{equBH}
        \BH(\mathbf{q})= \{i\in \range{m}\::\: q_i\leq \alpha \hat{\ell}/m\}, 
    \end{equation}
    where $\hat{\ell}=|\BH(\mathbf{q})|=\max\{\ell\in \range{m}\::\: q_{(\ell)}\leq \alpha \ell/m\}$ (with $\hat{\ell}=0$ if the set is empty) and where $q_{(1)}\leq \dots \leq q_{(m)}$ denote the ordered $q_i$'s.

\begin{definition}\label{def:basic}
The informative selective prediction set procedure (\texttt{InfoSP}) based on a $p$-value family $\mathbf{p}=(p^{(y)}_i, y\in \mathcal{Y}, i \in\range{m})$, 
is defined as $\RinfoSP(\mathbf{p})=(\mathcal{C}^{\alpha|\BH(\mathbf{q})|/m}_{n+i}(\mathbf{p}))_{i\in \BH(\mathbf{q})}$, that is, is given as follows:
\begin{enumerate}
    \item Apply the BH procedure on the corresponding adjusted vector $\mathbf{q}=(q_i,i\in \range{m})$ given by \eqref{qformula},\eqref{equBH} and select $\mathcal{S}(\mathbf{p})=\BH(\mathbf{q}) \subset \range{m}$;
    \item For each $i\in \mathcal{S}(\mathbf{p})$, consider the prediction set $\mathcal{C}^{\alpha|\mathcal S(\mathbf{p})|/m}_{n+i}(\mathbf{p})$ for $Y_{n+i}$, computed according to \eqref{PSuncorrected}.
\end{enumerate}
\end{definition}

The key theoretical difficulty in proving that the FCR of \texttt{InfoSP} is at most $\alpha$, is that the conformal $p$-values are dependent, and that the selection step is also $p$-value-based (by contrast with \citealp{bao2024selective}). Thus,  the error rate in \eqref{equFSRmixturemodel} for the iid model and 
\eqref{equFSRcondmodel} for the class-conditional model may not be controlled.   \cite{bates2023testing, marandon2024adaptive}  showed  that, for outlier detection, the special positive dependency between the conformal $p$-values is such that the BH procedure remains valid for FDR control. 
Their results are on a different set of conformal $p$-values, but we develop a similar result for our problem, which enables us to establish the desired error guarantee for various selection rules.

\begin{theorem}\label{thm-gen-basic}
Consider score functions satisfying Assumption~\ref{as:noties}, an informative subset collection $\mathcal{I}$ satisfying Assumption~\ref{assI}  and a $p$-value collection $\mathbf{p}$ being either $\pfullbf$ (full-calibrated) or $\pcondbf$ (class-calibrated), then the $\mathcal I$-informative selective prediction set procedure $\RinfoSP(\mathbf{p})$ (Definition~\ref{def:basic})  controls the FCR at level $\alpha$, respectively in the iid model ($\mathbf{p}=\pfullbf$) with the control \eqref{iidcontrol} or the class-conditional model ($\mathbf{p}=\pcondbf$) with the control \eqref{strongcontrol}. 
\end{theorem}

A proof is provided in \S~\ref{sec:thm-gen-basic}, which relies on a more general result, see Theorem~\ref{th:generalBY}. 
\add{
The latter provides FCR guarantee to more general $p$-value collections (see Assumption~\ref{as:generalpvalues}). For instance, it is valid  
for the two $p$-value collections $\pfullbf$ and $\pcondbf$ under less restrictive conditions that those considered in Theorem~\ref{thm-gen-basic}: the independence assumption can be relaxed to an exchangeable assumption (Propositions~\ref{prop:iidpvalues}~and~\ref{prop:condpvalues}), while the score  function can take a general form that can provide an extra improvement (Assumption~\ref{as:AllAdaptive}).
In addition, our result 
applies beyond informative selection, to any {\it concordant} selection rule (see Assumption~\ref{as:nonincreasingSelect}), as defined in \cite{benjamini2005false,benjamini2014selective}. 
For instance, it applies for any monotone $p$-value-based thresholding rule (see Section~\ref{sec:as:nonincreasingSelect} for more details). This is in sharp contrast with the work of \cite{bao2024selective} that does not provide finite sample FCR control at the desired level for $p$-value-based selection rules, because the selection rule therein should be independent of the calibration scores. 
}

\add{Let us also mention that an alternative proof of Theorem~\ref{thm-gen-basic} can be obtained by using the property that a particular $p$-value based family is PRDS (positive regression dependence on each one from a subset \citealp{BY2001}), see Section~\ref{sec:alternativeproof}. This can be seen as an extension of Lemma~5 (iii) in \cite{Jin2023selection} to our general informative setting.}

We apply Lemma~\ref{lem:FDRsmallerFCR} to obtain the following FDR guarantee for \texttt{InfoSP}.

\begin{corollary}
    In the setting of Theorem~\ref{thm-gen-basic}, the  selection rule $\BH(\mathbf{q})$ of $\RinfoSP(\mathbf{p})$ provides level $\alpha$ FDR control,  given either by \eqref{equFDRmixturemodel} in the iid model ($\mathbf{p}=\pfullbf$) or \eqref{equFDRcondmodel} in the class-conditional model ($\mathbf{p}=\pcondbf$).  
\end{corollary}

\begin{remark}\label{rem:iterativeBH}
For  \texttt{InfoSP}, we can  avoid the computation of $\mathbf{q}$ by using that the BH procedure is the iterative limit of a recursion \citep{gao2023constructive}. Indeed, since $\mathcal{C}^{\alpha k/m}_{n+i} \in \mathcal{I}$ if and only if $q_i\leq \alpha k/m$, the selection $S=\BH(\mathbf{q})$ can be obtained as follows:
 Step 1: $\mathcal{S}_1=\{i\in \range{m}\::\: \mathcal{C}^{\alpha}_{n+i} \in \mathcal{I} \}$; Step $t\geq 1$:  
    $\mathcal{S}_t=\{i\in \mathcal{S}_{t-1}\::\: \mathcal{C}^{\alpha|\mathcal{S}_{t-1}|/m}_{n+i} \in \mathcal{I} \}$;
Consider $t_0$ the first $t$ where $\mathcal{S}_t=\mathcal{S}_{t-1}$ and let $\mathcal{S}= \mathcal{S}_{t_0}$.
\end{remark}

\subsection{Informative selective conditional prediction sets (\texttt{InfoSCOP})}
\label{subsec-preprocess}

Throughout this section, we consider the iid model.
It turns out that \texttt{InfoSP} can be too  conservative in some contexts;
 this is manifest in the true FCR being  much smaller than the nominal $\alpha$ level, see illustrations in \S~\ref{sec:appliRegression} and \S~\ref{sec:classif}. To avoid this, we can adapt the conditional approach of \cite{bao2024selective} to our framework and start by selecting test samples and calibration samples. We would like, following the initial selection, to have as few as possible test samples for which $\mathcal I$-informative prediction sets cannot be constructed. For this purpose,  we further split the calibration sample, which is a classical trick in conformal literature  in order to preserve exchangeability with calibration samples after the initial selection (see Lemma~\ref{lemselective}).
Specifically, we follow the following steps
\begin{enumerate}
    \item Split the calibration sample $((X_j,Y_j),j\in \range{n})$ into two samples $((X_j,Y_j),j\in \range{r})$ and $((X_j,Y_j),j\in \range{r+1,n})$ for some $r\in \range{n-1}$.
    \item Apply an initial conformal selection rule $\mathcal{S}^{(0)}=\mathcal{S}^{(0)}((X_j,Y_j)_{j\in \range{r}},(X_j)_{j\in \range{r+1,n+m}})\subset \range{r+1,n+m}$ that considers as calibration sample $((X_j,Y_j),j\in \range{r})$ and test sample $((X_j,Y_j),j\in \range{r+1,n+m})$.   
    \item For  $i+n\in \mathcal S^{(0)}\cap\range{n+1,n+m}$, compute the conformal $p$-values using calibration set $\{(X_j,Y_j),j\in \mathcal S^{(0)}\cap\range{r+1,n}\}$ (i.e., using the conditional empirical distribution, post initial selection):
 \begin{equation}
     \pfull{0,y}{i}=\frac{1}{|\mathcal{S}^{(0)}\cap \range{r+1,n}|+1}\left(1+\sum_{j\in \mathcal{S}^{(0)}\cap \range{r+1,n}}\ind{S_{Y_{j}}(X_j)\geq S_{y}(X_{n+i})}\right).
\label{equprepropvalues}
 \end{equation} 
\end{enumerate}

We assume that the initial selection $\mathcal{S}^{(0)}$ satisfies the following permutation preserving assumption.
\begin{assumption}\label{as:S0}
    For any permutation $\sigma$ of $\range{r+1,n+m}$, $$\mathcal{S}^{(0)}((X_j,Y_j)_{j\in \range{r}}),(X_{\sigma(j)})_{j\in \range{r+1,n+m}})=\sigma\big(\mathcal{S}^{(0)}((X_j,Y_j)_{j\in \range{r}},(X_j)_{j\in \range{r+1,n+m}})\big).$$
\end{assumption}

The initial selection $\mathcal{S}^{(0)}$ is typically the result of a  multiple testing procedure that is applied to the examples in $\range{r+1,n+m}$, that uses $p$-values computed with  $((X_j,Y_j),j\in \range{r})$ as calibration sample and $(X_i, i\in \range{r+1,n+m})$ as covariate test sample, which immediately satisfies Assumption~\ref{as:S0}. 
We can use $\BH(\mathbf{q})$ as an initial selection stage (where the $q_i$'s are computed with the aforementioned sample split) so that selected examples are likely to correspond to $Y_{n+i}$ where an informative prediction set can be built. For excluding a null range in regression or a null class in classification, another choice is to use an appropriate BH procedure for testing that the examples from $\range{r+1,n+m}$ are from that null, see examples in \S~\ref{sec:excludeab} and \S~\ref{subsec-simul-bivariatenormal}.

\begin{definition}\label{def:prepro}
The informative selective conditional prediction set procedure pre-processed with the initial selection rule $\mathcal{S}^{(0)}$, called \texttt{InfoSCOP}, is defined as the \texttt{InfoSP} procedure of Definition~\ref{def:basic} applied with the pre-processed $p$-value family $\pfullbf^0$ \eqref{equprepropvalues}, that is, 
$\RinfoSCOP(\pfullbf)=\RinfoSP(\pfullbf^0)=(\mathcal{C}^{\alpha^0}_{n+i}(\pfullbf^0))_{i\in \BH(\mathbf{q}^0)},$
where $\alpha^0=\alpha|\BH(\mathbf{q}^0)|/|\mathcal S^{(0)}\cap\range{n+1,n+m}|$ and the $\mathcal{I}$-adjusted $p$-values $\mathbf{q}^0$ are computed via \eqref{qformula} from the pre-processed $p$-values $\pfullbf^0$. 
\end{definition}

 \begin{theorem}
     \label{thm-iid-withpreprocess}
     Consider the iid model (both for regression and classification), score functions satisfying Assumption~\ref{as:noties}, an informative subset collection $\mathcal{I}$ satisfying Assumption~\ref{assI}  and any initial selection rule $\mathcal{S}^{(0)}\subset \range{r+1,n+m}$ that satisfies Assumption~\ref{as:S0}. Then the \texttt{InfoSCOP} procedure of Definition~\ref{def:prepro} is such that
$ \FCR(\RinfoSCOP(\pfullbf)) \leq \alpha$. In addition, the associated selection rule $\BH(\mathbf{q}^0)$ controls the FDR \eqref{equFDRmixturemodel} at level $\alpha$.
 \end{theorem}
The proof, provided in \S~\ref{proof:thm-iid-withpreprocess}, follows directly from the general calibration splitting trick  Lemma~\ref{lemselective}.

As we will see in the next sections,  while it maintains the FCR guarantee, \texttt{InfoSCOP} can greatly improve over \texttt{InfoSP}. 
The main reason is that adjusting for selection is cheaper after the initial selection step: by reducing the fraction of examples in the test sample for which it is not possible to construct $\mathcal I$-informative prediction sets, the correction term $\alpha^0$ is expected to be close to $\alpha$ (or not much smaller than $\alpha$). Another reason is that the selection-conditional $p$-values $\pfullbf^0$ will be better in settings where the initial selection step tends to remove the examples from the calibration set that have large non-conformity scores, see  \S~\ref{sec:appliRegression}, \S~\ref{sec:moreregressionillustration}, and \S~\ref{sm_simulations} for such cases. 

In general, the way \texttt{InfoSCOP} can improve over \texttt{InfoSP} depends on the context.  For instance, for excluding a null range in  classification, a null class in regression or for length restriction in regression, we show in \S~\ref{sec:appliRegression} and \S~\ref{sec:classif}, respectively, the great potential advantage of using the initial selection. On the other hand, we also show that for non-trivial classification, there may be no advantage of initial selection (in fact, there can be a slight disadvantage since the calibration sample after initial selection is smaller, as demonstrated in \S~\ref{subsec:min-iid}).

\section{Application to regression}\label{sec:appliRegression}

This section is devoted to the regression case (that is, $\ZZ=\R$), as already introduced in Example~\ref{ex:regression}. Throughout the section, we work in the iid model. \add{Illustrations for other informative selections and an application for gene expression prediction from promoter sequences can be found in \S~\ref{sec:moreregressionillustration} in the SM.}

\subsection{Excluding $[a,b]$ from the prediction interval}\label{sec:excludeab}

We focus here on the case where the user wants to build prediction intervals only for outcomes such that $Y_{n+i}<a$ or $Y_{n+i}> b$, which corresponds to excluding  $\mathcal{Y}_0=[a,b]$ from the prediction interval, where $a<b$ are two benchmark values. This corresponds to common practice where users are interested only in reporting prediction intervals for individuals with ``abnormal'' outcomes. 
Setting $a=-\infty$ recovers the case where we only want to report prediction intervals for examples such that $Y_{n+i}>b$, which is the selection considered in \cite{Jin2023selection}. We focus on two-sided prediction intervals here (the case of one-sided prediction intervals is postponed to \S~\ref{sec:moreregressionillustration}). The choice of score function defines \texttt{InfoSP} and entails all the desired inferential guarantees. We formalize this for the locally weighted score function in Corollary \ref{cor:abexcluded}. Using other score functions is also possible, e.g., the score function that corresponds to conformal quantile regression, see Remark \ref{rem-excludingintervals-scores}.

\begin{corollary}\label{cor:abexcluded}
Consider the iid model in the regression case, the locally weighted  score function $S_y(x)=|\mu(x)-y|/\sigma(x)$ and suppose that Assumption~\ref{as:noties} holds. Then the following holds for \texttt{InfoSP} with informative collection $\mathcal{I}=\{\mbox{$I$ interval of } \R\::\: I\cap [a,b]=\emptyset \}$ and full-calibrated $p$-value collection $\pfullbf$ \eqref{standardpvalue}:
\begin{itemize}
    \item[(i)] \texttt{InfoSP} selects $\mathcal{S}=\BH(\mathbf{q})$ with 
    \begin{equation}\label{equqiband}
     q_i=\pfull{a}{i}\ind{\mu(X_{n+i})<a} + \pfull{b}{i}\ind{\mu(X_{n+i})>b}+ \ind{a \leq \mu(X_{n+i})\leq b}.   
    \end{equation}
    \item[(ii)] The selection $\mathcal{S}$ of \texttt{InfoSP}  controls the FDR at level $\alpha$ in the following sense:
    $$
    \sup_{P_{XY}} \E_{(X,Y)\sim P_{XY}}\left[\frac{\sum_{i\in \mathcal{S}} \ind{Y_{n+i} \in [a,b]}}{1\vee |\mathcal{S}|}\right]\leq \alpha.
    $$
    \item[(iii)] The selected prediction intervals do not \add{intersect} $[a,b]$ and are of the form
    $
    \mathcal{C}_{n+i}= [\mu(x)-S_{(n_\alpha(\mathbf{p}))} \sigma(x), \mu(x) + S_{(n_\alpha(\mathbf{p}))} \sigma(x) ] $, 
    where $S_{(1)}\leq \dots\leq S_{(n)}$ are the ordered calibration scores $S_{Y_j}(X_j)$, $1\leq j\leq n$ (with $S_{(n+1)}=+\infty$), and $n_\alpha(\mathbf{p})=\lceil (1- \alpha|\mathcal S(\mathbf{p})|/m)(n+1)\rceil$.
    \item[(iv)] These prediction intervals control the FCR at level $\alpha$ in the sense of \eqref{iidcontrol}.
\end{itemize}
\end{corollary}

\begin{proof}
\add{Point (i) follows from straightforward computations, see Section~\ref{sec:qicomput}. }
Point (ii) follows from Lemma~\ref{lem:FDRsmallerFCR}, (iii) from Remark~\ref{rem:caldepth} and the fact that $q_i\leq \alpha|\mathcal S(\mathbf{p})|/m$ iff $\mathcal{C}_{n+i}$ does not \add{intersect} $[a,b]$. Point (iv) follows from Theorem~\ref{thm-gen-basic}.
\end{proof}

Hence, our method, in addition to providing an FCR control on the selected \delete{(iv)} ensures that the obtained prediction intervals are informative in the sense that they do not include benchmark values (i.e., values in $[a,b]$). This ensures that the selection method is meaningful for the considered aim, which formally entails the FDR control (ii).  
Obviously, since \texttt{InfoSCOP} is an \texttt{InfoSP} method for preprocessed $p$-values \eqref{equprepropvalues}, a similar result holds for \texttt{InfoSCOP}, for any initial selection step $\mathcal{S}^{(0)}\subset \range{r+1,n+m}$ that satisfies the permutation preserving Assumption~\ref{as:S0}. 

Corollary~\ref{cor:abexcluded} is illustrated on Figure~\ref{fig-regressexnonullab} when $\mathcal{S}^{(0)}$ is taken here as $\BH(\mathbf{q})$ at level $2\alpha$ (with the score $S_y(x)=|\mu(x)-y|/\sigma(x)$). In the first row, errors are larger further away from $[a,b]$. Hence, while the marginal prediction intervals control the FCR at level $\alpha$ when selecting all the covariates (as granted by \eqref{globalFCRcontrol}), the FCR is inflated for a naive selection that selects each example with a prediction interval at level $\alpha$ not intersecting $[a,b]$ (that is, naive selection is given by $\mathcal{S}_1$ in the recursion of Remark~\ref{rem:iterativeBH}). To maintain the FCR control at level $\alpha=0.1$, \texttt{InfoSP} adjusts the width of the prediction interval 
in an accurate way to accommodate the informative constraint. \texttt{InfoSCOP} is roughly the same as \texttt{InfoSP} in this case, because the largest scores are kept in the calibration sample after initial selection. This is the most unfavorable setting for \texttt{InfoSCOP}, but it  nevertheless performs similarly to \texttt{InfoSP}. In the second row, errors are smaller further away from $[a,b]$,  which makes the FCR of the naive selection and \texttt{InfoSP} far too conservative. By contrast, the initial selection of \texttt{InfoSCOP} removes the largest scores of the calibration sample, resulting in  much narrower prediction intervals, and thus in a much larger resolution-adjusted power (even better than the naive procedure).

\begin{figure}[h!]
\begin{tabular}{ccc}
Classical conformal & \texttt{InfoSP}  & \texttt{InfoSCOP}
\\
\hspace{-9mm}\includegraphics[width=0.38\textwidth, height = 0.24\textheight]{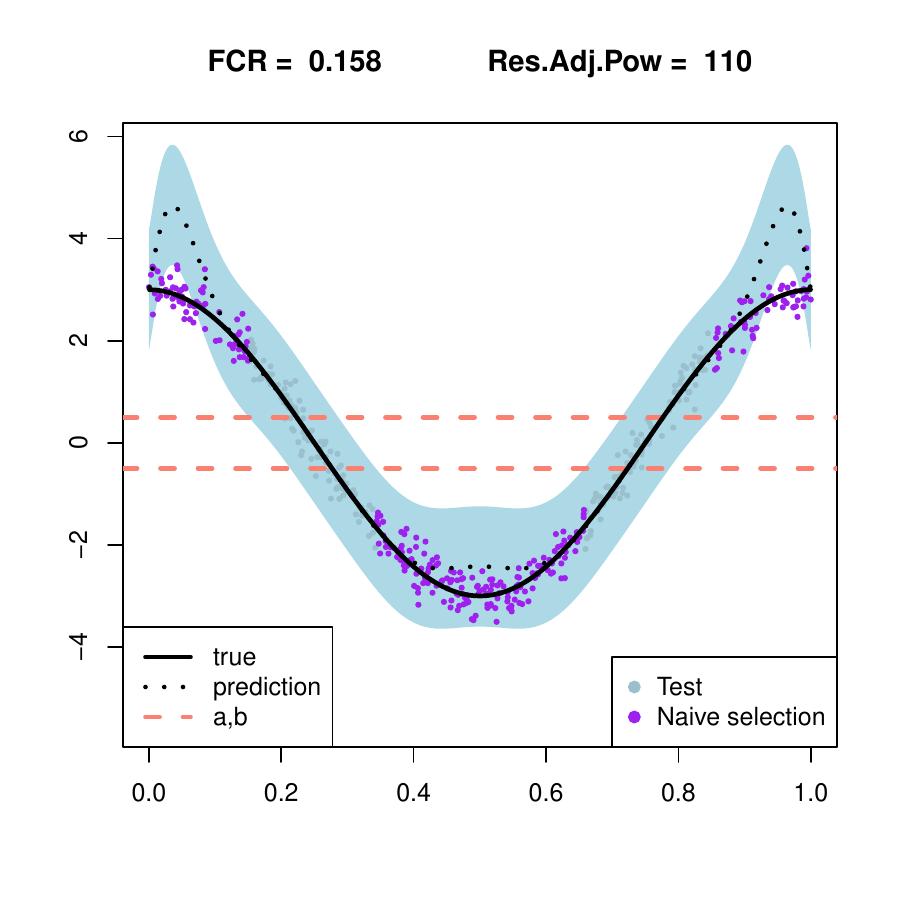}&\hspace{-9mm}
\includegraphics[width=0.38\textwidth,  height = 0.24\textheight]{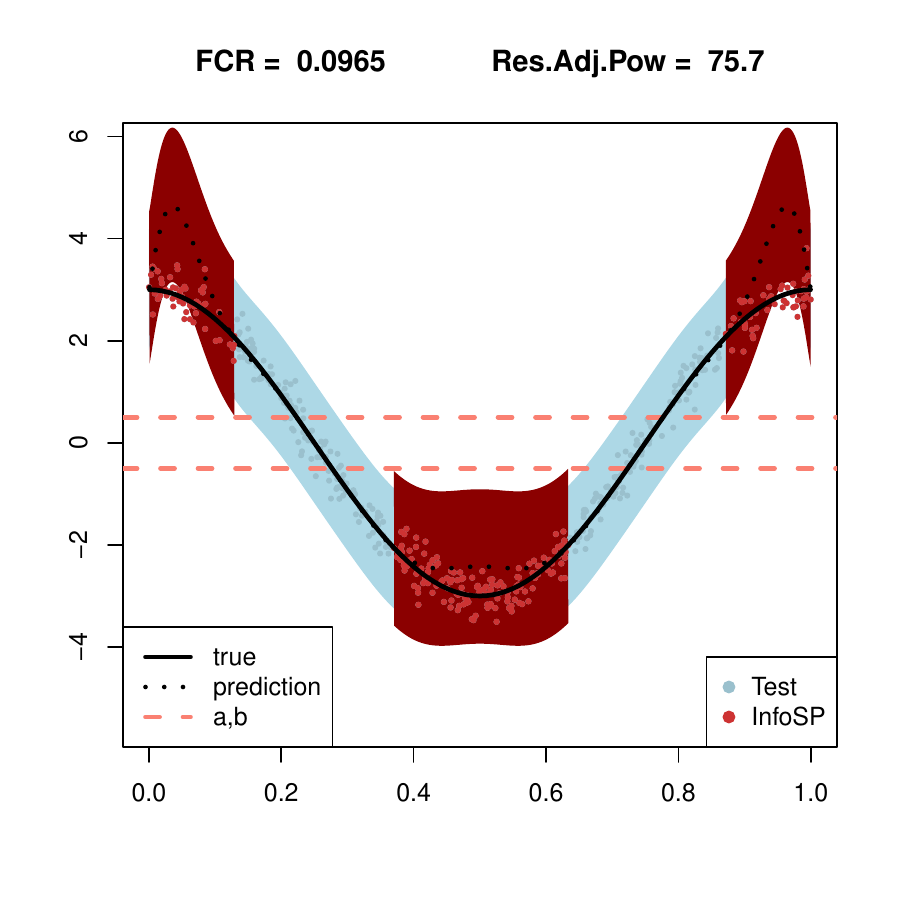}&\hspace{-9mm}
\includegraphics[width=0.38\textwidth,  height = 0.24\textheight]{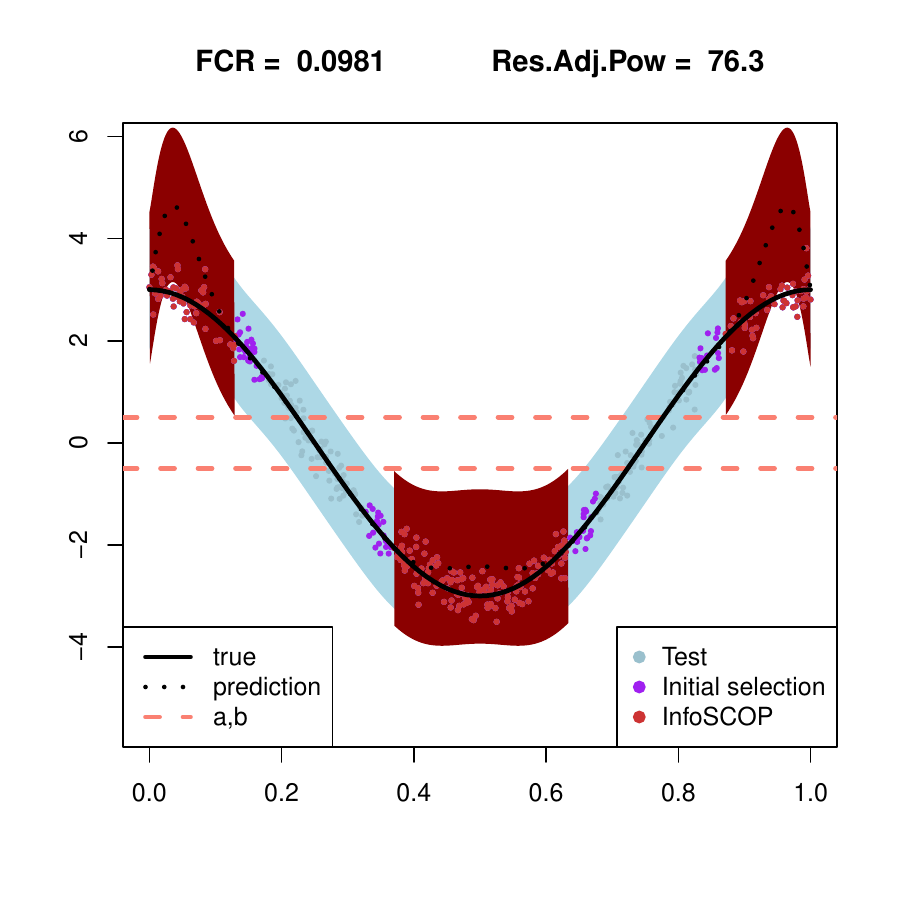}\vspace{-5mm}\\
\hspace{-9mm}\includegraphics[width=0.38\textwidth, height = 0.24\textheight]{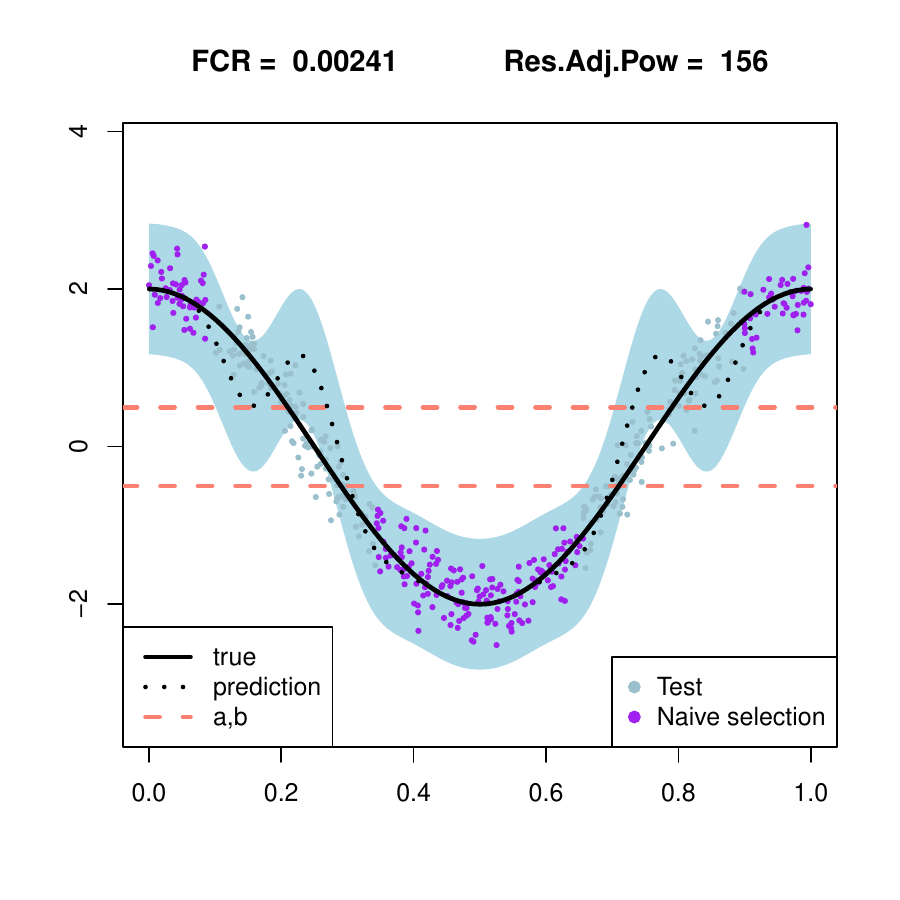}&\hspace{-9mm}
\includegraphics[width=0.38\textwidth,  height = 0.24\textheight]{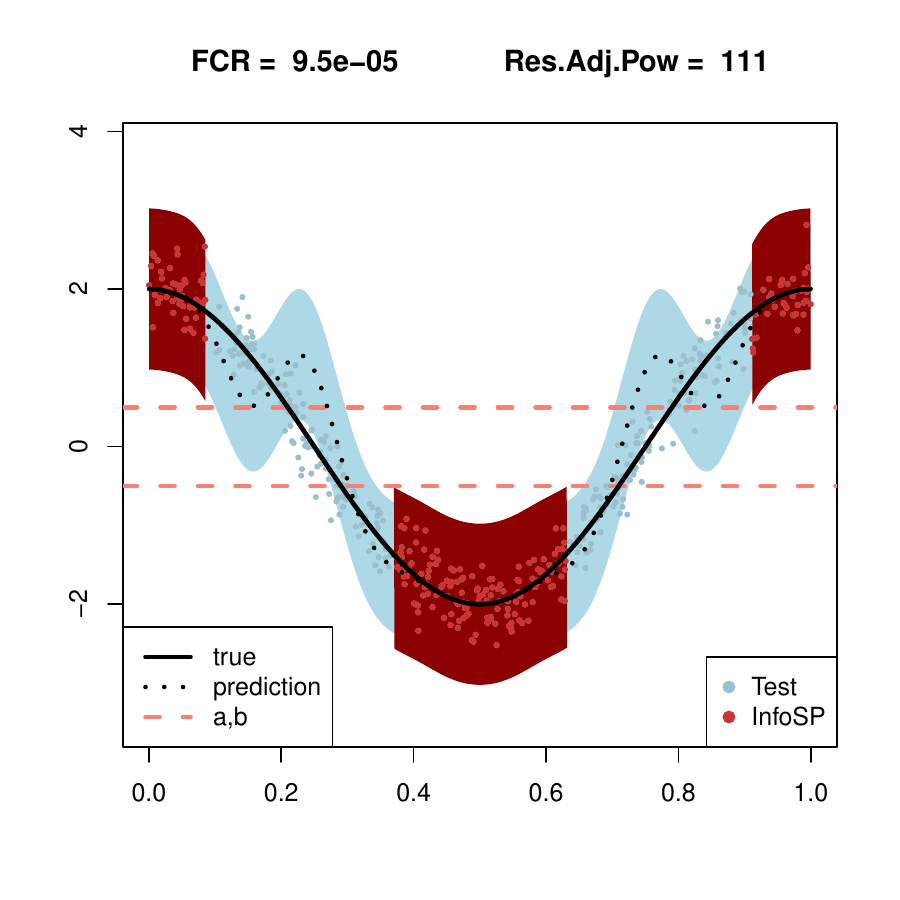}&\hspace{-9mm}
\includegraphics[width=0.38\textwidth,  height = 0.24\textheight]{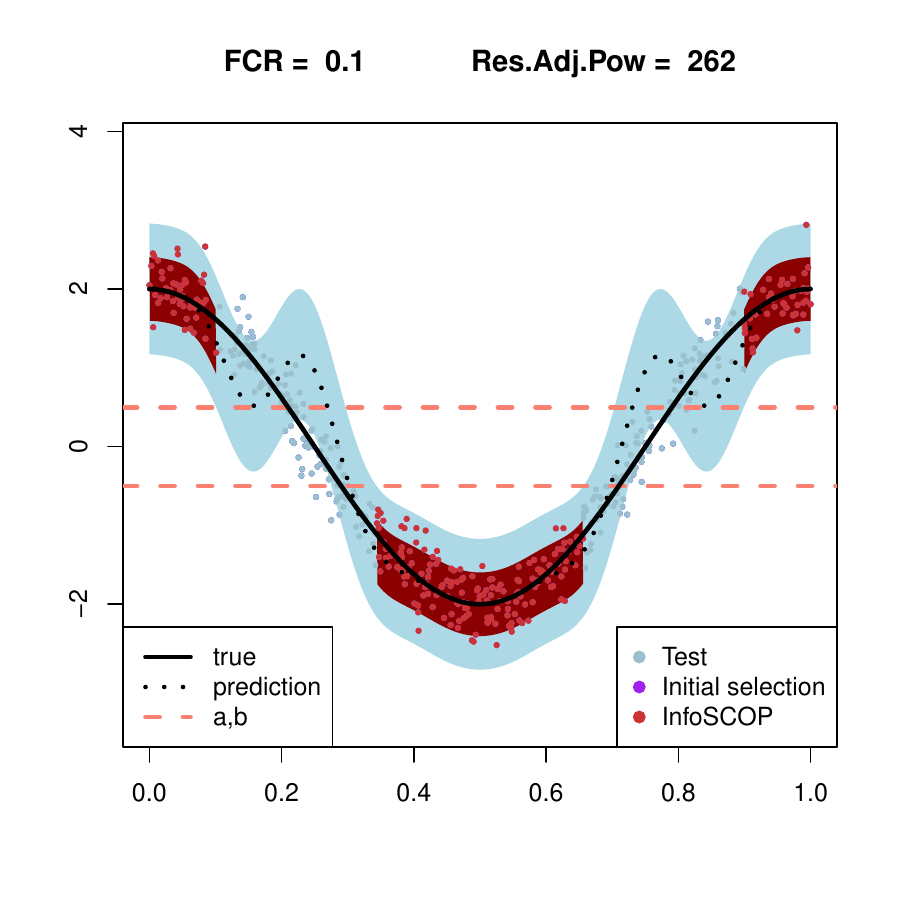}
\end{tabular}
\vspace{-5mm}
\caption{Informative prediction intervals when excluding $[a,b]$ (homoscedastic Gaussian regression model with perfect variance prediction), see text. The predictor $\mu$ (dotted line) does not approximate well the true $\mu^*(x)=\E[Y|X=x]$ (solid line) in the selection area (top row) and  out of the selection area (bottom row). The marginal  and informative prediction intervals (\texttt{InfoSP} and \texttt{InfoSCOP}) are depicted in light blue and red, respectively. While the plot corresponds to one data generation, the FCR and adjusted power computed in the title of each panel are computed with $100$ Monte-Carlo simulations. $n=1000$, $m=500$, $\alpha=0.1$.
\label{fig-regressexnonullab}} 
\end{figure}

\begin{remark}\label{rem-excludingintervals-scores}
   In Corollary~\ref{cor:abexcluded}, we consider the locally weighted score function $S_y(x)=|\mu(x)-y|/\sigma(x)$ for simplicity of exposition, but we can use any  score function satisfying Assumption~\ref{as:noties}. For instance, for the quantile-based score function $S_y(x)=\max(q_{\beta_0}(x)-y,y-q_{\beta_1}(x))$, the corresponding $q_i$ have the expression 
$        q_i=\pfull{a}{i}\ind{\mu(X_{n+i})<a} + \pfull{b}{i}\ind{\mu(X_{n+i})>b}+ \pfull{\mu(X_{n+i})}{i}\ind{a \leq \mu(X_{n+i})\leq b},    $    where $\mu(x)=(q_{\beta_0}(x)+q_{\beta_1}(x))/2$ and where the $\pfull{y}{i}$'s are computed by using this score function \add{(see the computations in  Section~\ref{sec:qicomput})}.    
    This leads to the prediction intervals 
    $
    \mathcal{C}_{n+i}= [q_{\beta_0}(X_{n+i})-S_{(n_\alpha(\mathbf{p}))}, q_{\beta_1}(X_{n+i}) + S_{(n_\alpha(\mathbf{p}))}]$ (which do not \add{intersect} $[a,b]$), by using the notation of Corollary~\ref{cor:abexcluded}. 
\end{remark}

  \subsection{Length-restricted prediction  intervals} 

   In this section, we consider the situation where the user only wants to report prediction intervals that are accurate enough, which corresponds to consider  $\mathcal{I}=\{[a,b]\subset \R\::\: 0<b-a\leq 2\lambda_0\}$ as the informative subset collection, for some size $\lambda_0>0$. 
   
   \begin{corollary}\label{cor:length}
Consider the iid model in the regression case, consider the locally weighted score function $S_y(x)=|\mu(x)-y|/\sigma(x)$ and suppose that Assumption~\ref{as:noties} holds. 
Then the following holds for \texttt{InfoSP} with informative collection $\mathcal{I}=\{[a,b]\subset \R\::\: 0<b-a\leq 2\lambda_0\}$ and full-calibrated $p$-value collection $\pfullbf$ \eqref{standardpvalue}:
\begin{itemize}
    \item[(i)] \texttt{InfoSP} selects $\mathcal{S}=\BH(\mathbf{q})$ with $q_i$ given by the formula of Example~\ref{ex:regression2}. 
    \item[(ii)] The selected prediction intervals are of length at most $2\lambda_0$ and are of the form
    $
    \mathcal{C}_{n+i}= [\mu(x)-S_{(n_\alpha(\mathbf{p}))} \sigma(x), \mu(x) + S_{(n_\alpha(\mathbf{p}))} \sigma(x) ]$, 
    where $S_{(1)}\leq \dots\leq S_{(n)}$ are the ordered calibration scores $S_{Y_j}(X_j)$, $1\leq j\leq n$ (with $S_{(n+1)}=+\infty$), and $n_\alpha(\mathbf{p})=\lceil (1- \alpha|\mathcal S(\mathbf{p})|/m)(n+1)\rceil$.
    \item[(iii)] These prediction intervals control the FCR at level $\alpha$ in the sense of \eqref{iidcontrol}.
\end{itemize}
\end{corollary}
   
A similar result holds for \texttt{InfoSCOP}. 
In Corollary~\ref{cor:length} (ii), the length of the prediction interval on the selection is always granted to be (at most) of the correct size $2\lambda_0$, even if adjusting the level is necessary to account for selection (which de facto enlarge the prediction interval). Thanks to the $\BH(\mathbf{q})$ selection the size-adjustment is automatic, while maintaining the FCR control.

\begin{proof}
\add{The expression of $q_i$ follows from straightforward computations, see Section~\ref{sec:qicomput}. This implies  that $|\mathcal{C}_{n+i}|=2\sigma(X_{n+i})S_{(\lceil (1-\alpha)(n_\alpha(\mathbf{p})+1)\rceil)}\leq 2\lambda_0$ since $q_i\leq \alpha|\mathcal S(\mathbf{p})|/m$. }
\end{proof}

Figure~\ref{fig-regressLength} displays length-restricted informative prediction intervals in particular settings. In the first row, errors are more likely to occur on the selection (due to under-estimation of the variance), while in the second row, errors are less likely to occur on the selection  (due to over-estimation of the variance). Hence, the comment is similar to the previous section: \texttt{InfoSP} and \texttt{InfoSCOP} are similar in the first situation but \texttt{InfoSCOP} improves \texttt{InfoSP} in the second.

\begin{remark}
Corollary~\ref{cor:length} easily extends to the case of conformalized quantile regression, by considering the quantile-based score function $S_y(x)=\max(q_{\beta_0}(x)-y,y-q_{\beta_1}(x))$. In that case, 
\add{the formula of the $q_i$'s are given in  Example~\ref{ex:regression2} (see Section~\ref{sec:qicomput} for a proof) and the prediction intervals are } 
    $
    \mathcal{C}_{n+i}= [q_{\beta_0}(X_{n+i})-S_{(n_\alpha(\mathbf{p}))}, q_{\beta_1}(X_{n+i}) + S_{(n_\alpha(\mathbf{p}))}]$ (of length at most $2\lambda_0$), with the notation of Corollary~\ref{cor:length}.
\end{remark}
   
\begin{figure}[h!]
\begin{tabular}{ccc}
Classical conformal & \texttt{InfoSP}  & \texttt{InfoSCOP}
\\\hspace{-9mm}\includegraphics[width=0.38\textwidth, height = 0.24\textheight]{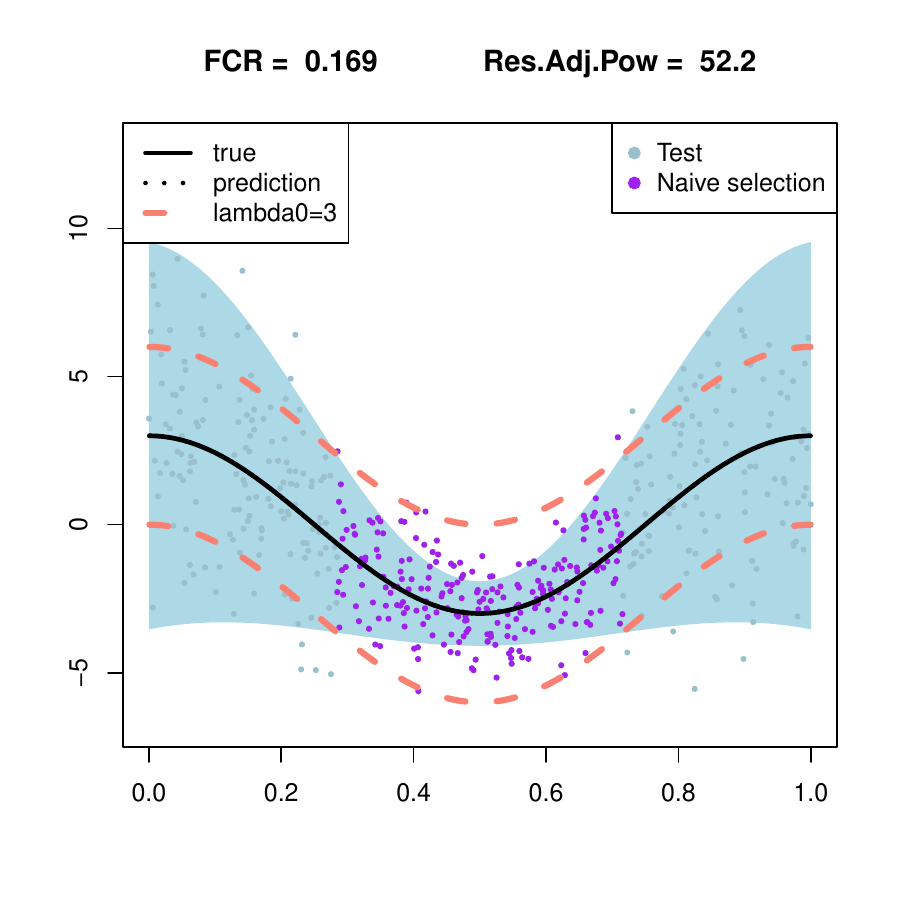}&\hspace{-9mm}
\includegraphics[width=0.38\textwidth,  height = 0.24\textheight]{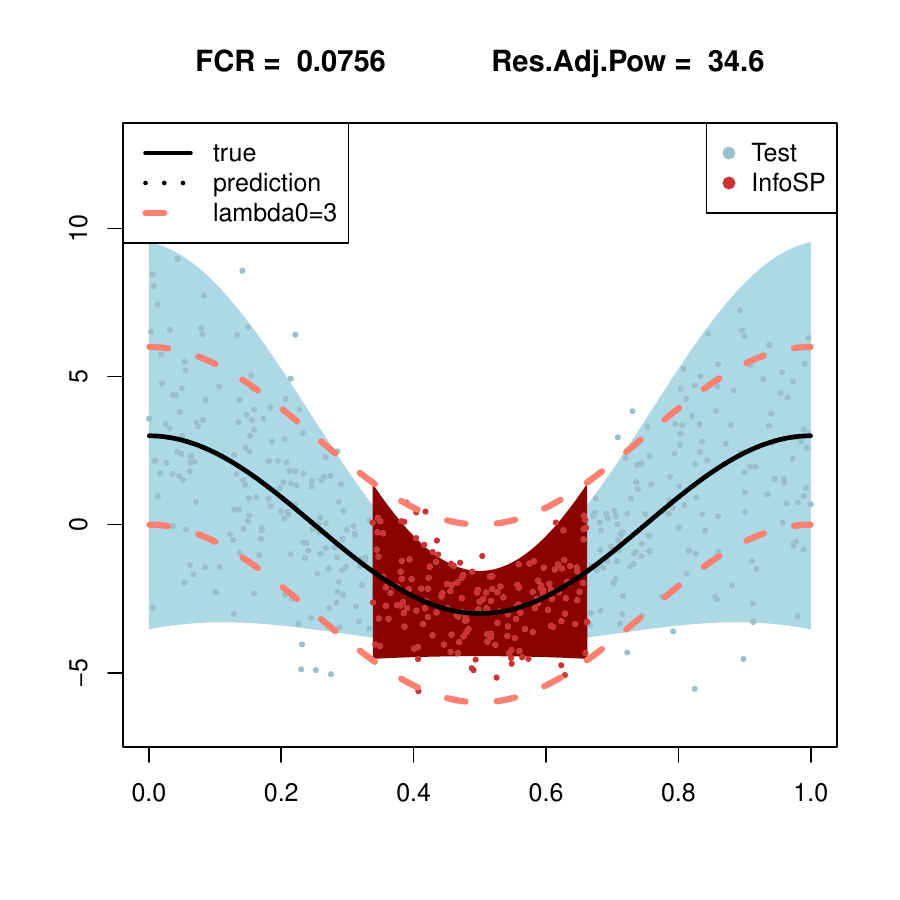}&\hspace{-9mm}
\includegraphics[width=0.38\textwidth,  height = 0.24\textheight]{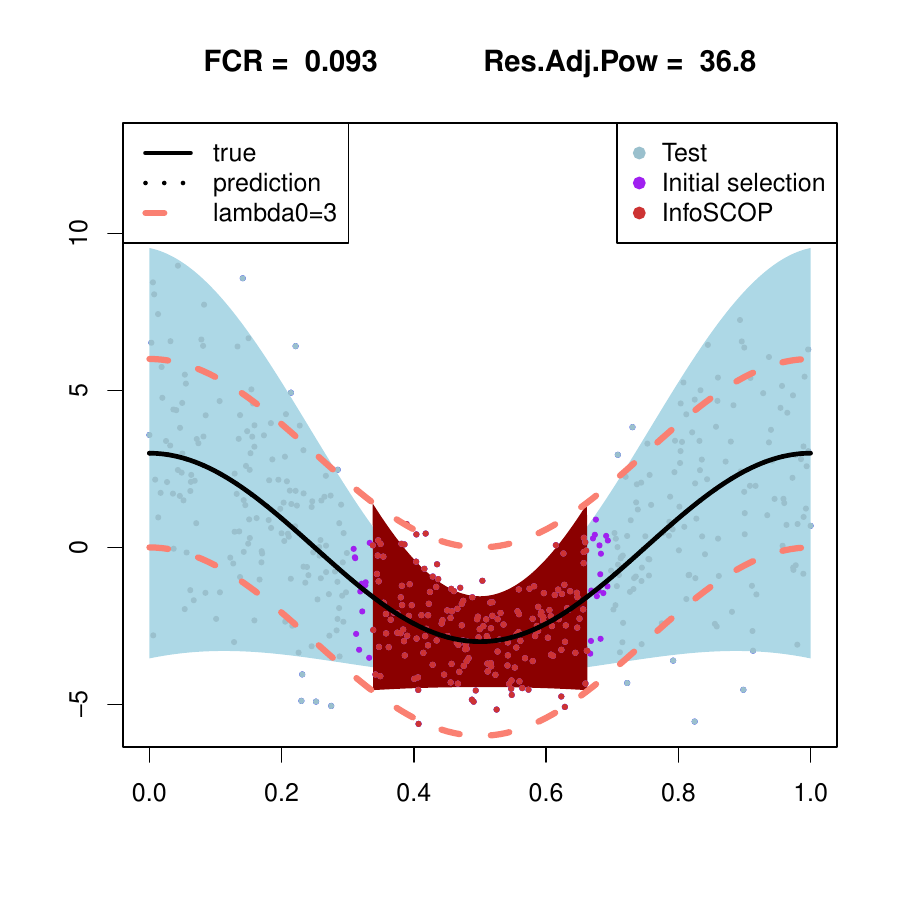}\vspace{-5mm}\\
\hspace{-9mm}\includegraphics[width=0.38\textwidth, height = 0.24\textheight]{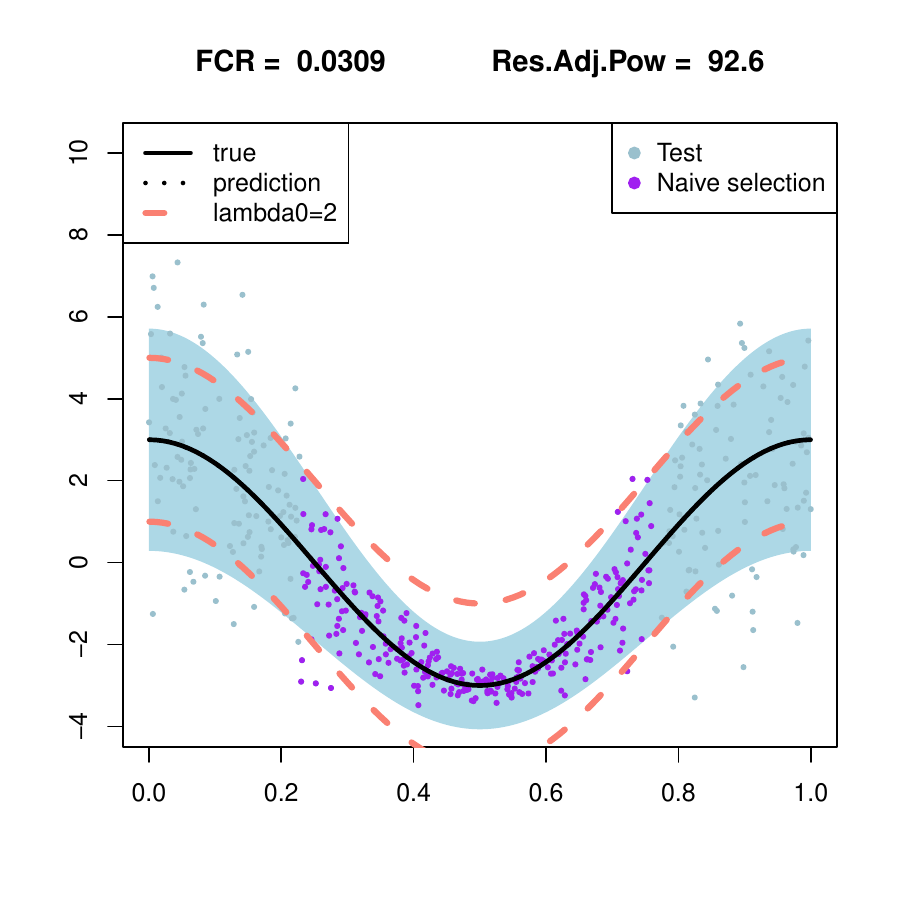}&\hspace{-9mm}
\includegraphics[width=0.38\textwidth,  height = 0.24\textheight]{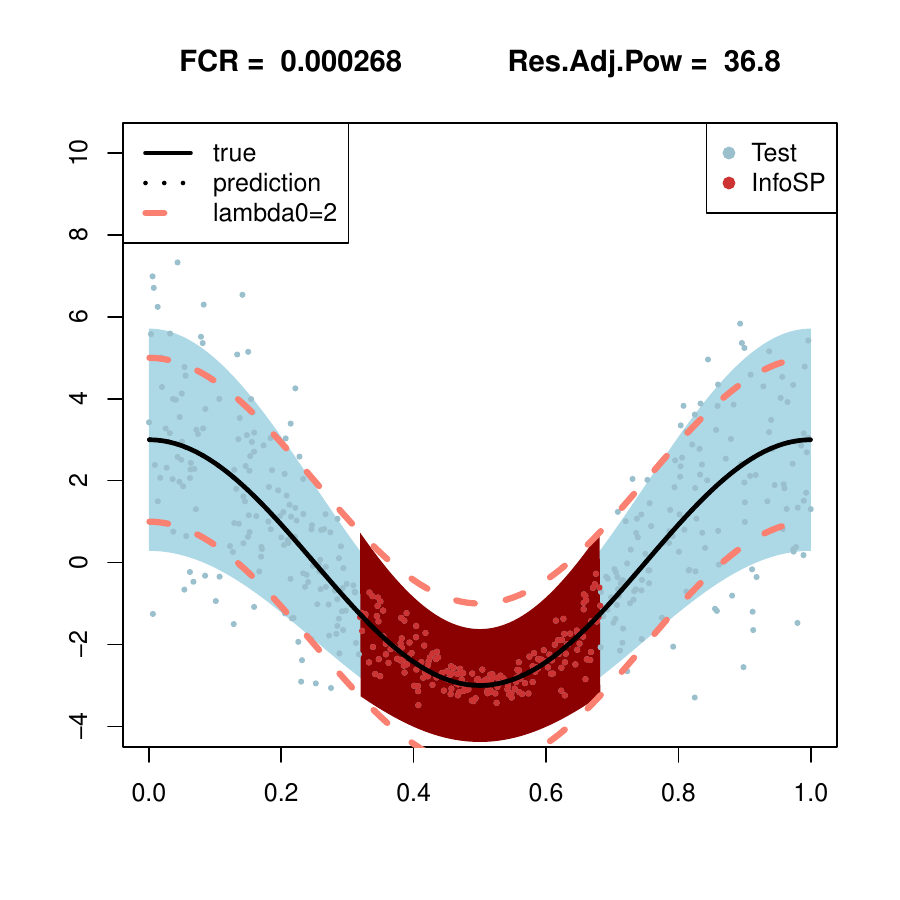}&\hspace{-9mm}
\includegraphics[width=0.38\textwidth,  height = 0.24\textheight]{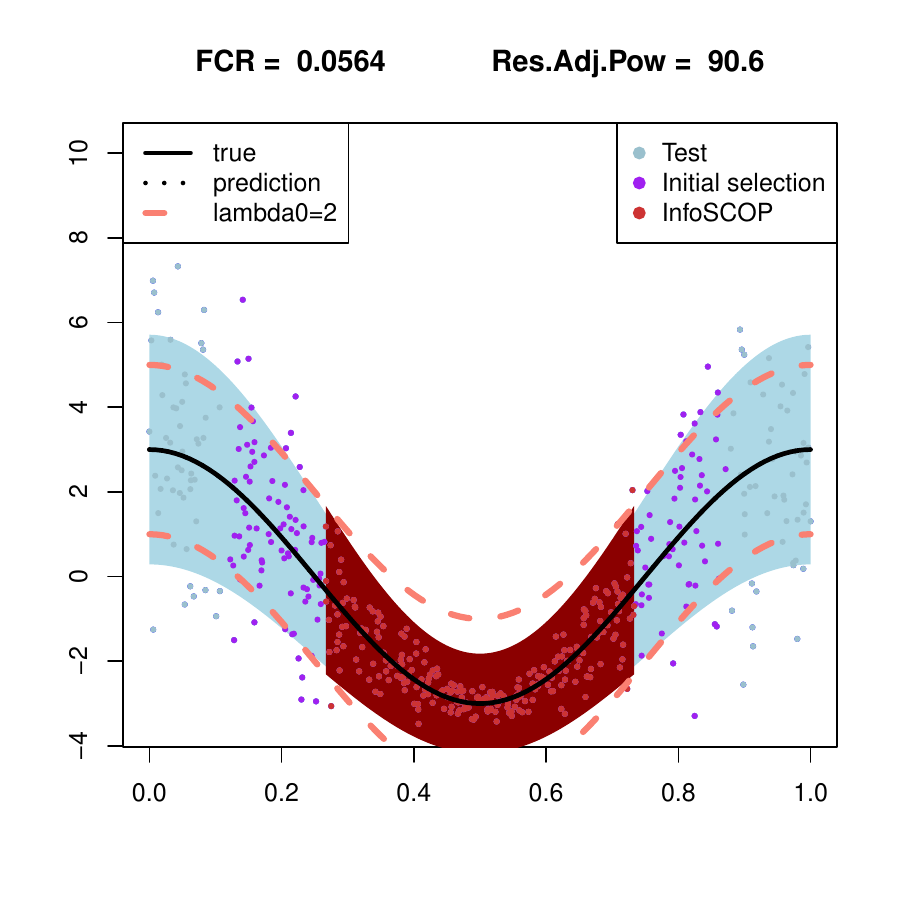}
\end{tabular}
\vspace{-5mm}
\caption{Informative prediction intervals when length-restricted (heteroscedastic Gaussian regression model with perfect mean prediction), see text. The predictor $\sigma$ under-estimates (top row) and over-estimates (bottom row) the true $\sigma^*(x)=\mathbb{V}^{1/2}[Y|X=x]$ in the selection area.  The marginal and  informative prediction intervals (\texttt{InfoSP} and \texttt{InfoSCOP}) are depicted in light-blue and red, respectively. While the plot corresponds to one data generation, the FCR and adjusted power computed in the title of each panel are computed with $100$ Monte-Carlo simulations. $n=1000$, $m=500$, $\alpha=0.1$.
\label{fig-regressLength}} 
\end{figure}

\section{Application to classification}\label{sec:classif}

We consider the classification case $\ZZ=\range{K}$,  for both  the iid model and the class-conditional model. 
Importantly, in classification, any selective prediction set procedure $\mathcal{R}=(\mathcal{C}_{n+i})_{i\in \mathcal{S}}$ is post-processed by setting, for $i\in \mathcal{S}$, $\mathcal{C}_{n+i}=\arg\min_{k\in \range{K}}\{S_k(X_{n+i})\}$ 
whenever $\mathcal{C}_{n+i}=\emptyset$ (that is, if empty take the smallest non-conformity score).
Clearly, this operation can only decrease the FCP while it can only increase the adjusted power, so it should always be preferred in the classification case.
In this paper, \texttt{InfoSP} and \texttt{InfoSCOP} always refer to the post-processed procedures in the classification case.

\subsection{Choosing the appropriate $p$-value collection in classification} 

While the family of full-calibrated $p$-values are only valid for the iid model, the family of class-calibrated $p$-values are valid  both in the iid and in the class-conditional model.

Thus, for  the iid model,  we can in principle use either  full-calibrated $p$-values or class-calibrated $p$-values.  In the applications we consider next,  it appears that using  full-calibrated $p$-values in \texttt{InfoSP} is best. We support this claim by theory  for non-trivial classification in \S~\ref{subsec:min-iid}, and by numerical experiments in \S~\ref{subsec-simul-bivariatenormal}, \S~\ref{subsec:min-iid}, and \S~\ref{sm_simulations}. We note that for the initial selection step in \texttt{InfoSCOP}, class-calibrated $p$-values can be useful, as demonstrated in \S~\ref{subsec-simul-bivariatenormal}.

For the  class-conditional model, there can be a label shift from the calibration to the test sample. So the full-calibrated $p$-values are not valid, and  class-calibrated $p$-values must be used.  In  \S~\ref{SM-sec:adaptproc} we  consider more generally  weighted class-calibrated $p$-values, where the weights are functions of estimators of the proportion of labels in each class in the test sample.

\subsection{An illustrative example: prediction sets excluding a null class}\label{subsec-simul-bivariatenormal}

Suppose the analyst is interested in reporting prediction sets that exclude a null class, say class $y_0=1$ (see first items of Examples~\ref{ex:classif}~and~\ref{ex:classif2}). We consider the following novel procedures, in addition to the naive procedure using the classic conformal procedure, that reports $\mathcal C^{\alpha}_{n+i}$ only if $\mathcal C^{\alpha}_{n+i}$ does not \add{intersect} the null class: 
    first, \texttt{InfoSP}  on full-calibrated $p$-values, denoted by $\RinfoSP(\pfullbf)$. 
  Second, \texttt{InfoSCOP} on full-calibrated $p$-values, denoted  by  $\RinfoSCOP(\pfullbf)$, with initial selection step $\mathcal{S}^{(0)}\subset \range{r+1,n+m}$ being the BH procedure applied to the class-calibrated adaptive $p$-value family $(\pcond{1}{i,{\tiny \mbox{adapt}}}, i\in \range{r+1,n+m})$, given by
$\pcond{1}{i,{\tiny \mbox{adapt}}}=\hat{\pi}_1 \pcond{1}{i},\: i \in\range{r+1,n+m},$ 
  using $\{(X_j,Y_j),j\in \range{r}\}$ and $\{(X_j,Y_j),j\in \range{r+1,n+m}\}$ as calibration and test samples, respectively, and with calibrated-based estimator $\hat \pi_1 = (r+1)^{-1}\big(\sum_{i=1}^r\ind{Y_i=1}+1\big)$. 
    Third, \texttt{InfoSP}  on class-calibrated $p$-values, denoted by $\RinfoSP(\pcondbf)$.
Note that  \texttt{InfoSCOP} above uses class-calibrated $p$-values for initial selection (because these are better to detect examples from the non-null class), and  full-calibrated $p$-values on the selected examples from $\range{r+1,n+m}$ in the second step (because these are better $p$-values for building prediction sets in the iid model).

We consider a Gaussian mixture model with $K=3$ components, where each component is bivariate normal. The centers for the three components are (0,0), (SNR,0), and (SNR,SNR). So the overlap between components is larger as SNR decreases. 
We consider the case of balanced classes in the calibration sample, as well as the case of unbalanced classes where  the null class is much larger than the others. Specifically, the balanced case has class probabilities 0.33, 0.33, and 0.34, and the unbalanced case  (depicted in Figure \ref{fig-exchang-snapshot} for an SNR value of 3) has  class probabilities 0.15, 0.10, and 0.75 (the null class). 
In the balanced case,  we consider the iid setting where the test sample has the same class probabilities as the calibration sample, as well as class-conditional setting where the test sample has class probabilities  0.2,0.2, and 0.6 (the null class), so the label shift is large. 
We estimate the probability of being in each class with a support vector classifier implemented by the {\it e1071} R package \cite{svmpackage}. 

Figure \ref{fig-nullselection-classification} shows the FCR and resolution-adjusted power of all  procedures considered. As expected,  
the classic conformal procedure does not control the FCR for any data generating model (it uses the full-calibrated and class-calibrated $p$-values in the iid and class-conditional setting, respectively).  All other procedures control the FCR.

For the iid model (Figure \ref{fig-nullselection-classification} left and middle columns), \texttt{InfoSCOP}  on full-calibrated $p$-values has better power than the  alternatives for prediction sets excluding a null class. Its advantage over \texttt{InfoSP} is primarily due to the fact that after pre-processing, almost all test examples are non-null, as illustrated in Figure \ref{fig-exchang-snapshot} in the SM for a single data generation. The differences between the procedures are qualitatively the same, but even greater, when the overlap between components is larger, see Figure \ref{SM-fig-exchang-hard} in the SM. This is because when the overlap with the null class is large, after pre-processing only examples with better scores are considered, as illustrated in Figure \ref{SM-fig-exchang-snapshot} in the SM.  For completeness, we also provide \texttt{InfoSP} on class-calibrated $p$-values in the iid setting, to demonstrate numerically the potentially large power advantage from using the full-calibrated $p$-values over the class-calibrated $p$-values.  

For the class-conditional model (Figure \ref{fig-nullselection-classification} right column),  \texttt{InfoSP} has lower power than classic conformal, but the power is reasonable. In Figure \ref{fig-nullselection-classification-adaptive} in the SM we compare it to two procedures that weigh the classes according to their estimated relative frequencies.

\begin{figure}[h!]
\begin{center}
 {\includegraphics[width=0.36\textwidth, height = 0.27\textheight, page=2]{legends.pdf}}

\vspace{-3.5cm}
  \begin{tabular}{ccc}
\hspace{-5mm}    \includegraphics[width=0.34\textwidth, height = 0.2\textheight, page=3]{subscriptnonnullSelection1.pdf}   &
\hspace{-5mm}
      \includegraphics[width=0.34\textwidth, height = 0.2\textheight, page=3]{subscriptnonnullSelection2.pdf}   & 
  \hspace{-5mm}  \includegraphics[width=0.34\textwidth, height = 0.2\textheight, page=4]{MainClassCondNONNULL1.pdf}\vspace{-5mm}\\
 \hspace{-5mm}   \includegraphics[width=0.34\textwidth,  height = 0.2\textheight,page=1]{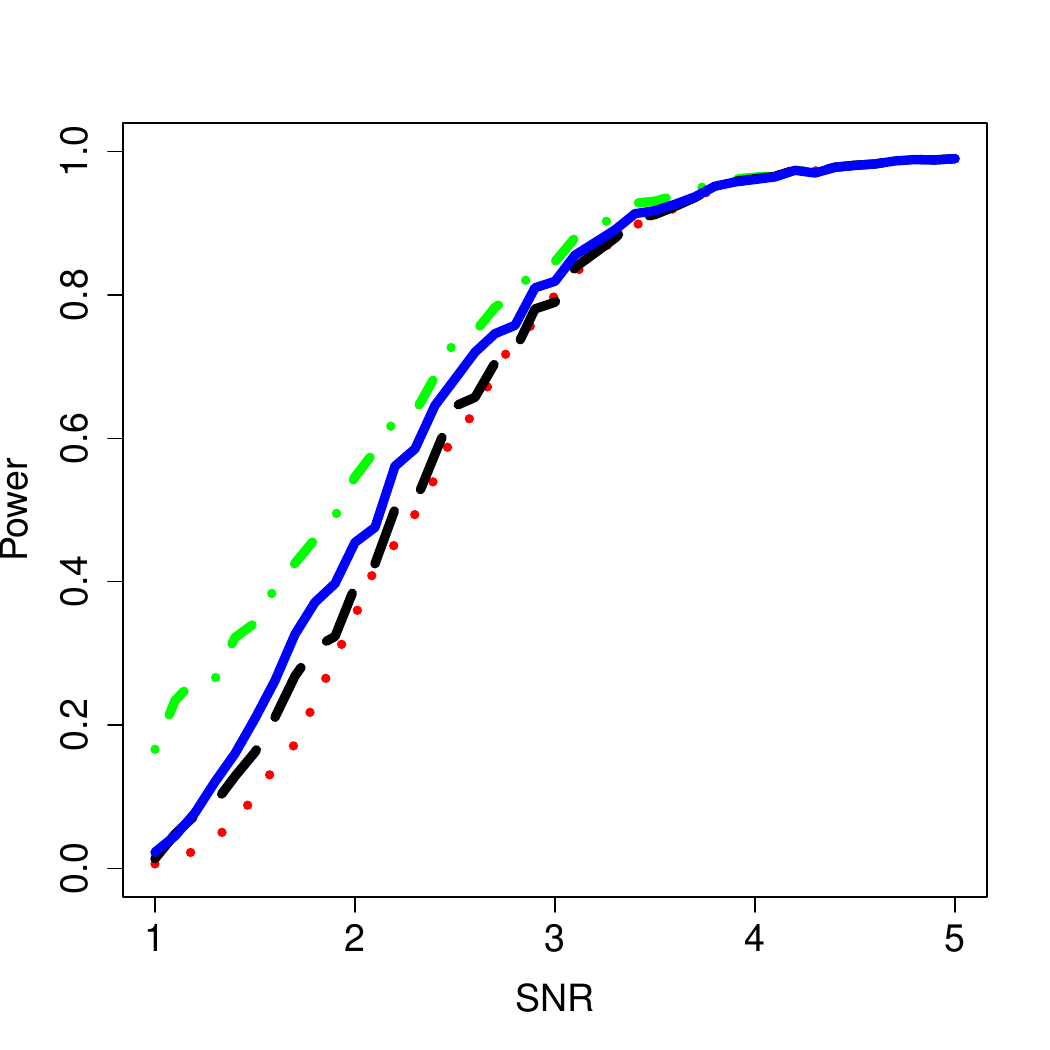} &
  \hspace{-5mm}  \includegraphics[width=0.34\textwidth,  height = 0.2\textheight,page=1]{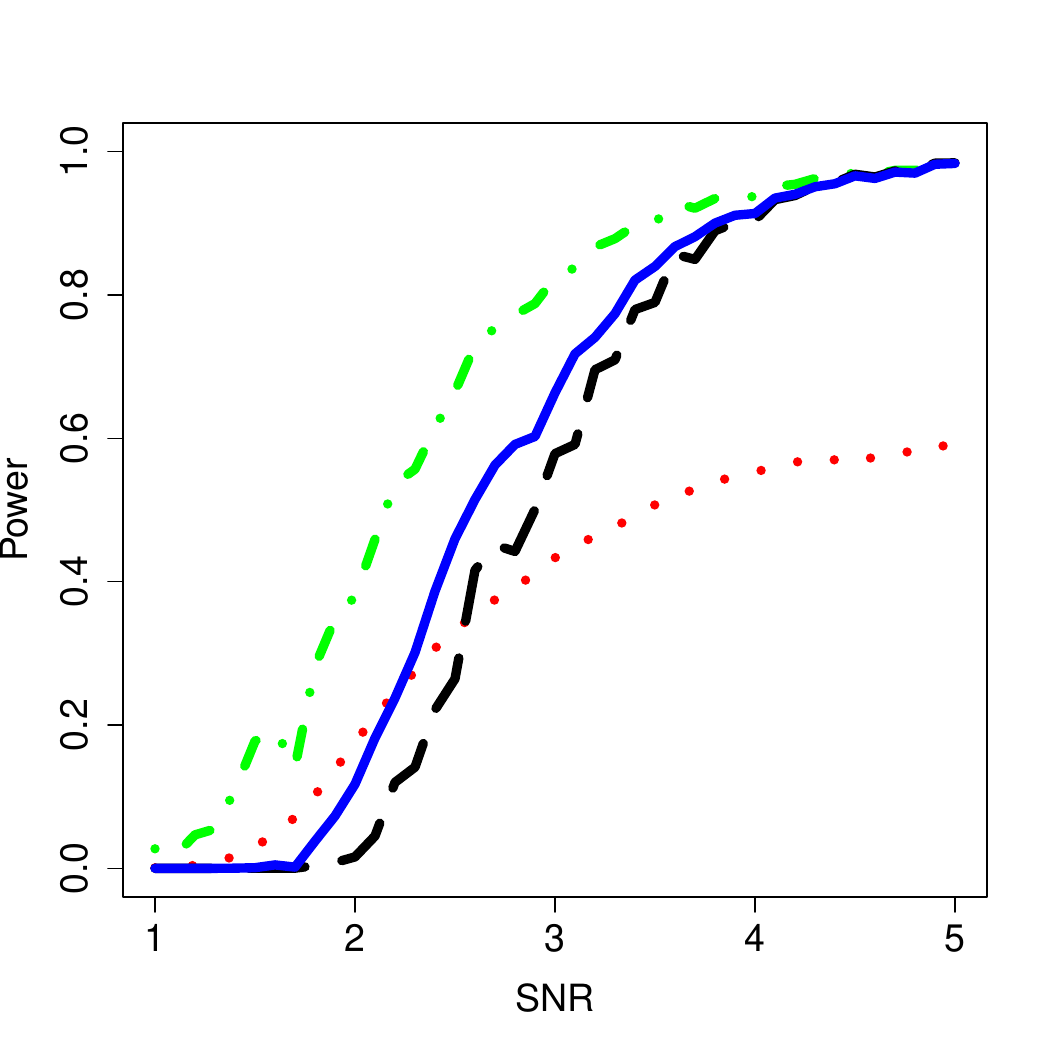}\vspace{-5mm}
&
\hspace{-5mm}\includegraphics[width=0.34\textwidth,  height = 0.2\textheight,page=3]{MainClassCondNONNULL1.pdf}\vspace{-5mm}\\
  \end{tabular}   
\end{center}

\caption{\label{fig-nullselection-classification} 
Selecting prediction sets excluding  a null class in a classification setting.
FCR (top row), and resolution-adjusted power  (bottom row) versus SNR. 
The iid setting in columns 1 and 2, with  balanced classes and unbalanced classes, respectively. The class-conditional setting in column 3, with  a  large label shift: 
 the class probabilities are equal in the calibration sample and 0.2,0.2, and 0.6 (the null class) in the test sample. The number of data generations was 2000, 1000 data points were used for training,  and $n =m = 500$. See details of the data generation in \S~\ref{subsec-simul-bivariatenormal}. } 
\end{figure}

\subsection{Selecting non-trivial prediction sets}\label{subsec:min-iid}

Suppose the analyst is interested in reporting prediction sets that are not equal to $\range{K}$ (see first item of Examples~\ref{ex:classif}~and~\ref{ex:classif2}).
In that case, we argue that \texttt{InfoSP} has an FCR close to $\alpha$ in the iid model.
Intuitively, this comes from the selection rule $\mathcal S=\BH(\pfullbf)$ which is such that $\ind{Y_{n+i}\notin \mathcal{C}^{\alpha|\mathcal S|/m}_{n+i}(\pfullbf), i\in \mathcal S} = \ind{Y_{n+i}\notin \mathcal{C}^{\alpha|\mathcal S|/m}_{n+i}(\pfullbf)}.$ It means that it is not possible to fail to cover at the adjusted level without being selected (because otherwise the prediction set is trivial).
This is not the case for other selection rules, e.g.,  excluding a null class, where it is possible that the true class label is not covered at the adjusted level even if the example is not selected (thus implying that the adjusted level is conservative, since for FCR control we guard against non-coverage at the adjusted level for all examples). 
We formalize fully the argument for $K=2$ in the following result.

\begin{proposition}\label{prop:InfoSPClassif}
In the iid classification model with $K=2$, consider the non-trivial informative subset collection $\mathcal{I}=\{C\subset \range{K}\::\: |C|\leq 1\}$ and assume that the score functions satisfy Assumption~\ref{as:noties} with $\sum_{k\in \range{K}}S_k(x)=1$ and $S_k(x)\geq 0$. 
Let $p_0$ be the probability that $S_{Y_i}(X_i)$ is the maximum score $\max \{ S_{1}(X_i),S_2(X_i)\}$. 
Then if $(n+1)\alpha /m$ is an integer,  we have 
$\FCR(\RinfoSP(\pfullbf),P_{X,Y})=\alpha  (1-(1-p_0)^{n+1})$.
  \end{proposition}

The proof is given in \S~\ref{proof:prop:InfoSPClassif}, which also shows that $\RinfoSP(\pfullbf)$ coincides with the procedure of \cite{zhao2023controlling} for $K=2$. 
In typical applications (where classes are not very well separated) the value of $(1-(1-p_0)^{n+1})$ is close to one, which means that the FCR of our procedure should be close to $\alpha$, at least for $K=2$.

To complement our theoretical result, we provide numerical results in Figure \ref{fig-nontrivial-classification}.
\texttt{InfoSP} has the best power, with  \texttt{InfoSCOP} a close second, on full-calibrated $p$-values.  \texttt{InfoSP} on class-calibrated $p$-values has much lower power in the unbalanced setting. As expected,  
the classic conformal procedure does not control the FCR and the level for \texttt{InfoSP} is about $0.05$ for a range of SNR values.  All other procedures control the FCR.

\section{Informative prediction sets for 3 classes of animals} \label{sec:appli}

In this section, we illustrate the performance of our methods on real data. 
We use the image dataset CIFAR-10 (\url{https://www.cs.toronto.edu/~kriz/cifar.html}), which consists of 60000 32x32 colour images in 10 classes, with 6000 images per class. We restrict the analysis to 3 classes: birds, cats and dogs. \add{The code used for these experiments can be found at \url{https://github.com/arianemarandon/infoconf}.}

We consider 4 scenarios: in the iid setting,  non-trivial classification (scenario a) and non-null classification (scenario b); in the class-conditional setting, non-trivial classification (scenario c) and non-null classification (scenario d). 
 The null class is the bird class in scenarios b and d. 
By default, the classes are in equal proportions.  We introduce a label shift between the calibration sample and test sample for the class-conditional settings,  by modifying the classes proportions: of the calibration sample to 20\% for the bird class versus 40\% for the cat/dog class in scenario c;   of the test sample to 50\% for the bird class versus 25\% for the cat/dog class  in scenario d. In each experiment, the test size is $m=1000$ and the calibration size is $n=5000$. 

In the iid settings (scenarios a and b), we evaluate the procedures \texttt{InfoSP} and \texttt{InfoSCOP} with full-calibrated $p$-values and in the class-conditional settings (scenarios c) and d)) we evaluate \texttt{InfoSP} with class-calibrated $p$-values, each being also compared with classical conformal prediction (denoted by \texttt{CC}). For all procedures the non-conformity score is $S_y(x) = 1-\pi_y(x)$ where $\pi_y(x)$ is an estimator of the probability that the class of $x$ is $y$ and is learned using a convolutional neural network (CNN) with 2 convolutional layers, one pooling layer, and 3 fully-connected layers, trained for 20 epochs with a learning rate of 0.01 on a sample size of 5000. 
To assess the power of the methods, in addition to the resolution-adjusted power, we plot the selection rate (SR) of the procedures, defined as the average proportion of informative prediction sets returned, and the average size of the informative prediction sets. 

In each setting, the FDR and each power metric for the methods are evaluated by using 100 runs and the results are reported in Figure \ref{fig:app} scenarios a and b and in Figure \ref{fig:app2} scenarios c and d for $\alpha =  0.1$. 
The conclusions are qualitatively similar to the experiments of \S~\ref{sec:classif}: in all settings, classical conformal prediction yields an FCR that severely exceeds the marginal level with an inflation of about 50\%. By contrast, our procedures \texttt{InfoSP} and \texttt{InfoSCOP} control the FCR at the target nominal level, both without and with label shift. 
In terms of power, concerning the iid settings, pre-processing is not useful for non-trivial classification as expected and the performances of \texttt{InfoSP} and \texttt{InfoSCOP} are similar with an FCR close to the nominal level for both in that case. For the non-null classification task, \texttt{InfoSP} is conservative while \texttt{InfoSCOP} is more powerful and displays an FCR close to $\alpha$. 
When there is label shift, in the case of non-trivial classification \texttt{InfoSP} displays an FCR close to $\alpha$. In the non-null classification case, however, the label shift increases the difficulty of the task in the sense that for a fixed selection, under-covering the null class in the test sample results in more false informative sets. Hence, the power is low in that case. Finally, in all scenarios,  our procedures \texttt{InfoSP} and \texttt{InfoSCOP} output a lower number of informative sets compared to \texttt{CC}, but this is necessary in order to control the FCR. The average size of the prediction sets that are informative is comparable.

\begin{figure}[h!]
    a) Non-trivial classification
    \begin{center}
    \vspace{-2mm}
    \includegraphics[width=0.9\linewidth]{"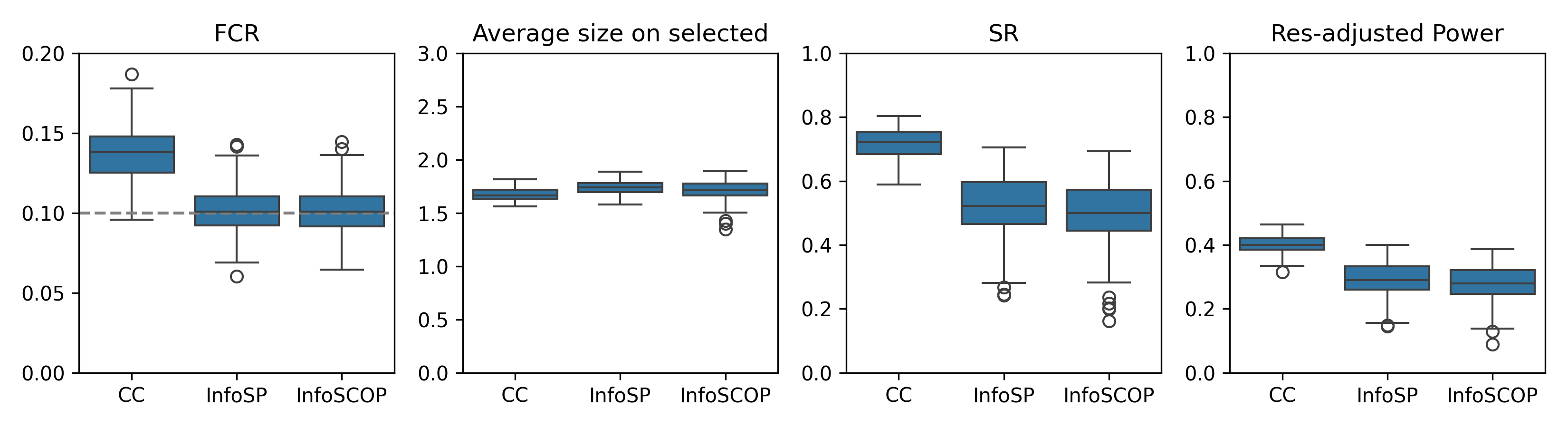"}\\
    \vspace{-2mm}
    \end{center}
    b) Non-null classification
    \begin{center}
    \vspace{-2mm}
    \includegraphics[width=0.9\linewidth]{"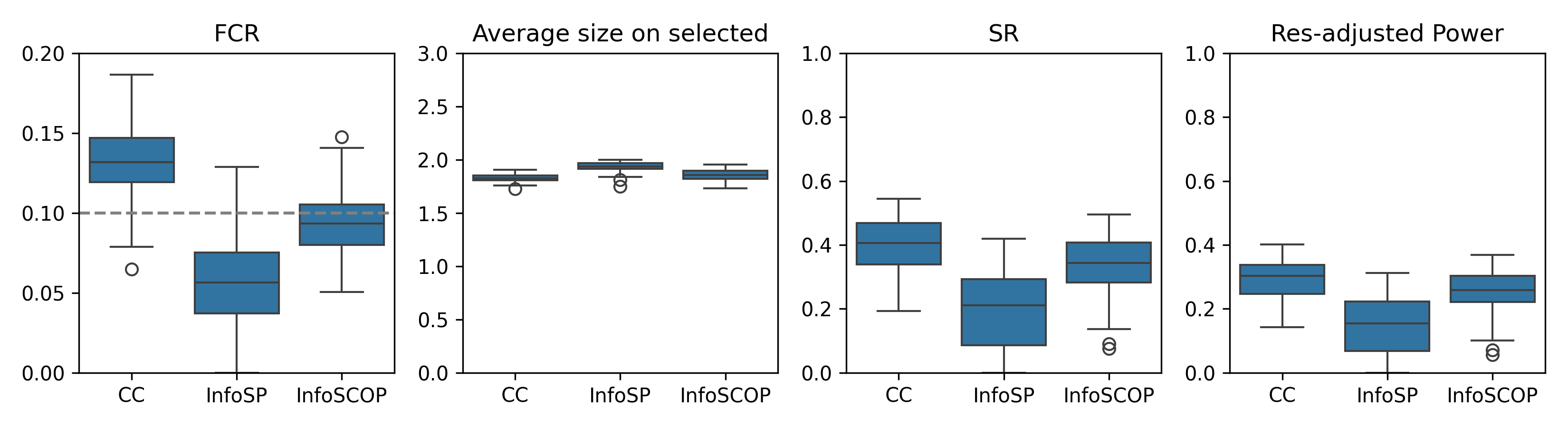"}\\
    \end{center}
    \vspace{-2mm}
    \caption{FCR, average size of the selected, SR, and resolution-adjusted power for the methods, for $\alpha=0.1$. }
    \label{fig:app}
\end{figure}

\section{Conclusion and discussion}

In this paper, we have introduced new methods for providing conformal prediction sets after selection with controlled FCR, that impose a user-specific constraint on the produced prediction sets, corresponding to a collection of so-called informative subsets $\mathcal{I}$. In contrast with previous literature in the field, the selection and prediction stages are intertwined, which results in a BH-type selection procedure on adjusted $p$-values (the $q_i$'s) that can further be explicitly derived in specific settings and that by definition produces prediction sets satisfying the desired constraint of belonging to $\mathcal{I}$. 

Our methods are very general, and they are relevant to applications in classification and regression. We showed examples  in \S~\ref{sec:appliRegression}-\S~\ref{sec:appli} for informative subsets of interest. We used common scores for the examples, but many other scores from the literature can be used with our suggested procedures \texttt{InfoSP}  and \texttt{InfoSCOP} as long as Assumption~\ref{as:noties} is satisfied. 
In addition, transfer learning scores can also be handled by our theory (see Assumption~\ref{as:AllAdaptive} in SM). This is known to greatly improve the conformal inference when there is a domain shift between the learning sample and the calibration+test samples \citep{courty2017joint,gazin2023transductive} and developing specific application cases for the latter is of interest for future investigations.

For the iid model,  \texttt{InfoSCOP} improves over \texttt{InfoSP} in all considered examples, except when selecting non-trivial prediction subsets in classification (Example~\ref{ex:classif} item 2), for which we establish that \texttt{InfoSP} almost exhausts the FCR level (see Proposition~\ref{prop:InfoSPClassif}for $K=2$). The procedure \texttt{InfoSCOP} splits the calibration sample in order to apply an efficient initial selection step on part of the calibration sample and the test sample.
There are many ways to perform the initial selection. The choice  is important  because it defines the pre-processed $p$-values \eqref{equprepropvalues} that can be seen as $p$-values ``conditionally on being selected''. Different ways of choosing $S^{(0)}$   have been investigated: trying to rule out all the examples in null class (\S~\ref{subsec-simul-bivariatenormal}) or trying to mimic the $\BH(\mathbf{q})$ selection that will be applied at the second stage to reduce the selection effect (\S~\ref{sec:appliRegression}). Finding an optimal way of calibrating $S^{(0)}$ is an interesting avenue for future research.   

For non-trivial prediction subsets in classification, \texttt{InfoSP} is optimal when $K=2$ for oracle scores, i.e.,  $S_k(X_j) = \mP(Y_j \neq k\mid X_j), k\in \range{K}, j\in \range{n+m}$. Specifically,   \cite{zhao2023controlling} showed that their classification procedure, which coincides with \texttt{InfoSP} for $K=2$, is 
 optimal for controlling the expected number of non-covering prediction sets divided by the expected number of selected examples, denoted by mFSR in their paper. An open question is whether \texttt{InfoSP} for $K>2$ is optimal when the scores are oracle scores for the resolution adjusted power objective or a variant thereof.  More generally, developing an optimality theory for selective informative prediction sets (for non-trivial prediction sets as well as for other notions informativeness) is of great interest. 
 
We provided a class-conditional variant of  \texttt{InfoSP}, with  class-conditional guarantees. We proved that our strategy can be followed in the case where the classes of the calibration and test samples are arbitrary fixed, even when the class proportions in the calibration are very different than in the test.  The main point is that working with the class-calibrated $p$-value collection allows to maintain the FCR control in this strong sense. In \S~\ref{SM-sec:adaptproc} we suggest additionally weighted procedures, that incorporate the estimated class proportions. These procedures are not necessarily more powerful than  \texttt{InfoSP}, and further research is needed in order to make recommendations about when to use the weighted procedures, and about weight adaptation to the specifics of the data.

\section*{Acknowledgments} 
\add{The authors would like to thank two anonymous referees and an associate editor for helpful comments.}
The authors acknowledge  grants ANR-21-CE23-0035 (ASCAI) and ANR-23-CE40-0018-01 (BACKUP) of the French National Research Agency ANR,  the Emergence project MARS of Sorbonne Universit\'e, and Israeli Science Foundation grant no. 2180/20.

\bibliographystyle{apalike} 
\bibliography{biblio}

\begin{thebibliography}{}

\bibitem[Bao et~al., 2024]{bao2024selective}
Bao, Y., Huo, Y., Ren, H., and Zou, C. (2024).
\newblock Selective conformal inference with false coverage-statement rate
  control.
\newblock {\em Biometrika}, page asae010.

\bibitem[Bates et~al., 2023]{bates2023testing}
Bates, S., Cand{\`e}s, E., Lei, L., Romano, Y., and Sesia, M. (2023).
\newblock Testing for outliers with conformal p-values.
\newblock {\em Ann. Statist.}, 51(1):149--178.

\bibitem[Benjamini and Bogomolov, 2013]{benjamini2014selective}
Benjamini, Y. and Bogomolov, M. (2013).
\newblock {Selective Inference on Multiple Families of Hypotheses}.
\newblock {\em Journal of the Royal Statistical Society Series B: Statistical
  Methodology}, 76(1):297--318.

\bibitem[Benjamini and Hochberg, 1995]{BH1995}
Benjamini, Y. and Hochberg, Y. (1995).
\newblock Controlling the false discovery rate: a practical and powerful
  approach to multiple testing.
\newblock {\em J. Roy. Statist. Soc. Ser. B}, 57(1):289--300.

\bibitem[Benjamini et~al., 2006]{BKY2006}
Benjamini, Y., Krieger, A.~M., and Yekutieli, D. (2006).
\newblock Adaptive linear step-up procedures that control the false discovery
  rate.
\newblock {\em Biometrika}, 93(3):491--507.

\bibitem[Benjamini and Yekutieli, 2001]{BY2001}
Benjamini, Y. and Yekutieli, D. (2001).
\newblock The control of the false discovery rate in multiple testing under
  dependency.
\newblock {\em Ann. Statist.}, 29(4):1165--1188.

\bibitem[Benjamini and Yekutieli, 2005]{benjamini2005false}
Benjamini, Y. and Yekutieli, D. (2005).
\newblock False discovery rate--adjusted multiple confidence intervals for
  selected parameters.
\newblock {\em Journal of the American Statistical Association},
  100(469):71--81.

\bibitem[Blanchard and Roquain, 2008]{BR2008}
Blanchard, G. and Roquain, E. (2008).
\newblock Two simple sufficient conditions for {FDR} control.
\newblock {\em Electron. J. Stat.}, 2:963--992.

\bibitem[Courty et~al., 2017]{courty2017joint}
Courty, N., Flamary, R., Habrard, A., and Rakotomamonjy, A. (2017).
\newblock Joint distribution optimal transportation for domain adaptation.
\newblock In {\em Advances in neural information processing systems 30 (NIPS
  2017)}, volume~30.

\bibitem[Ding et~al., 2023]{ding2023class}
Ding, T., Angelopoulos, A.~N., Bates, S., Jordan, M.~I., and Tibshirani, R.~J.
  (2023).
\newblock Class-conditional conformal prediction with many classes.
\newblock {\em arXiv preprint arXiv:2306.09335}.

\bibitem[Efron et~al., 2001]{ETST2001}
Efron, B., Tibshirani, R., Storey, J.~D., and Tusher, V. (2001).
\newblock Empirical {B}ayes analysis of a microarray experiment.
\newblock {\em J. Amer. Statist. Assoc.}, 96(456):1151--1160.

\bibitem[Ferreira and Zwinderman, 2006]{FZ2006}
Ferreira, J.~A. and Zwinderman, A.~H. (2006).
\newblock On the {B}enjamini-{H}ochberg method.
\newblock {\em Ann. Statist.}, 34(4):1827--1849.

\bibitem[Gao et~al., 2023]{gao2023constructive}
Gao, Z., Hu, W., and Zhao, Q. (2023).
\newblock A constructive approach to selective risk control.
\newblock {\em arXiv preprint arXiv:2401.16651}.

\bibitem[Gazin et~al., 2024]{gazin2023transductive}
Gazin, U., Blanchard, G., and Roquain, E. (2024).
\newblock Transductive conformal inference with adaptive scores.
\newblock In Dasgupta, S., Mandt, S., and Li, Y., editors, {\em Proceedings of
  The 27th International Conference on Artificial Intelligence and Statistics},
  volume 238 of {\em Proceedings of Machine Learning Research}, pages
  1504--1512. PMLR.

\bibitem[Guo and Romano, 2015]{guo2015stepwise}
Guo, W. and Romano, J. (2015).
\newblock On stepwise control of directional errors under independence and some
  dependence.
\newblock {\em Journal of Statistical Planning and Inference}, 163:21--33.

\bibitem[Gupta et~al., 2022]{gupta2022nested}
Gupta, C., Kuchibhotla, A.~K., and Ramdas, A. (2022).
\newblock Nested conformal prediction and quantile out-of-bag ensemble methods.
\newblock {\em Pattern Recognition}, 127:108496.

\bibitem[Jin and Candes, 2023]{Jin2023selection}
Jin, Y. and Candes, E.~J. (2023).
\newblock Selection by prediction with conformal p-values.
\newblock {\em Journal of Machine Learning Research}, 24(244):1--41.

\bibitem[Jin and Ren, 2024]{jin2024confidence}
Jin, Y. and Ren, Z. (2024).
\newblock Confidence on the focal: Conformal prediction with
  selection-conditional coverage.

\bibitem[Lei et~al., 2018]{lei2018distribution}
Lei, J., G'Sell, M., Rinaldo, A., Tibshirani, R.~J., and Wasserman, L. (2018).
\newblock Distribution-free predictive inference for regression.
\newblock {\em J. Amer. Stat. Assoc.}, 113(523):1094--1111.

\bibitem[Lei et~al., 2014]{lei2014conformal}
Lei, J., Rinaldo, A., and Wasserman, L. (2014).
\newblock A conformal prediction approach to explore functional data.
\newblock {\em Annals of Mathematics and Artificial Intelligence}, 74:29--43.

\bibitem[Marandon et~al., 2024]{marandon2024adaptive}
Marandon, A., Lei, L., Mary, D., and Roquain, E. (2024).
\newblock Adaptive novelty detection with false discovery rate guarantee.
\newblock {\em The Annals of Statistics}, 52(1):157--183.

\bibitem[Marandon et~al., 2022]{marandon2022false}
Marandon, A., Rebafka, T., Roquain, E., and Sokolovska, N. (2022).
\newblock False clustering rate control in mixture models.
\newblock {\em arXiv preprint arXiv:2203.02597}.

\bibitem[Mary-Huard et~al., 2022]{mary2022error}
Mary-Huard, T., Perduca, V., Martin-Magniette, M.-L., and Blanchard, G. (2022).
\newblock Error rate control for classification rules in multiclass mixture
  models.
\newblock {\em The international journal of biostatistics}, 18(2):381--396.

\bibitem[Meyer et~al., 2023]{svmpackage}
Meyer, D., Dimitriadou, E., Hornik, K., Weingessel, A., and Leisch, F. (2023).
\newblock {\em e1071: Misc Functions of the Department of Statistics,
  Probability Theory Group (Formerly: E1071), TU Wien}.
\newblock R package version 1.7-13.

\bibitem[Papadopoulos et~al., 2002]{papadopoulos2002inductive}
Papadopoulos, H., Proedrou, K., Vovk, V., and Gammerman, A. (2002).
\newblock Inductive confidence machines for regression.
\newblock In {\em 13th European Conference on Machine Learning (ECML 2002)},
  pages 345--356. Springer.

\bibitem[Ramdas et~al., 2019]{ramdas2019unified}
Ramdas, A.~K., Barber, R.~F., Wainwright, M.~J., and Jordan, M.~I. (2019).
\newblock A unified treatment of multiple testing with prior knowledge using
  the p-filter.
\newblock {\em The Annals of Statistics}, 47(5):2790--2821.

\bibitem[Rava et~al., 2021]{rava2021burden}
Rava, B., Sun, W., James, G.~M., and Tong, X. (2021).
\newblock A burden shared is a burden halved: A fairness-adjusted approach to
  classification.
\newblock {\em arXiv preprint arXiv:2110.05720}.

\bibitem[Romano and Wolf, 2005]{RW2005}
Romano, J.~P. and Wolf, M. (2005).
\newblock Exact and approximate stepdown methods for multiple hypothesis
  testing.
\newblock {\em J. Amer. Statist. Assoc.}, 100(469):94--108.

\bibitem[Romano et~al., 2019]{romano2019conformalized}
Romano, Y., Patterson, E., and Candes, E. (2019).
\newblock Conformalized quantile regression.
\newblock {\em Advances in neural information processing systems}, 32.

\bibitem[Roquain and Villers, 2011]{RV2011}
Roquain, E. and Villers, F. (2011).
\newblock Exact calculations for false discovery proportion with application to
  least favorable configurations.
\newblock {\em Ann. Statist.}, 39(1):584--612.

\bibitem[Sadinle et~al., 2019]{sadinle2019least}
Sadinle, M., Lei, J., and Wasserman, L. (2019).
\newblock Least ambiguous set-valued classifiers with bounded error levels.
\newblock {\em Journal of the American Statistical Association},
  114(525):223--234.

\bibitem[Sarkar, 2008]{Sar2008}
Sarkar, S.~K. (2008).
\newblock On methods controlling the false discovery rate.
\newblock {\em Sankhya, Ser. A}, 70:135--168.

\bibitem[Sesia and Romano, 2021]{sesia2021conformal}
Sesia, M. and Romano, Y. (2021).
\newblock Conformal prediction using conditional histograms.
\newblock {\em Advances in Neural Information Processing Systems},
  34:6304--6315.

\bibitem[Storey, 2002]{Storey2002}
Storey, J.~D. (2002).
\newblock A direct approach to false discovery rates.
\newblock {\em J. R. Stat. Soc. Ser. B Stat. Methodol.}, 64(3):479--498.

\bibitem[Vaishnav et~al., 2022]{Vaishnav2022}
Vaishnav, E.~D., de~Boer, C.~G., Molinet, J., Yassour, M., Fan, L., Adiconis,
  X., Thompson, D.~A., Levin, J.~Z., Cubillos, F.~A., and Regev, A. (2022).
\newblock The evolution, evolvability and engineering of gene regulatory dna.
\newblock {\em Nature}, 603(7901):455--463.

\bibitem[Vovk et~al., 2005]{vovk2005algorithmic}
Vovk, V., Gammerman, A., and Shafer, G. (2005).
\newblock {\em Algorithmic learning in a random world}.
\newblock Springer.

\bibitem[Weinstein and Ramdas, 2020]{weinstein20online}
Weinstein, A. and Ramdas, A. (2020).
\newblock Online control of the false coverage rate and false sign rate.
\newblock {\em Proceedings of the 37th International Conference on Machine
  Learning, PMLR}, 119:10193--10202.

\bibitem[Weinstein and Yekutieli, 2020]{weinstein20}
Weinstein, A. and Yekutieli, D. (2020).
\newblock {Selective sign-determining multiple confidence intervals with {FCR}
  control}.
\newblock {\em Statistica Sinica}, 30:531--555.

\bibitem[Zhao and Su, 2023]{zhao2023controlling}
Zhao, G. and Su, Z. (2023).
\newblock Controlling fsr in selective classification.
\newblock {\em arXiv preprint arXiv:2311.03811}.

\end{thebibliography}

\appendix
\section{Connections to existing works}\label{subsec-prevworks}

For the class-conditional model, we can view $(Y_{n+i})_{i\in \range{m}}$ as fixed. Thus, informative prediction sets can be viewed as informative confidence sets for parameters. This has been considered in a particular setting by \cite{weinstein20}. They considered building confidence intervals only for the selected parameters that will be sign determining. They showed that if the test statistics are independent and the confidence intervals satisfy some monotonicity properties, then the FCR can be controlled. Their theoretical framework is different than the one we consider, but their approach of selecting only sign-determining confidence intervals is very similar to ours, of selecting only informative prediction sets when informativeness is defined by sign determination. Moreover, this approach has been considered in \cite{weinstein20online} with the broader scope of only reporting confidence intervals if they are ``localizing'' appropriately the true parameter in  the sense that the confidence interval is entirely contained in one element of a pre-specified partition of the $\ZZ$ space. They investigate this task in the online setting where the sequence of unknown parameters is fixed, and at each time step an  independent observation is observed for the corresponding parameter, which is substantially different than our batch setting where the conformal $p$-values are dependent and the outcome may be random. We show in Remark~\ref{rem:asaf} that our informativeness theory covers their localizing notion in our setting.
Next, we discuss inspiring works connected to ours that assume the iid model.

 \cite{zhao2023controlling} suggested procedures for average error control in multi-class classification, so $\mathcal C_{n+i}$ are singletons. Since their procedure  only reports a single class for each selected  example, ambiguous examples will not be selected.  However,  in most classification tasks, there are examples whose true class is difficult to determine, yet it is possible to narrow down the possible set of classes \citep{sadinle2019least}. We  suggest procedures that  produce $\mathcal C_{n+i}$ that are not necessarily singletons for $K>2$. For $K=2$,  their procedure coincides with an instance in our framework, see details in \S~\ref{subsec:min-iid}. However, for $K>2$, while our suggestion as well as their suggestion provides level $\alpha$ FCR control, we select more examples, and although the prediction sets may be at a coarser resolution than singletons, they are still informative since they narrow down the possible set of classes. We note that for $K>2$, we can recover their procedure if we define as informative only prediction sets of size one, since then their procedure coincides with \texttt{InfoSP} for the iid model.

\cite{Jin2023selection} addressed the problem of discovering outcomes with values above a threshold. So $\mathcal C_{n+i}$ is of the form $(c_{i}, \infty)$ for predefined $(c_i)_{i\in \range{m}}$. They cast the problem as that of testing the family of null hypotheses $\{Y_{n+i}\leq c_i, i\in \range{m} \}$. In \S~\ref{sec:JCselect}, it is demonstrated that by defining $Y_{n+i}\leq c_i$ as uninformative,  we can complement the discoveries of \cite{Jin2023selection} with one-sided prediction intervals, while providing the same false discovery rate guarantee on the selected. Moreover, we show how to obtain two-sided prediction intervals for the informative examples in \S~\ref{sec:excludeab}.

\cite{bao2024selective}  considered the regression framework. Their first result (Proposition~1 therein) is to prove that for  selection rules that do not depend on the calibration sample, classic conformal prediction intervals at level $\alpha|\mathcal S|/m$ for the $\mathcal S$ selected examples (i.e., the correction factor $|\mathcal S|/m$ for selection suggested in \citealp{benjamini2005false}), provide level $\alpha $ FCR control.  For our purpose of informative selection, this result is not useful because informative selection involves all conformal $p$-values and therefore involves the calibration sample in a specific way. We need a different set of assumptions that are detailed in our novel Theorem \ref{th:generalBY}.

\cite{bao2024selective} further argue that the resulting prediction intervals, $\big(\mathcal C^{\alpha|\mathcal S|/m}_{n+i}\big)_{i\in \mathcal S}$ are too wide. 
They suggest a novel approach that performs selection on both the calibration set and test set, and then constructs $\alpha$ level conformal prediction intervals for the selected test candidates using the conditional empirical distribution obtained by the post-calibration set. For exchangeable selection rules, they show that the FCR is controlled at level $\alpha$. 
Their selection process cannot guarantee that all the constructed prediction intervals are of interest to the analyst.  
For example, for predicting the affinity of drug-target pairs, the analyst may not be interested in pairs with affinity below, say, $y_0$ (the case in item 1 of Example \ref{ex:regression}).
Using the novel approach of \cite{bao2024selective}, prediction intervals will be constructed following the selection of  calibration and test examples for which the predicted affinity from the machine learning algorithm is above a selection threshold. They require that the selection procedure be a thresholding procedure of the scores $S_{Y_{j}}(X_{j})$
with a threshold $\tau$ that is either independent of the calibration sample (their Proposition 1), or exchangeable with respect to both calibration and test samples (their Theorem 1).  Some of these prediction intervals may include $y_0$, and thus be useless for the analyst (an illustration is given in \S~\ref{sec:comparison} in SM, see Figure~\ref{fig:comparison} therein). However, it is not possible to additionally select only the examples with prediction intervals above $y_0$, since after performing this additional selection, the prediction intervals of \cite{bao2024selective} on the selected no longer have an $\alpha$ level FCR guarantee. Our procedures for regression thus complement the work of \cite{bao2024selective} when the focus is that all the  prediction intervals eventually constructed are informative.  

In a very recent work\footnote{This work appeared when we were in the final stage of writing, our work has been done independently.}, \cite{jin2024confidence} generalize the work of \cite{bao2024selective}, by considering more general selection rules for finite-sample exact coverage conditional on the unit being selected. Their conditional guarantee is achieved by a careful swapping argument which identifies for each selection rule, the appropriate subset of the calibration examples for each example from the test sample. Their conditional error guarantee implies FCR control under some conditions. 
\add{The control is valid for any selection rule that is exchangeable within the calibration sample, which is a  weaker assumption than concordance. Hence, their selection may be based on some notion of informativeness. }
\add{However, since the prediction sets should be inflated after applying the selection, the latter are not necessary informative according to our definition.}
For example, after selecting by a multiple testing procedure on the family of null hypotheses  $\{Y_{n+i}\leq c_i, i\in \range{m} \}$ in the setting of \cite{Jin2023selection}, their prediction sets  may include the $c_i$'s for some of the discoveries.  We  suggest procedures where selection and construction of prediction sets are inseparable, since we require that each selected prediction set be informative (along with the requirement that FCR $\leq \alpha$).

 \section{Application to directional FDR control}
\label{sec:directionalFDR}

Consider for this section that we have $(X_i,Z_i)_{i\in \range{n+m}}$ with real-valued outcomes $Z_i\in \R$, and we aim at excluding $Z_{n+i} \in [a,b]$, as well as at deciding whether  $Z_{n+i}<a$ or $Z_{n+i}>b$ (without producing prediction intervals for the $Z_{n+i}$), for two benchmark values $a\leq b$. More formally, we want to build a selection $\mathcal{S}\subset \range{m}$ and a (point-wise) decision $\widehat{Y}\in \{1,3\}$ from the observed samples such that 
\begin{equation}\label{dirFDR}
\FDR_{{\tiny \mbox{dir}}}(\mathcal{S},\widehat{Y}) :=   \sup_{P_{X|Y}, Y} \E_{X\sim P_{X|Y}}\left[\frac{\sum_{i\in \mathcal{S}} \ind{Y_{n+i}\neq \widehat{Y}_{n+i}}}{1\vee |\mathcal{S}|}\right]\leq \alpha,
\end{equation}
where $
Y_j=\ind{Z_{j}<a}+2\ind{Z_{j} \in [a,b]}+ 3\ind{Z_{j}>b},
$
$j\in \range{n+m}$.
Our theory yields the following result.

\begin{corollary}\label{cor-dFDR}
   Consider the class-conditional (classification) model on $(X_i,Y_i)_{i\in \range{n+m}}$  
with $K=3$ classes.
   Consider any score function satisfying Assumption~\ref{as:noties} (in this classification model). Consider the  procedure that selects $\mathcal{S}=\BH(\mathbf{q})$ with $$q_i= \max(\pcond{2}{i}, \min(\pcond{1}{i},\pcond{3}{i})), \:i\in \range{m},$$ where $\pcond{y}{i}$ are the class-conditional $p$-values computed as in \eqref{confpvalues0}, and with the decision
    \begin{align*}
    \widehat{Y}_{n+i}&= \ind{S_1(X_{n+i})\geq S_3(X_{n+i})}+3 \ind{S_3(X_{n+i})> S_1(X_{n+i})}, i\in \mathcal{S}.
\end{align*}
 Then this procedure controls the directional FDR at level $\alpha$ in the sense of \eqref{dirFDR}.
\end{corollary}

\begin{proof}
We consider  $\mathcal{I}=\{C\subset \range{K}\::\: y_0\notin C, |C|\leq 1\}$ for $y_0=2$ 
(see Example~\ref{ex:combine}) in the classification setting based on the sample $(X_j,Y_j)$'s ($K=3$), and we note that the FCR coincides with the directional FDR in that case. Hence, the result comes directly from Theorem~\ref{thm-gen-basic} (note that \texttt{InfoSP} is post-processed here, see \S~\ref{sec:classif}).
\end{proof}

\begin{remark}
For $b=a$ and continuous outcomes (with, say, $\pcond{2}{i}=0$),  there are almost surely only two classes, $Z_i=1$ corresponding to $Y_{i}<a$ and $Z_i=3$ corresponding to $Y_{i}>a$. The procedure is thus \texttt{InfoSP} for non-trivial classification, with class-calibrated $p$-values for $K=2$. This procedure coincides with the directional FDR procedure in \cite{guo2015stepwise} applied to conformal $p$-values for the parameters $Y_{n+1}, \ldots, Y_{n+m}$. The proof of validity in  \cite{guo2015stepwise} assumes that the test statistics for the $m$ hypotheses are independent. Interestingly, with the dependence induced by the conformal $p$-values, the same procedure is still valid, as formalized in Corollary \ref{cor-dFDR}. 
\end{remark}

\begin{remark}
\label{rem:predictioninterval}
If one wants to obtain prediction intervals in addition to the directional FDR control, it turns out that in the setting of Corollary~\ref{cor:abexcluded}, not only the FDR control (ii) holds but also the directional FDR control
\begin{equation*}
\FDR_{{\tiny \mbox{dir}}}(\mathcal{R}) :=   \sup_{P_{X,Y}} \E_{(X,Y)\sim P_{X,Y}}\left[\frac{\sum_{i\in \mathcal{S}} \ind{D_i=1,Y_{n+i}>a} + \ind{D_i=3, Y_{n+i}<b}}{1\vee |\mathcal{S}|}\right]\leq \alpha,
\end{equation*}
for the procedure $\mathcal{R}=(\mathcal{C}_{n+i})_{i\in \mathcal{S}}$ defined therein with the directional rule $D_i=\ind{ \mathcal{C}_{n+i} \subset (-\infty,a) }+3\ind{ \mathcal{C}_{n+i} \subset (b,+\infty) } $ for $i\in \mathcal{S}$.
It can be slightly less powerful than the directional FDR controlling procedure of Corollary~\ref{cor-dFDR} (because the latter uses classification scores), but provides additional information.

\end{remark}

\section{FCR control for \add{concordant} selection rules}\label{sec:proofFCRcontrol}

In this section, we present a general approach for FCR control (\S~\ref{sec:genBYresult}), which relies on the following:
\begin{itemize}
\item a general class of $p$-values, including both full-calibrated and class-calibrated conformal $p$-values (\S~\ref{sec:as:generalpvalues});
\item \add{the class of concordant selection rules  \citep{benjamini2005false,benjamini2014selective}, including informative selection rules (\S~\ref{sec:as:nonincreasingSelect}). }
\end{itemize}

\subsection{General statement}\label{sec:genBYresult}

Let $m\geq 1$, $\ZZ\subset \R$ and consider a family $\mathbf{p}=(p^{(y)}_i, i\in \range{m}, y\in \ZZ)$ of random variables taking values in $[0,1]$ and a vector $Y=(Y_{n+i}, i\in \range{m})$ taking values in $\ZZ$. We make use of the notation $\mathbf{p}_{-i}:=(p^{(y)}_{j})_{j\neq i,y\in \ZZ}$ for all $i\in \range{m}$. 

First, we introduce the following assumption on $\mathbf{p}$ and $Y$:
\begin{assumption}\label{as:generalpvalues}
There exists a vector $W=(W_{i}, i\in \range{m})$ of multivariate random variables with
\begin{itemize}
\item[(i)] for all $i\in \range{m}$, the random vector $\mathbf{p}_{-i}=(p^{(y)}_{j})_{j\neq i,y\in \ZZ}$ can be almost surely written as $\Psi_i(p^{(Y_{n+i})}_{i},W_i)$ where $u\in [0,1] \mapsto \Psi_i(u,W_i)\in \R^{\range{m-1}\times \ZZ}$ is a nondecreasing function (in a coordinate-wise sense for the image space). 
\item[(ii)] for all $i\in \range{m}$,  the following super-uniformity property holds
\begin{equation}\label{equ:superunifgen}
\P(p_i^{(Y_{n+i})}\leq t \:|\: W_i)\leq t, \:\: t\in [0,1].
\end{equation}
\end{itemize}
\end{assumption}

Second, for any selection rule $\mathcal{S}\subset \range{m}$, that is, any measurable function of $\mathbf{p}$ valued in the subsets of $\range{m}$, we introduce the following quantity  \citep{benjamini2005false,bao2024selective}:
\begin{equation}\label{equsmin}
s_i^{\min}(\mathbf{p}_{-i})=\min_{z\in \mathcal{A}(\mathbf{p}_{-i})} \left |\mathcal{S}(z,\mathbf{p}_{-i})\right | ,\:\:\: i\in\range{m},
\end{equation}
where $\mathcal{A}(\mathbf{p}_{-i})=\{ z\in[0,1]^\ZZ : i\in\mathcal{S}(z,\mathbf{p}_{-i})\}$ (and by convention $s_i^{\min}(\mathbf{p}_{-i})=0$ if $\mathcal{A}(\mathbf{p}_{-i})$ is empty). \add{The following assumption corresponds to the concordance assumption of \cite{benjamini2014selective} (itself extending the former concordance assumption of \citealp{benjamini2005false} to a collection of $p$-value families).}

\begin{assumption}\label{as:nonincreasingSelect}
\delete{For $i\in\range{m}$, $\mathcal{A}(\mathbf{p}_{-i})$ is almost surely not empty and $s_i^{\min}(\mathbf{p}_{-i})\geq 1$ is a coordinate-wise nonincreasing function of $\mathbf{p}_{-i}$.}
\add{For all $i\in\range{m}$, $s_i^{\min}(\mathbf{p}_{-i})$ is a coordinate-wise nonincreasing function of $\mathbf{p}_{-i}$.} 
\end{assumption}

\delete{While Assumption~\ref{as:nonincreasingSelect} is not satisfied for a selection rule that selects the $k$-smallest p-values, for some $k\geq 1$ (because $\mathcal{A}(\mathbf{p}_{-i})$ can be empty), it is satisfied for $p$-value thresholding rules which are of interest in our setting (see Section~\ref{sec:as:nonincreasingSelect}).
}

\add{As we show in Section~\ref{sec:as:nonincreasingSelect}, $\mathcal{I}$-informative selection rules are concordant. Other examples include monotone $p$-value-based thresholding rules. However, top-$k$ type selections are  concordant only if we use a tie-breaker in the conformal $p$-value definitions (that is, if we use randomized conformal $p$-values \citealp{vovk2005algorithmic,bates2023testing,Jin2023selection}). More formal details are provided in Section~\ref{sec:as:nonincreasingSelect}.}

\begin{theorem}\label{th:generalBY}
Let us consider a $p$-value family $\mathbf{p}$ and a label/outcome vector $Y$ satisfying Assumption~\ref{as:generalpvalues}, for a selection rule $\mathcal{S}\subset \range{m}$ with $s_i^{\min}=s_i^{\min}(\mathbf{p}_{-i})$ \eqref{equsmin} satisfying Assumption~\ref{as:nonincreasingSelect}. Then, the procedure
$\mathcal{R}_{\alpha}=\big(\mathcal{C}^{\alpha \add{(1\vee s_i^{\min})}/m}_{n+i}(\mathbf{p})\big)_{i\in\mathcal{S}}$, for which  the prediction set is defined as in \eqref{PSuncorrected} with level \add{$\alpha (1\vee s_i^{\min})/m$}, satisfies
$
\E (\FCP (\mathcal{R}_\alpha,  Y  )) \leq \alpha,
$
for which the expectation $\E$ is taken w.r.t. the same probability as the one of \eqref{equ:superunifgen} in Assumption~\ref{as:generalpvalues} (ii).
\end{theorem}

The inequality $\E (\FCP (\mathcal{R}_\alpha,  Y  )) \leq \alpha$ in Theorem~\ref{th:generalBY} will be used both in the cases where $Y_{n+i}$ is fixed (conditional model) or not (iid model).
The proof is provided in \S~\ref{proof:th:generalBY}. Note that the considered assumptions make the proof particularly simple. Assumptions~\ref{as:generalpvalues} and~\ref{as:nonincreasingSelect} are studied in \S~\ref{sec:as:generalpvalues} and \S~\ref{sec:as:nonincreasingSelect}, respectively.

\subsection{Proof of Theorem~\ref{th:generalBY}}\label{proof:th:generalBY}

By definition \eqref{equFSP}, we have
\begin{align*}
\FCP(\mathcal{R}_\alpha,Y)&=\frac{\sum_{i\in \mathcal{S}} \ind{Y_{n+i} \notin \mathcal{C}_{n+i}^{\alpha (\add{1\vee s_i^{\min}})/m}(\mathbf{p})} }{1\vee |\mathcal{S}|}= \sum_{i\in \range{m}}
\frac{ \ind{ i \in \mathcal{S}, p_i^{(Y_{n+i})}\leq  \alpha (\add{1\vee s^{\min}_i})/m}}{1\vee |\mathcal{S}|}\\
&\leq \sum_{i\in \range{m}}
\frac{ \ind{  i \in \mathcal{S},p_i^{(Y_{n+i})}\leq  \alpha (\add{1\vee s^{\min}_i})/m}}{\add{1\vee s^{\min}_i}} \leq \sum_{i\in \range{m}}
\frac{ \ind{ p_i^{(Y_{n+i})}\leq  \alpha (\add{1\vee s^{\min}_i})/m}}{\add{1\vee s^{\min}_i}},
\end{align*}
where the first inequality follows from definition \eqref{equsmin} and the second inequality follows  from ignoring the fact that $i \in \mathcal{S}$. This entails
\begin{align*}
\E\left (\FCP(\mathcal{R}_\alpha,Y)\right )
&\leq \sum_{i\in \range{m}} 
\E\left(\frac{ \ind{  p_i^{(Y_{n+i})}\leq  \alpha (\add{1\vee s^{\min}_i})/m}}{\add{1\vee s^{\min}_i}}\right)\\
&= \sum_{i\in \range{m}} 
\E\left(\E\left[\frac{ \ind{  p_i^{(Y_{n+i})}\leq  \alpha (\add{1\vee s^{\min}_i}(\mathbf{p}_{-i}))/m}}{ \add{1\vee s^{\min}_i}(\mathbf{p}_{-i})}\:\middle|\: W_i\right]\right),
\end{align*}
by using the random vector ${W}=(W_{i}, i\in \range{m})$ defined in Assumption~\ref{as:generalpvalues}.
Now, combining Assumption~\ref{as:generalpvalues} (i) with Assumption~\ref{as:nonincreasingSelect}, we have 
 that 
 $$
 \add{1\vee s^{\min}_i}(\mathbf{p}_{-i}) = \add{1\vee s^{\min}_i}(\Psi_i(p^{(Y_{n+i})}_{i},W_i))
 $$
is a nonincreasing function of $p^{(Y_{n+i})}_{i}$.
By Assumption~\ref{as:generalpvalues} (ii) and applying Lemma~\ref{lem:BR} (conditionally on $W_i$), we obtain
$$
\E\left[\frac{ \ind{  p_i^{(Y_{n+i})}\leq  \alpha (\add{1\vee s^{\min}_i}(\mathbf{p}_{-i}))/m}}{ \add{1\vee s^{\min}_i}(\mathbf{p}_{-i})}\:\middle|\: W_i\right]\leq \frac{\alpha}{m},
$$
Putting this back into the FCR bound implies the result.

\subsection{Examining Assumption~\ref{as:generalpvalues}}\label{sec:as:generalpvalues}

We show here that the full-calibrated and class-calibrated $p$-value families satisfy Assumption~\ref{as:generalpvalues} in the iid and conditional models, respectively. For this, it is interesting to relax Assumption~\ref{as:noties} by assuming that the score functions can use the covariates of the calibration+test samples in an exchangeable way, as suggested in  \cite{marandon2024adaptive,gazin2023transductive}. 
This is useful for instance  when the learning sample and the calibration+test samples are not based on the same distribution, so that the scores may be improved by using transfer learning; we refer to \cite{gazin2023transductive} for more details on this.

\begin{assumption}\label{as:AllAdaptive}
{For any $y\in\ZZ$; $\Score_{y}(\cdot)$ is of the form $\Score_{y}\big(\cdot\:;\:\dtrain, (X_i)_{i\in \range{n+m}} \big)$ for an independent training data sample $\dtrain$ and is invariant by permutation of the elements of $(X_i)_{i\in \range{n+m}} $. In addition, the scores have no ties and the score function is regular in the sense of Assumption~\ref{as:noties}.}
\end{assumption}

\paragraph{$p$-value family $\pfullbf$}

The $p$-value family $\pfullbf$ given by \eqref{standardpvalue} satisfies Assumption~\ref{as:generalpvalues} in the iid model.

\begin{proposition}\label{prop:iidpvalues}
Let us consider a model where the variables $(X_i,Y_i)$, $i\in \range{n+m}$, are exchangeable  conditionally on $\dtrain$ and score functions satisfying Assumptions~\ref{as:AllAdaptive}.
 Then the $p$-value family $\pfullbf$ given by \eqref{standardpvalue} satisfies Assumption~\ref{as:generalpvalues}.\end{proposition}

\begin{proof}
Let us first establish Assumption~\ref{as:generalpvalues} (i) by following an argument similar to the one of \cite{bates2023testing,marandon2024adaptive}. By Assumption~\ref{as:AllAdaptive}, we can work on an event where the elements of the sets $A_i=\{S_{Y_k}(X_k),k\in \range{n}\}\cup \{S_{Y_{n+i}}(X_{n+i})\}$ are all distinct.  
For any $i\in \range{m}$, we have by \eqref{standardpvalue} that for all $j\in \range{m} \backslash\{i\}$ and $y\in \ZZ$,
\begin{align*}
\pfull{y}{j}&= \frac{1}{n+1}\Big(1+\sum_{k=1}^n \ind{S_{Y_k}(X_k)\geq S_{y}(X_{n+j})} \Big)\\
&=\frac{1}{n+1}\Big(\ind{S_{Y_{n+i}}(X_{n+i})< S_{y}(X_{n+j})}+\sum_{s\in A_i} \ind{s\geq S_{y}(X_{n+j})} \Big).
\end{align*} 
Denoting $A_i=\{a_{i,(1)},\dots,a_{i,(n+1)}\}$ with $a_{i,(1)}>\dots>a_{i,(n+1)}$, and noting that $S_{Y_{n+i}}(X_{n+i})=a_{i,(\ell)}$ with $\pfull{y}{i}=\ell/(n+1)$, we may write 
$$\pfullbf_{-i}:=(\pfull{y}{j})_{j\in\range{m}\backslash\{i\}, y\in\ZZ}=\Psi(\pfull{Y_{n+i}}{i},\Wfull_i),$$
by letting for $u\in (0,1]$, 
\begin{align*}
\Wfull_i&:=(A_i,(S_{y}(X_{n+j}))_{ j\in\range{m}\backslash\{i\}, y\in\ZZ})\\
\Psi(u,\Wfull_i) &:= \left(\frac{1}{n+1}\Big(\ind{a_{i,(\lceil u(n+1)\rceil )}< S_{y}(X_{n+j})}+\sum_{s\in A_i} \ind{s\geq S_{y}(X_{n+j})} \Big)\right)_{j\in\range{m}\backslash\{i\}, y\in\ZZ}.
\end{align*} 
 Clearly, each of the elements inside $\Psi(u,\Wfull_i)$ is nondecreasing in $u$,  which gives Assumption~\ref{as:generalpvalues} (i) (note that $\Psi$ does not depend on $i$ in this context).

Next, we establish Assumption~\ref{as:generalpvalues} (ii). Since the variables $(X_i,Y_i)$, $i\in \range{n+m}$, are exchangeable and since the scores functions, the set $A_i$, and $(S_{y}(X_{n+j}))_{ j\in\range{m}\backslash\{i\}, y\in\ZZ}$ are invariant by permutations of  $(X_1,Y_1),\dots, (X_n,Y_n),(X_{n+i},Y_{n+i})$ (by using Assumptions~\ref{as:AllAdaptive}), we have that the random vector $(S_{Y_1}(X_1),\dots, S_{Y_n}(X_n),S_{Y_{n+i}}(X_{n+i}))$ is exchangeable conditionally on $\Wfull_i$. Since there are no ties in the vector, it follows that $(n+1)\pfull{Y_{n+i}}{i}$ (i.e., the rank of $S_{Y_{n+i}}(X_{n+i})$ in $A_i$) is uniformly distributed in $\range{n+1}$ conditionally on $\Wfull_i$. Thus Assumption~\ref{as:generalpvalues} (ii) is satisfied (note that the conditional probability is well defined thanks to the regularity condition in Assumption~\ref{as:AllAdaptive}).
\end{proof}

\paragraph{$p$-value family $\pcondbf$}

The $p$-value family $\pcondbf$ given by \eqref{confpvalues0} satisfies Assumption~\ref{as:generalpvalues} in the conditional model.

\begin{proposition}\label{prop:condpvalues}
In the case $\ZZ=\range{K}$, let us consider score functions satisfying Assumptions~\ref{as:AllAdaptive} and a model for the variables $(X_i,Y_i)$, $i\in \range{n+m}$, for which $(Y_i, i\in \range{n+m})$ is a deterministic vector and, for each $y\in \ZZ$, the variables $(X_i)_{i\in \range{n+m}: Y_i=y}$, are exchangeable.
 Then the $p$-value family $\pcondbf$ given by \eqref{confpvalues0} satisfies Assumption~\ref{as:generalpvalues}.\end{proposition}

The proof is similar to the proof of Proposition~\ref{prop:iidpvalues}. We provide it below for completeness. 

\begin{proof}
For any $i\in \range{m}$, we have by \eqref{confpvalues0} that for all $j\in \range{m} \backslash\{i\}$ and $y\in \ZZ$,
\begin{align*}
\pcond{y}{j}&= \frac{1}{|\mathcal{D}^{(y)}_{{\tiny \mbox{cal}}}|+1}\Big(1+\sum_{k\in \mathcal{D}^{(y)}_{{\tiny \mbox{cal}}}} \ind{S_{y}(X_k)\geq S_{y}(X_{n+j})} \Big)\\
&=\frac{1}{n^{(y)}_i}\Big(\ind{S_{y}(X_{n+i})< S_{y}(X_{n+j})}+\sum_{s\in A^{(y)}_i} \ind{s\geq S_{y}(X_{n+j})} \Big),
\end{align*} 
by letting $n^{(y)}_i=|A^{(y)}_i|$  and $A^{(y)}_i=\{S_{y}(X_k),k\in \mathcal{D}^{(y)}_{{\tiny \mbox{cal}}}\}\cup \{S_{y}(X_{n+i})\}$ whose elements can be assumed to be all distinct by Assumptions~\ref{as:AllAdaptive} (strictly, this is only true for labels $y\in \ZZ$ that appears at least once in the fixed sample $(Y_i, i\in \range{n+m})$, but the labels $y$ not appearing in $(Y_i, i\in \range{n+m})$ can be trivially handled because they correspond to $p$-values all equal to $1$). 
Denoting $A^{(y)}_i=\{a^{(y)}_{i,(1)},\dots,a^{(y)}_{i,(n^{(y)}_i)}\}$ with  $a^{(y)}_{i,(1)}>\dots>a^{(y)}_{i,(n^{(y)}_i)}$, and noting that $S_{y}(X_{n+i})=a^{(y)}_{i,(n^{(y)}_i \pcond{y}{i})}$, 
we may write 
$$\pcondbf_{-i}:=(\pcond{y}{j})_{j\in\range{m}\backslash\{i\}, y\in\ZZ}=\Psi_i(\pcond{Y_{n+i}}{i},\Wcond_i),$$
by letting 
\begin{align}\label{equWcond}
&\Wcond_i:=\big( (A^{(Y_{n+i})}_i,(S_{Y_{n+i}}(X_{n+j}))_{ j\in\range{m}\backslash\{i\}}); (X_j,Y_j)_{ j\in\range{n+m}: Y_j\neq Y_{n+i}}\big)
\end{align} 
and for $u\in (0,1]$,  $\Psi_i(u,\Wcond_i) := (\Psi^{(y)}_i(u,\Wcond_i))_{y \in \ZZ}$ where 
\begin{align*}
&\Psi^{(y)}_i(u,\Wcond_i) \\
&:= \left\{\begin{array}{cc}
  \Big( \frac{1}{n^{(y)}_i}\Big(\ind{a_{i,(\lceil u n^{(y)}_i\rceil )}< S_{y}(X_{n+j})}+\sum_{s\in A^{(y)}_i} \ind{s\geq S_{y}(X_{n+j})} \Big)\Big)_{j\in \range{m} \backslash\{i\}}  & \mbox{ if $y=Y_{n+i}$;}  \\
  (\pcond{y}{j})_{j\in \range{m} \backslash\{i\}}   &  \mbox{ if $y\neq Y_{n+i}$.}
\end{array}\right.
\end{align*} 
 Clearly, the elements inside $\Psi_i(u,\Wcond_i)$ are nondecreasing in $u$ which gives Assumption~\ref{as:generalpvalues} (i).

Next, to establish Assumption~\ref{as:generalpvalues} (ii), we use that by assumption the vector $$(S_{Y_{n+i}}(X_{n+i}),S_{Y_{n+i}}(X_k),k\in \mathcal{D}^{(Y_{n+i})}_{{\tiny \mbox{cal}}})$$ is exchangeable conditionally on $\Wcond_i$ (by permutation invariance of $\Wcond_i$, which also comes from Assumptions~\ref{as:AllAdaptive}).
Since there are no ties in the vector, it follows that $n^{(Y_{n+i})}_i\pcond{Y_{n+i}}{i}$ (i.e., the rank of $S_{Y_{n+i}}(X_{n+i})$ in $A^{(Y_{n+i})}_i$) is uniformly distributed in $\range{n^{(Y_{n+i})}_i}$ conditionally on $\Wcond_i$. Thus  Assumption~\ref{as:generalpvalues} (ii) is satisfied. 
\end{proof}

\subsection{Concordant selection rules}\label{sec:as:nonincreasingSelect}

\begin{proposition}\label{prop:inforuleok}
 Assumption~\ref{as:nonincreasingSelect} holds for the informative selection rule $\mathcal{S}(\mathbf{p})=\BH(\mathbf{q})$ with $s_i^{\min}(\mathbf{p}_{-i})=|\mathcal{S}(\mathbf{p})|$ whenever $i\in \mathcal{S}(\mathbf{p})$.
\end{proposition}

\delete{Note that the above result is also true for selection rule of the type $\mathcal{S}(\mathbf{p})=\{i\in \range{m}\::\: q_i\leq \tau\}$ for some fixed threshold $\tau$. }

\begin{proof}
First, a classical property of BH procedure is the following leave-one-out property: for all $i\in \range{m}$, 
$i\in \BH(\mathbf{q})$ if and only if $\BH(\mathbf{q})= \BH(\mathbf{q}^{0,i})$ where $\mathbf{q}^{0,i}$ is the vector $\mathbf{q}$ where the $i$-th coordinate has been replaced by $0$, see for instance \cite{FZ2006,Sar2008,RV2011,ramdas2019unified}. This implies
\begin{align*}
s_i^{\min}(\mathbf{p}_{-i})&=\min_{z\in[0,1]^\ZZ : i\in\mathcal{S}(z,\mathbf{p}_{-i})} \left |\mathcal{S}(z,\mathbf{p}_{-i})\right |=|\BH(\mathbf{q}^{0,i})|.
\end{align*}
Hence, the result is proved as soon as $\mathbf{q}$ is proved to be coordinate-wise nondecreasing in each $p$-value trajectory.
This holds by Lemma~\ref{lem:monotonicityq}.
\end{proof}

\begin{lemma}\label{lem:monotonicityq}
Suppose Assumption~\ref{assI}, then \add{for each $i\in \range{m}$, almost surely, $q_i$ defined by \eqref{qformula} is a nondecreasing function of each  ${p}^{(y)}_i$, $y\in \mathcal{Y}$.}
\delete{$\mathbf{q}$ defined by \eqref{qformula} is a nondecreasing function of each $p$-value ${p}^{(y)}_i$, $i\in \range{m}$ and $y\in \mathcal{Y}$.}
\end{lemma}

\begin{proof}
\add{First note that each $\mathcal{C}^{\alpha}_{n+i}(\mathbf{p})$ only depends on the collection $\mathbf{p}=({p}^{(y)}_i, y\in \mathcal{Y}, j\in \range{m})$ through $({p}^{(y)}_i, y\in \mathcal{Y})$, see \eqref{predictionsetthreshold}, hence $q_i$ is only a function of $({p}^{(y)}_i, y\in \mathcal{Y})$. Also, this allows to denote $\mathcal{C}^{\alpha}_{n+i}(\mathbf{p})$ simply by $\mathcal{C}^{\alpha}_{n+i}({p}^{(y)}_i, y\in \mathcal{Y})$.
Let $({p}^{(y)}_i, y\in \mathcal{Y})$ and $({p}'^{(y)}_i, y\in \mathcal{Y})$  be two $p$-value collections with ${p}^{(y)}_i\leq {p}'^{(y)}_i$, for all $y\in \mathcal{Y}$, with corresponding values $q_i$ and $q'_i$. By definition, $\mathcal{C}^{q'_i}_{n+i}({p}'^{(y)}_i, y\in \mathcal{Y})\in \mathcal{I}$ and $\mathcal{C}^{q'_i}_{n+i}({p}^{(y)}_i, y\in \mathcal{Y})\subset \mathcal{C}^{q'_i}_{n+i}({p}'^{(y)}_i, y\in \mathcal{Y})$. By Assumption~\ref{assI} (i) (iii), 
we have $\mathcal{C}^{q'_i}_{n+i}({p}^{(y)}_i, y\in \mathcal{Y})\in \mathcal{I}$, which in turn implies $q_i\leq q'_i$ by definition of $q_i$.}
\end{proof}

\add{
Let us now discuss the concordance assumption for other, non-informative, selection rules. 
First note that Assumption~\ref{as:nonincreasingSelect} is true for any selection rule of the type $\mathcal{S}(\mathbf{p})=\{i\in \range{m}\::\: q_i\leq t\}$ for some deterministic threshold $t$. More generally, let us consider the monotone $p$-value-based thresholding selection rule
$$
\mathcal{S}(\mathbf{p})=\{i\in \range{m}\::\: f(p^{(y)}_i,y\in \ZZ)\leq t\},
$$
where $f:[0,1]^\ZZ\mapsto \R$ is some measurable coordinate-wise nondecreasing function. This selection rule satisfies Assumption~\ref{as:nonincreasingSelect}. Indeed, we easily check in this case $\mathcal{A}(\mathbf{p}_{-i})=\{ z\in[0,1]^\ZZ : f(z)\leq t\}$ (which does not depend on $\mathbf{p}_{-i}$) and thus 
$$s_i^{\min}(\mathbf{p}_{-i})=\Big(1+ \sum_{j\neq i} \ind{f(p^{(y)}_j,y\in \ZZ)\leq t}\Big)\ind{\exists z\in[0,1]^\ZZ : f(z)\leq t}
,$$
which is clearly nonincreasing in each $p$-value. For instance, examples include $f(p^{(y)}_i,y\in \ZZ)=\sum_{y\in \ZZ}\log(p^{(y)}_i)$ (Fisher's combination) in the classification case and $f(p^{(y)}_i,y\in \ZZ)=\int_{y\in \R} p^{(y)}_i e^{-y^2} dy$ in the regression case.
}

\add{
Finally, we mention that Assumption~\ref{as:nonincreasingSelect} is not  satisfied in our context for top-$k$ type selection rules because of ties. Indeed, let us consider the simple case where 
$
\mathcal{S}(\mathbf{p})
$
selects the $2$-smallest $p$-values among the $p$-values $p_i:=p_i^{(y_0)}$, $i\in \range{m}$, 
for some pre-specified $y_0\in \ZZ$ and $m=3$. 
When the $p$-values are at the minimum value $1/(n+1)$, the selection rule should make a unique choice and let us say that it selects $\{1,2\}$ when $(p_1,p_2,p_3)=(1/(n+1),1/(n+1),1/(n+1))$. Then, for  $i=3$, we easily obtain 
$s_i^{\min}(1/(n+1),1/(n+1),1/(n+1))=0$ (because $i=3$ is not selected), while $s_i^{\min}(1/(n+1),2/(n+1),1/(n+1))=2$ (because $i=3$ is selected). This contradicts the concordance assumption. 
}

\add{Nevertheless, by using a classical randomization trick for the conformal $p$-values \cite{vovk2005algorithmic,bates2023testing,Jin2023selection} 
the $p$-value family has almost surely distinct values and top-$k$ type selection rules are concordant (under light additional assumptions), which concurs with the observation made in \cite{benjamini2014selective}. More specifically, for $k\in \range{m}$, assume $
\mathcal{S}(\mathbf{p})
$ selects the $k$-smallest $f_i=f(p^{(y)}_i,y\in \ZZ)$ where $f:(0,1)^\ZZ\mapsto (0,1)$ is some measurable coordinate-wise nondecreasing surjective function and that the $f_i$'s have almost surely no ties. Then $\mathcal{A}(\mathbf{p}_{-i})=\{ z\in[0,1]^\ZZ : i\in\mathcal{S}(z,\mathbf{p}_{-i})\}$ is not empty, because $z$ such that $f(z)<\min_{j\in \range{m}}f_{j}$  always belong to it. Hence, by definition $s_i^{\min}(\mathbf{p}_{-i})=k$ for all values of $\mathbf{p}_{-i}$, which establishes concordance. For instance, the concordance holds when using randomized conformal $p$-values and selecting the $k$-smallest $\mathcal{I}$-adjusted $p$-values, with $f(p^{(y)}_i,y\in \ZZ)=q_i$ being equal to $p_i^{(y_0)}$ (excluding $y_0$), $\max_{y\in [a,b]} p_i^{(y)}$ (excluding $[a,b]$) or $\min_{k\in \range{K}} p_i^{(k)}$ (non-trivial classification).
}

\section{Procedures using weighted class-calibrated $p$-values}\label{SM-sec:adaptproc}

Procedure \texttt{InfoSP} does not take into account the proportion of labels in each class in the test sample. However, in the classification case, these proportions are estimable from the data, and the estimates can aid inference. 

Let $\pi_k= \sum_{i=1}^m \ind{Y_{n+i}=k}/m$ the true proportion of examples with label $k$ in the test sample, $k\in \range{K}$, and consider the following possible estimates:
\begin{itemize}
    \item Calibration-based estimator: 
$\hat{\pi}^{{\tiny \mbox{cal}}}_k = (|\mathcal{D}^{(k)}_{{\tiny \mbox{cal}}}|+1)/(n+1) =(1+\sum_{j\in \range{n}}\ind{Y_j=k})/(n+1),$ $k\in \range{K}$;
\item Storey-$\lambda$ estimator: 
$\hat{\pi}^{{\tiny \mbox{Storey}}}_k = \Big(1+\sum_{i=1}^m \ind{p^{(k)}_{i}>\lambda}\Big)/(m(1-\lambda))$, $k\in \range{K}$. It is similar to the classical estimator of true null hypotheses proportion in multiple testing \citep{Storey2002}. 
\end{itemize}
Given the class-calibrated $p$-value family $\pcondbf$ and one of the estimators $\hat{\pi}_k\in \{\hat{\pi}^{{\tiny \mbox{cal}}}_k,\hat{\pi}^{{\tiny \mbox{Storey}}}_k\}$, we define the corresponding adaptive (weighted) $p$-value collection $\pcondbf_{{\tiny \mbox{adapt}}}=(\pcond{k}{i,{\tiny \mbox{adapt}}}, k\in \range{K}, i \in\range{m}\}$ by
\begin{equation}\label{adaptpvalues}
\pcond{k}{i,{\tiny \mbox{adapt}}}=\frac{\hat{\pi}_k}{w_k} \pcond{k}{i},\:\:k\in \range{K}, i \in\range{m},
\end{equation}  
where $(w_k, k\in \range{K})$ are deterministic nonnegative weights such that $\sum_{k\in \range{K}} w_k=1$.
The rationale behind \eqref{adaptpvalues} is that the term $\hat{\pi}_k$ balances the false coverage errors between classes by trying to decrease $p$-values related to labels which do not appear much in the test sample. The weights $w_k$ are additional parameter that add flexibility, but they have to sum to one. If we use equal weights than the class-calibrated $p$-value is multiplied by $K\times \hat{\pi}_k$ which will be less than one only if $\hat{\pi}_k<1/K$.

Applying \texttt{InfoSP} with these adaptive $p$-values gives rise to a new procedure $\RinfoSP(\pcondbf_{{\tiny \mbox{adapt}}})$ that we denote by \texttt{Adapt-InfoSP}.

\begin{proposition}\label{mainthm0caladaptive}
Consider an informative subset collection $\mathcal{I}$ satisfying Assumption~\ref{assI}, score functions satisfying Assumption~\ref{as:noties} and 
the $p$-value collection $\pcondbf_{{\tiny \mbox{adapt-cal}}}=(\pcond{k}{i,{\tiny \mbox{adapt-cal}}}, k\in \range{K}, i \in\range{m}\}$ defined as in \eqref{adaptpvalues} with the calibration-based estimator $\hat{\pi}^{{\tiny \mbox{cal}}}_k = (|\mathcal{D}^{(k)}_{{\tiny \mbox{cal}}}|+1)/(n+1),$ $k\in \range{K}$. Then the corresponding \texttt{Adapt-InfoSP} procedure $\RinfoSP(\pcondbf_{{\tiny \mbox{adapt-cal}}})$ satisfies the following: 
\begin{itemize}
    \item[(i)] in the class-conditional model,
\begin{equation}
    \label{equ-boundadapt}
    \sup_{P_{X\mid Y},Y} \FCR(\RinfoSP(\pcondbf_{{\tiny \mbox{adapt-cal}}}),P_{X\mid Y},Y)\leq \alpha \: \sum_{y\in \range{K}}    w_y  \frac{n+1}{m} \frac{\sum_{i\in \range{m}}\ind{Y_{n+i} = y} } {\sum_{j \in \range{n}} \ind{Y_j= y}+1} .
\end{equation}
\item[(ii)] in the iid model, $\sup_{P_{X,Y}} \FCR(\RinfoSP(\pcondbf_{{\tiny \mbox{adapt-cal}}}),P_{X,Y})\leq \alpha$, that is, \texttt{Adapt-InfoSP} satisfies the FCR control \eqref{iidcontrol}.
\end{itemize}
\end{proposition} 

Proposition~\ref{mainthm0caladaptive} is proved in \S~\ref{proofmainthm0caladaptive}.
The bound \eqref{equ-boundadapt} is only sharp when the labels are generated in the same way in the calibration and test sample (which implies the correct control in (ii)), so \texttt{Adapt-InfoSP} should not be used if the label proportions are expected to be (very) different between calibration and test samples.

We illustrate in Figure \ref{fig-nullselection-classification-adaptive}  the performance of the adaptive procedures for nonnull selection and non-trivial selection, respectively, in the set-up of  unbalanced classes described in  \S~\ref{subsec-simul-bivariatenormal}.  
For the adaptive versions, $w_k=1/K$ for all $k\in\range{K}$, and $\lambda = 1/2$. We  consider two settings for the class-conditional model:  the test sample has class probabilities 0.1,0.1, 0.8 (i.e., a small label shift), and class probabilities 0.4,0.4,0.2 (i.e., a large label shift).

The only procedure with a theoretical class-conditional FCR guarantee is \texttt{InfoSP} on class calibrated $p$-values. The adaptive procedure with $\hat \pi_{k}^{cal}, k\in \range{K}$ violates FCR control only when the label-shift is large for non-trivial selection. Interestingly, this procedure has excellent power when the label shift is small. The adaptive procedure with $\hat \pi_{k}^{Storey}, k\in \range{K}$ is a close second in this case, but when the label shift is large it is no better than \texttt{InfoSP} in the settings considered. 
The fact that the adaptive procedure with $\hat \pi_{k}^{cal}, k\in \range{K}$  tends to control the FCR (or inflate it only by a little), suggests (arguably) that for power purposes it may be reasonable  to use it  if the label shift is small. 

In the simulations we carried out for the iid settings, $\texttt{Adapt-InfoSP}$ on class-conditional $p$-values had worse power than $\texttt{InfoSCOP}$ (omitted for brevity).  

\begin{figure}[h!]
\begin{center}
\vspace{-2cm}
 {\includegraphics[width=0.36\textwidth, height = 0.27\textheight, page=3]{legends.pdf}}

\vspace{-3cm}
  \begin{tabular}{ccc}
\hspace{-5mm}\includegraphics[width=0.34\textwidth, height = 0.2\textheight, page=2]{ClassCondNonnull1.pdf} &
\hspace{-5mm}\includegraphics[width=0.34\textwidth,  height = 0.2\textheight,page=1]{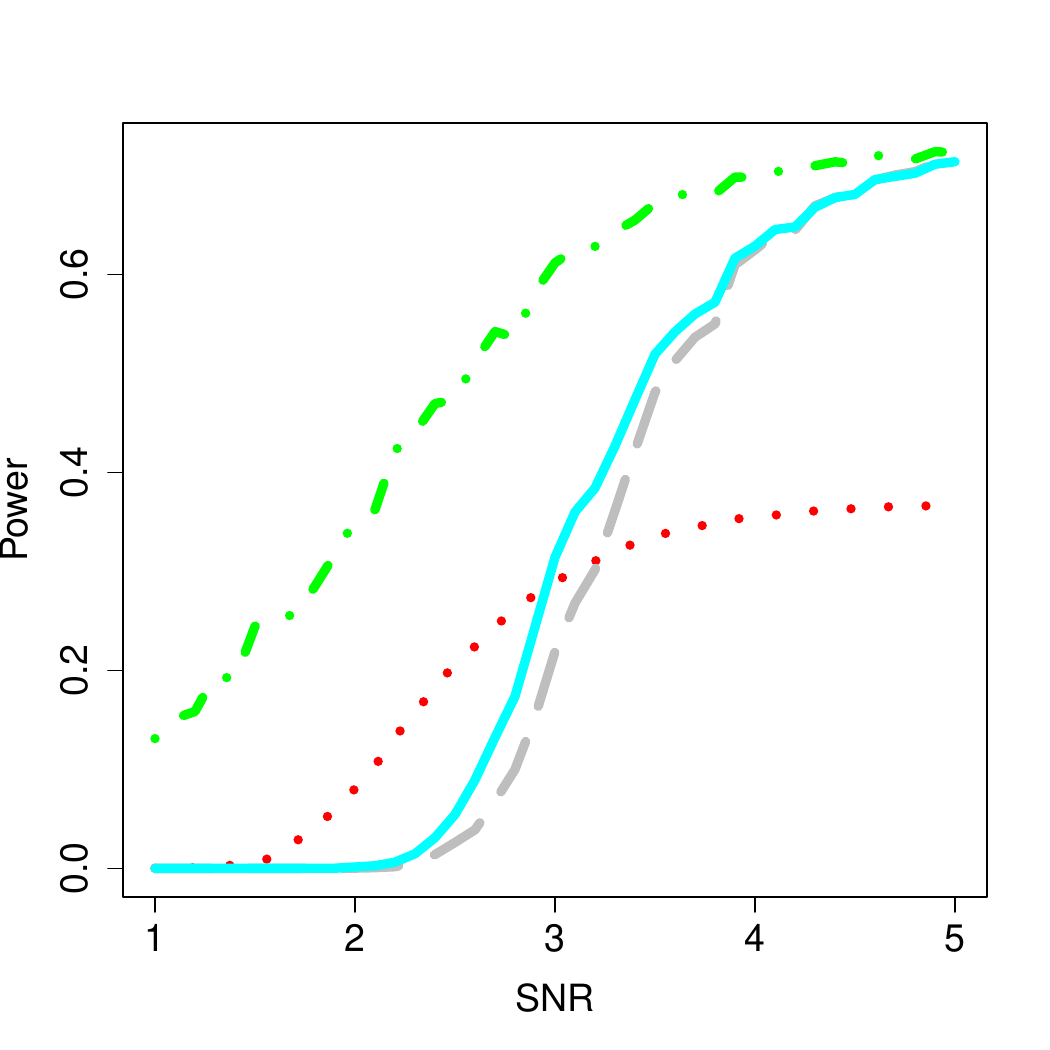} & \hspace{-5mm}\includegraphics[width=0.34\textwidth,  height = 0.2\textheight,page=5]{ClassCondNonnull1.pdf}\vspace{-5mm}\\
\hspace{-5mm}\includegraphics[width=0.34\textwidth, height = 0.2\textheight, page=2]{ClassCondNonnull2.pdf}&
\hspace{-5mm}\includegraphics[width=0.34\textwidth,  height = 0.2\textheight,page=1]{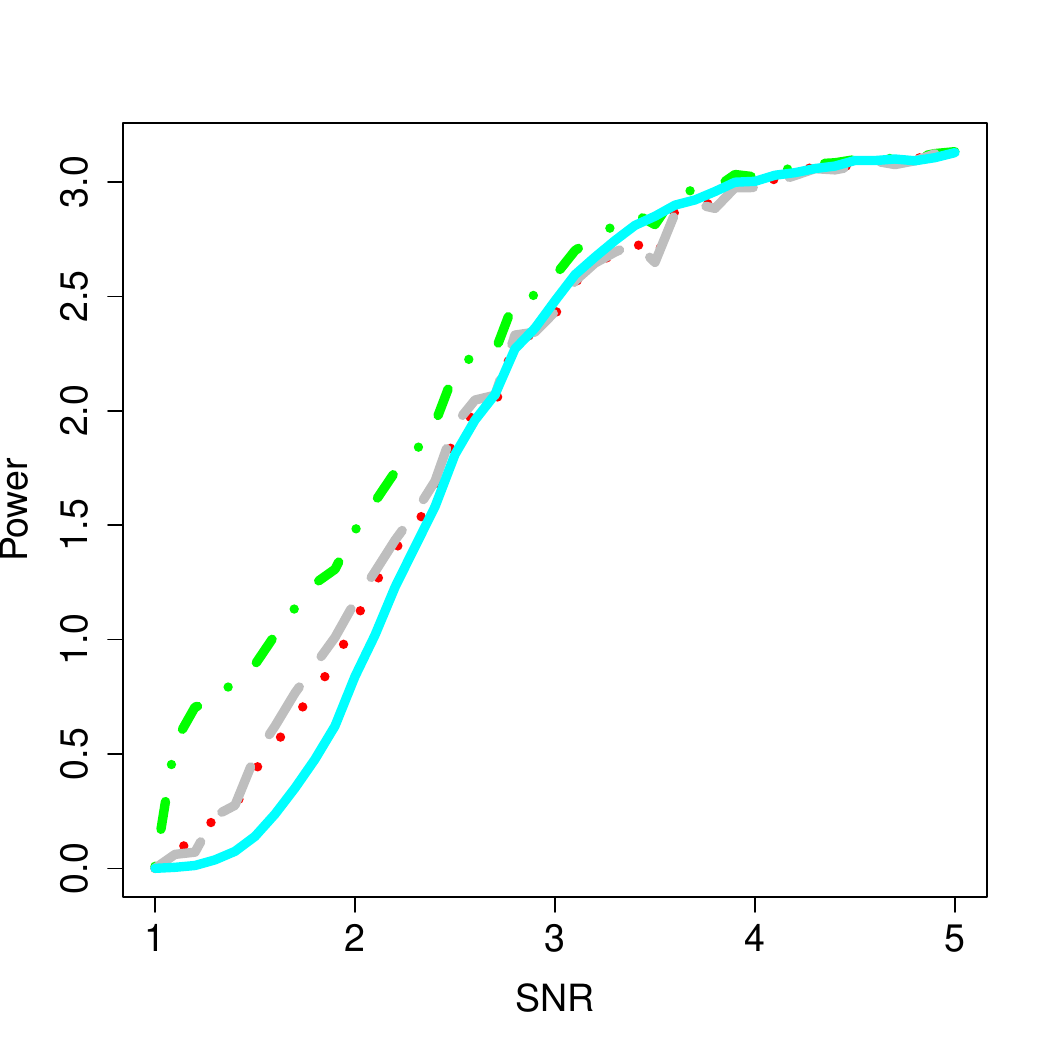} &
\hspace{-5mm}\includegraphics[width=0.34\textwidth,  height = 0.2\textheight,page=5]{ClassCondNonnull2.pdf}\vspace{-5mm}\\
\hspace{-5mm}    \includegraphics[width=0.34\textwidth, height = 0.2\textheight, page=2]{ClassCondMinP1.pdf} &
\hspace{-5mm}\includegraphics[width=0.34\textwidth,  height = 0.2\textheight,page=1]{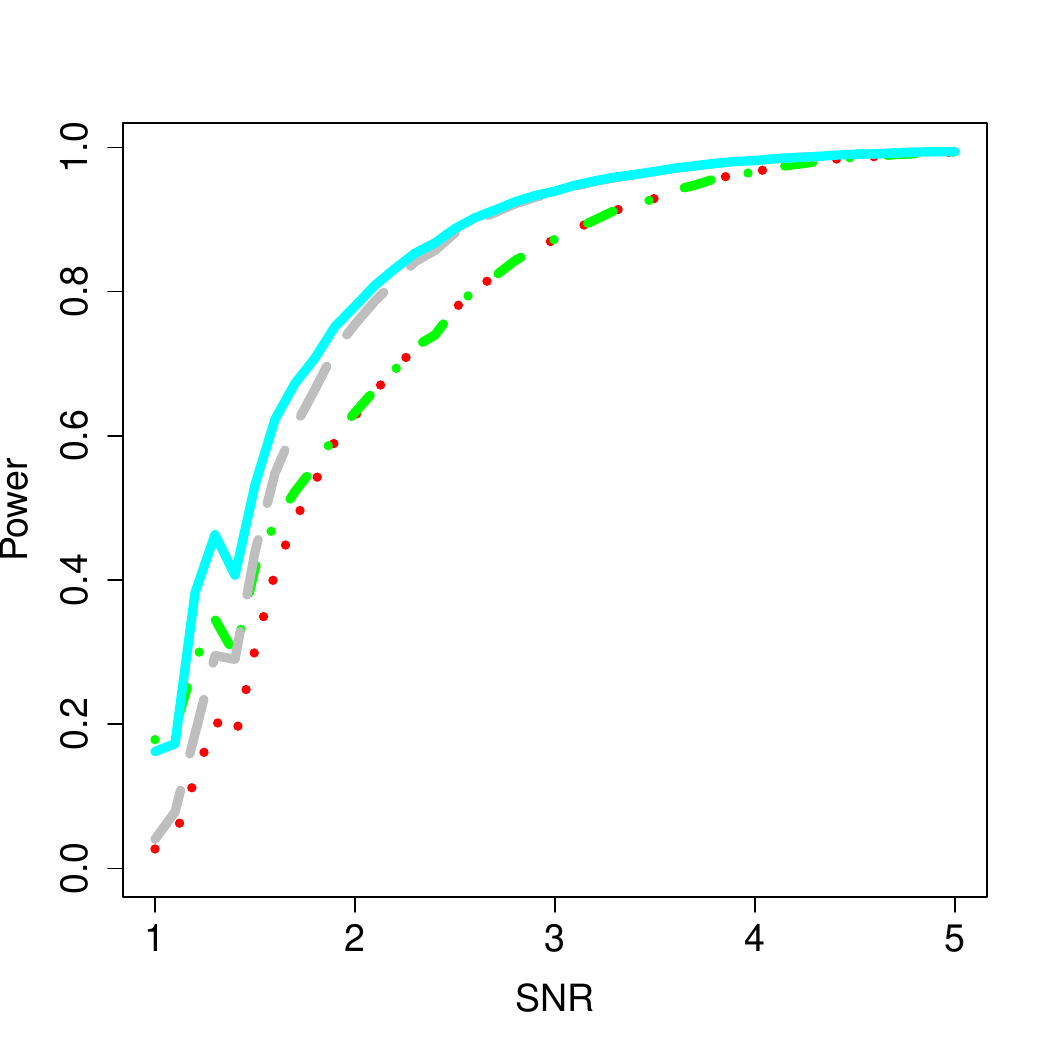} &
\hspace{-5mm}\includegraphics[width=0.34\textwidth,  height = 0.2\textheight,page=5]{ClassCondMinP1.pdf}\vspace{-5mm}\\
\hspace{-5mm}\includegraphics[width=0.34\textwidth, height = 0.2\textheight, page=2]{ClassCondMinP2.pdf}&
\hspace{-5mm}\includegraphics[width=0.34\textwidth,  height = 0.2\textheight,page=1]{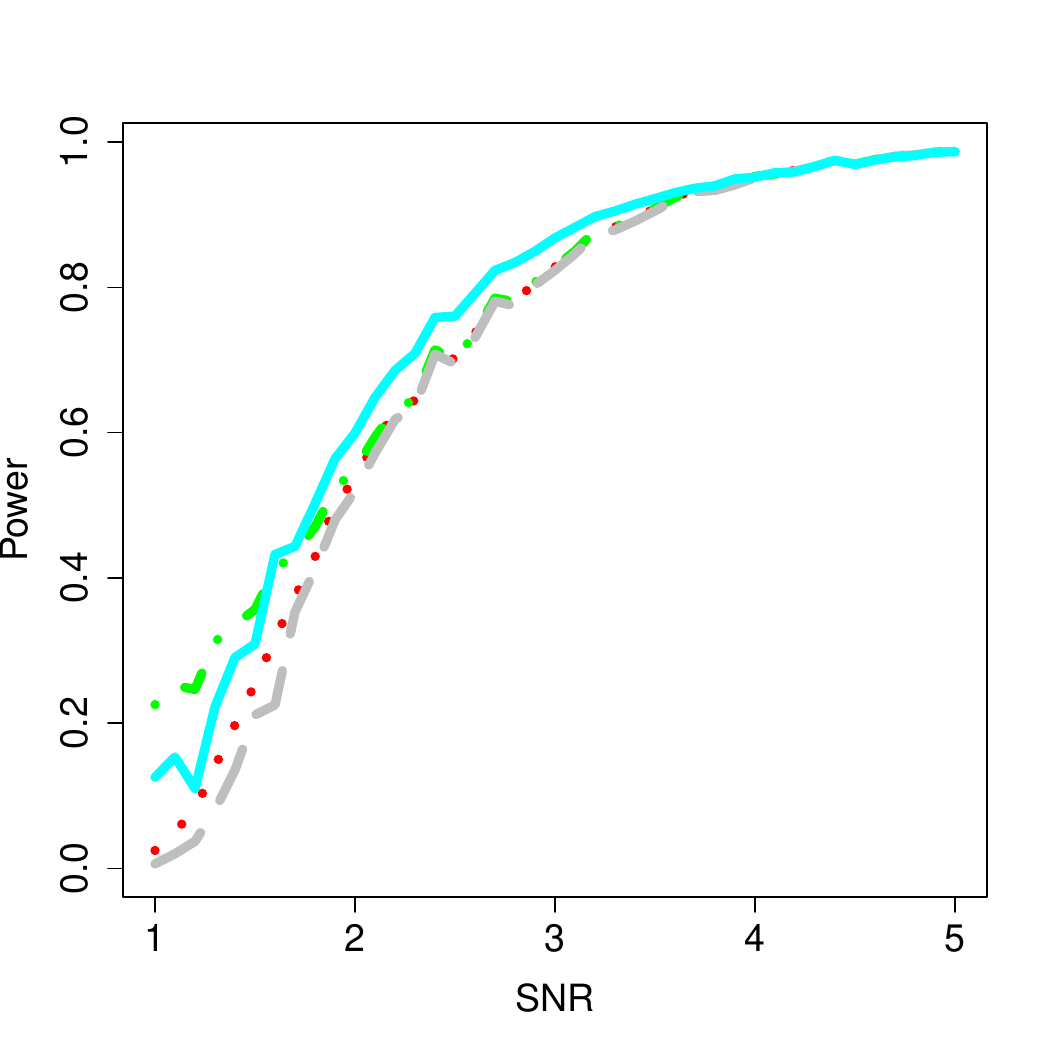} &
\hspace{-5mm}\includegraphics[width=0.34\textwidth,  height = 0.2\textheight,page=5]{ClassCondMinP2.pdf}\vspace{-5mm}\\
  \end{tabular}   
\end{center}

\caption{\label{fig-nullselection-classification-adaptive}  Selecting informative prediction sets in the class-conditional setting.  FCR (left column), resolution-adjusted power  (middle column), and the expected fraction of covering prediction sets  (right column) versus SNR in a classification setting where prediction sets excluding  a null class are of interest (top two rows) and when prediction sets excluding a trivial class are of interest (bottom two rows).  The class probabilities in the calibration sample are 0.15, 0.1, and 0.75; in the test sample, we have a small label shift (rows 1 and 3) and a large label shift (rows 2 and 4). The number of data generations was 2000, and $n=m = 500$. } 
\end{figure}

\section{Proofs}

\subsection{Proof of Theorem~\ref{thm-gen-basic}} \label{sec:thm-gen-basic}

The proof is straightforward from the theory developed in \S~\ref{sec:proofFCRcontrol}. Namely, we apply Theorem~\ref{th:generalBY} by checking the two required assumptions: Assumption~\ref{as:generalpvalues} holds for the two considered $p$-value collections (\S~\ref{sec:as:generalpvalues}); Assumption~\ref{as:nonincreasingSelect} holds for the considered $\BH(\mathbf{q})$ selection (\S~\ref{sec:as:nonincreasingSelect}). \add{The result thus follows from the fact that by Proposition~\ref{prop:inforuleok} we have 
$$\mathcal{R}_{\alpha}=\big(\mathcal{C}^{\alpha (1\vee s_i^{\min})/m}_{n+i}(\mathbf{p})\big)_{i\in\mathcal{S}}=\big(\mathcal{C}^{\alpha |\BH(\mathbf{q})|/m}_{n+i}(\mathbf{p})\big)_{i\in\BH(\mathbf{q})}=\RinfoSP(\mathbf{p}).
$$}

 \subsection{Proof of Theorem~\ref{thm-iid-withpreprocess}}\label{proof:thm-iid-withpreprocess}

The proof is a consequence of Lemma~\ref{lemselective} applied with the FCR criterion: condition (i) in Lemma~\ref{lemselective} is satisfied from Theorem~\ref{thm-gen-basic} (which is true more broadly in the case of exchangeable samples, see Theorem~\ref{th:generalBY} and Proposition~\ref{prop:iidpvalues}); condition (ii) in Lemma~\ref{lemselective} follows from the assumed permutation preserving property of $\mathcal S^{(0)}$. Hence, the conclusion of Lemma~\ref{lemselective} applies which gives the FCR control of \texttt{InfoSCOP}.

Finally, the FDR control is a consequence of the FCR control by applying Lemma~\ref{lem:FDRsmallerFCR}.

\subsection{Proof of Proposition \ref{mainthm0caladaptive}}\label{proofmainthm0caladaptive}

Let us first prove (i) by considering the class-conditional model. 
We follow the proof of Theorem~\ref{th:generalBY} (see \S~\ref{proof:th:generalBY}) and we use that the $p$-value collection $\pcondbf$ satisfies Assumption~\ref{as:generalpvalues} (see Proposition~\ref{prop:condpvalues}), where the probability in the super-uniform property \eqref{equ:superunifgen} holds in the class-conditional model with $\Wcond_i$, $i\in \range{m}$ given by \eqref{equWcond}. We also use that the informative selection rule $\mathcal{S}(\cdot)$ satisfies Assumption~\ref{as:nonincreasingSelect} (see Proposition~\ref{prop:inforuleok}). Hence, following the same approach as in  \S~\ref{proof:th:generalBY}, and denoting $\pcondbf_{{\tiny \mbox{cal}}}:=\pcondbf_{{\tiny \mbox{adapt-cal}}}$ and $\pcondbf_{-i,{\tiny \mbox{cal}}}:=(\pcond{y}{j,{\tiny \mbox{cal}}})_{j\neq i,y\in \ZZ}$, we obtain

\begin{align*}
\FCR(\RinfoSP(\pcondbf_{{\tiny \mbox{cal}}}),P_{X\mid Y},Y)
&\leq \sum_{i\in \range{m}}  \E_{P_{X\mid Y}}\left(\E\left[\frac{ \ind{  \pcond{Y_{n+i}}{i,{\tiny\mbox{cal}}}\leq  \alpha s^{\min}_i(\pcondbf_{-i,{\tiny\mbox{cal}}})/m}}{ s^{\min}_i(\pcondbf_{-i,{\tiny\mbox{cal}}})}\:\middle|\: \Wcond_i\right]\right)\\
&=\sum_{i\in \range{m}}  \E_{P_{X\mid Y}}\left(\E\left[\frac{ \ind{  \pcond{Y_{n+i}}{i}\leq  \alpha s^{\min}_i(\pcondbf_{-i,{\tiny\mbox{cal}}})w_{Y_{n+i}}/(m\hat{\pi}^{{\tiny \mbox{cal}}}_{Y_{n+i}})}}{ s^{\min}_i(\pcondbf_{-i,{\tiny\mbox{cal}}})}\:\middle|\: \Wcond_i\right]\right).
\end{align*}
Using now the super-uniform property \eqref{equ:superunifgen}, the fact that $\pcondbf_{-i,{\tiny \mbox{cal}}}=(\frac{\hat{\pi}_k}{w_k} \pcond{k}{j})_{k\neq i,y\in \ZZ}=\Phi(Y,\pcondbf_{-i})$, with $\Phi(Y,\cdot)$ coordinate-wise nondecreasing, and Assumption~\ref{as:generalpvalues} (i) entail
$$
s^{\min}_i(\pcondbf_{-i,{\tiny\mbox{cal}}})=s^{\min}_i(\pcondbf_{-i} )= s^{\min}_i(\Phi(Y,\Psi_i(p^{(Y_{n+i})}_{i},W_i))).
$$
Hence, $s^{\min}_i(\pcondbf_{-i,{\tiny\mbox{cal}}})$ can be written as a function $g(p^{(Y_{n+i})}_{i})$, with  $g:u\mapsto s^{\min}_i(\Phi(Y,\Psi_i(u,W_i)))$ nonincreasing and only depending on $Y$ and $W_i$. 
Applying Lemma~\ref{lem:BR} for $c=w_{Y_{n+i}}/(m\hat{\pi}^{{\tiny \mbox{cal}}}_{Y_{n+i}})$, we obtain
\begin{equation*}
\E_{P_{X\mid Y}}\left[\frac{ \ind{  \pcond{Y_{n+i}}{i}\leq  \alpha s^{\min}_i(\pcondbf_{-i,{\tiny\mbox{cal}}})w_{Y_{n+i}}/(m\hat{\pi}^{{\tiny \mbox{cal}}}_{Y_{n+i}})}}{ s^{\min}_i(\pcondbf_{-i,{\tiny\mbox{cal}}})}\:\middle|\: \Wcond_i\right]\leq \alpha\frac{w_{Y_{n+i}}}{m\hat{\pi}^{{\tiny \mbox{cal}}}_{Y_{n+i}}}.
\end{equation*}

As a consequence, we derive 
\begin{align*}
\FCR(\RinfoSP(\pcondbf_{{\tiny \mbox{cal}}}),P_{X\mid Y},Y)
&\leq \alpha\sum_{k\in\range{K}}w_k\left (\frac{\sum_{i\in\range{m}}\ind{Y_{n+i}=k}}{m}\frac{1}{\hat{\pi}^{{\tiny \mbox{cal}}}_k}\right ),
\end{align*}
which proves (i).
We deduce (ii) by a simple integration:
\begin{align*}
\FCR(\RinfoSP(\pcondbf_{{\tiny \mbox{cal}}}),P_{X,Y})&\leq \alpha\sum_{k\in\range{K}}w_k\E\left (\frac{\sum_{i\in\range{m}}\ind{Y_{n+i}=k}}{1+\sum_{j\in\range{m}}\ind{Y_j=k}}\frac{n+1}{m}\right )\\
&=\alpha \sum_{k\in\range{K}}w_k(n+1)\P(Y_1=k)\E\left (\frac{1}{1+\sum_{j\in\range{m}}\ind{Y_j=k}}\right )\\
&\leq \alpha \sum_{k\in\range{K}}w_k(n+1)\P(Y_1=k)\frac{1}{(n+1)\P(Y_1=k)}
=\alpha\sum_{k\in\range{K}}w_k=\alpha,
\end{align*}
by using Lemma~\ref{lem:BY} for the last inequality. 

\subsection{Proof of Proposition~\ref{prop:InfoSPClassif}}\label{proof:prop:InfoSPClassif}

First define $R_j$ the rank of $S_{Y_j}(X_j)$ in $\{ S_{k}(X_j),k\in \range{K}\}$ (ordered in increasing order) for  $j\in \range{n+m}$, and consider the slightly smaller conformal $p$-values
\begin{equation}\label{standardpvalue2}
\pzhao{k}{i}= \frac{1}{n+1}\left(1+\sum_{j=1}^n \ind{R_j>1}\ind{S_{Y_j}(X_j)\geq S_{k}(X_{n+i})} \right)\leq \pfull{k}{i},\:\: i\in \range{m}, k\in \range{K},
\end{equation}
which means that the calibration is only made with examples having a label not minimizing the score function. The rationale behind using this $p$-value rather than $\pfull{k}{i}$ is that, due to the post-processing, the elements $X_{n+i}$ of the test sample cannot produce an error provided that $R_{n+i}=1$ so that we can restrict the test sample to those with $R_{n+i}>1$ when computing the FCR.

 Assume $K=2$. We first prove that $\RinfoSP(\pfullbf)=\RinfoSP(\pzhaobf)$. 
It is enough to prove that the adjusted $p$-values $q_i$ obtained from $\pfullbf$ and $\pzhaobf$ are the same for this non-trivial selection (since for the selected $i$, the procedure always chooses $\arg\min_{k\in \range{K}}S_k(X_{n+i})$ due to the post-processing). Letting $S_{\min}(x)=\min_{k\in \range{K}}S_k(x)$, $S_{\max}(x)=\max_{k\in \range{K}}S_k(x)$, this comes from
\begin{align*}
\min_{k\in \range{K}}\{ \pzhao{k}{i}\}&= \frac{1}{n+1}\left(1+\sum_{j=1}^n  \ind{R_j>1}  \ind{S_{Y_j}(X_j)\geq S_{\max}(X_{n+i})} \right)\\
&= \frac{1}{n+1}\left(1+\sum_{j=1}^n   \ind{S_{Y_j}(X_j)\geq S_{\max}(X_{n+i})} \right)=\min_{k\in \range{K}}\{ \pfull{k}{i}\},
\end{align*}
where the second equality holds because $S_{Y_j}(X_j)\geq S_{\max}(X_{n+i})$ is impossible for $R_j=1$, that is, when $S_{Y_j}(X_j)$ is a minimum ($S_{\min}(X_j)< 1/2 < S_{\max}(X_j)$ almost surely by the assumptions on the score function).  

Now prove Proposition~\ref{prop:InfoSPClassif}, by showing the equality for $\RinfoSP(\pzhaobf)$ and by carefully modifying the proof of Theorem~\ref{th:generalBY} (\S~\ref{proof:th:generalBY}) in order to get only equalities, by using that we consider the case of the non-trivial selection, that is, $\mathcal{S}=\BH(\mathbf{q})$, with $q_i=\min_{k\in \range{K}} \pzhao{k}{i}$. 
By definition of the FCP \eqref{equFSP}, we have (remember also that the prediction set always includes $\arg\min_{k\in \range{K}}S_k(X_{n+i})$ due to the post-processing)
\begin{align*}
\FCP(\RinfoSP(\pzhaobf),Y)&=\frac{\sum_{i\in \mathcal{S}} \ind{Y_{n+i} \notin \mathcal{C}_{n+i}^{\alpha |\mathcal{S}|/m}(\pzhaobf)} }{1\vee |\mathcal{S}|}\\
&= \sum_{i\in \range{m}} \ind{R_{n+i}>1} \frac{ \ind{ i \in \mathcal{S}, \pzhao{Y_{n+i}}{i}\leq  \alpha |\mathcal{S}|/m}}{1\vee |\mathcal{S}|},
\end{align*}
because no error can occur when $R_{n+i}=1$. 
Now, since $\mathcal{S}=\BH(\mathbf{q})$ and $q_i\leq \pzhao{Y_{n+i}}{i}$, $\ind{ i \in \mathcal{S}, \pzhao{Y_{n+i}}{i}\leq  \alpha |\mathcal{S}|/m}=\ind{\pzhao{Y_{n+i}}{i}\leq  \alpha |\mathcal{S}|/m}$. This entails
\begin{align*}
\FCP(\RinfoSP(\pzhaobf),Y)= \sum_{i\in \range{m}} \ind{R_{n+i}>1} \frac{ \ind{ \pzhao{Y_{n+i}}{i}\leq  \alpha |\mathcal{S}|/m}}{1\vee |\mathcal{S}|}.\end{align*}
Let $\xi=(\ind{R_j>1})_{j\in \range{n+m}}$ and now prove
\begin{equation}
    \label{tobeproved}
    \E[\FCP(\RinfoSP(\pzhaobf),Y)\:|\: \xi] = \sum_{i\in \range{m}} \ind{R_{n+i}>1} \E \Big[ \frac{ \lfloor (n+1)\alpha K_i /m \rfloor/(n'+1) }{K_i}\:\Big|\: \xi\Big],
\end{equation}
with $n'=\sum_{j\in \range{n}} \ind{R_{j}>1}$ for some random variables $K_i\geq 1$, $i\in \range{m}$. 
This implies the result because the last display is at most 
$$
\alpha   \frac{(n+1)\sum_{i\in \range{m}} \ind{R_{n+i}>1}}{m(\sum_{j\in \range{n}} \ind{R_{j}>1}+1)}
$$
(with equality if $(n+1)\alpha /m$ is an integer). By Lemma~\ref{lem:BY}, the expectation of the latter is equal to $\alpha (1-(1-\P(R_1>1))^{n+1})$.

Let us now prove \eqref{tobeproved}. For this, fix $i\in \range{m}$ with $R_{n+i}>1$ and note that for all $j\in \range{m}$, $j\neq i$, 
\begin{align*}
    q_j &= \frac{1}{n+1}\left(1+\sum_{k=1}^n  \ind{R_k>1}  \ind{S_{Y_k}(X_k)\geq S_{\max}(X_{n+j})} \right)\\
    &= \frac{1}{n+1}\Big(\ind{S_{Y_{n+i}}(X_{n+i})< S_{\max}(X_{n+j})}+\sum_{s\in A_i} \ind{s\geq S_{\max}(X_{n+j}) } \Big)\\
    &= \frac{1}{n+1}\Big(\ind{S_{\max}(X_{n+i})< S_{\max}(X_{n+j})}+\sum_{s\in A_i} \ind{s\geq S_{\max}(X_{n+j}) } \Big)\\
    &\geq \frac{1}{n+1}\sum_{s\in A_i} \ind{s\geq S_{\max}(X_{n+j}) } =:q'_j,
\end{align*}
by letting $A_i=\{S_{Y_k}(X_j),j\in \range{n}: R_k>1\}\cup \{S_{Y_{n+i}}(X_{n+i})\}$ (all distinct almost surely). The third equality above is true because $K=2$ and $R_{n+i}>1$ and thus $S_{Y_{n+i}}(X_{n+i})=S_{\max}(X_{n+i})$. 
We apply now Lemma~\ref{BHsmallerp} because $\mathbf{q}=(q_j,1\leq j\leq m)$ and $\mathbf{q}'=(q'_j,1\leq j\leq m)$ (defined as above and with $q_i'=1/(n+1)$) satisfy \eqref{propppprime}. Indeed, if $q_j> q_i$ for $j\neq i,$ then $S_{\max}(X_{n+j})\leq S_{\max}(X_{n+i})$ and $q_j=q'_j$. Hence, we obtain, by letting $K_i=|\BH(\mathbf{q'})|$ (note that $\mathbf{q'}$ depends on $i$) 
$$
\{q_i\leq \alpha |\BH(\mathbf{q})|/m \}=\{ q_i\leq \alpha K_i/m \}\subset \{ |\BH(\mathbf{q})|=K_i\}.
$$
 In addition, $\mathbf{q}'=(q'_j,1\leq j\leq m)$ is a vector measurable w.r.t. the variable
$
W_i=(A_i,(S_{\max}(X_{n+j}))_{ j\in\range{m}\backslash\{i\}}).
$ 
Now, we use that by exchangeability of the elements of $A_i$ conditionally on $\xi$, $W_i$, and $R_{n+i}>1)$ (with no ties), 
$$
\P(\pzhao{Y_{n+i}}{i}\leq t\:|\: W_i, \xi, R_{n+i}>1) = \frac{\lfloor (n+1)t \rfloor}{n'+1}.
$$
Applying this with $t=\alpha K_i/m$ entails \eqref{tobeproved}.

\section{{Computing $\mathcal{I}$-adjusted $p$-values}}\label{sec:qicomput}

\add{We gather here details for computing the $q_i$'s \eqref{qformula} for different informative set collection $\mathcal{I}$. In the regression case, recall that $p_i^{(y)}$ necessarily corresponds to the full-calibrated $p$-values \eqref{standardpvalue}.}

\paragraph{{Regression and excluding $\mathcal{Y}_0\subset \mathcal{Y}$.}}

\add{
The informative set collection is given by
$\mathcal{I}=\{\mbox{$I$ interval of } \R\::\: I\cap \mathcal{Y}_0=\emptyset \}$.
\begin{itemize}
    \item The general expression is $q_i=\sup_{y\in \mathcal{Y}_0}p_i^{(y)}$. Indeed, by definition, $\mathcal{C}^{\alpha}_{n+i}(\mathbf{p})\in \mathcal{I}$ iff  $\mathcal{C}^{\alpha}_{n+i}(\mathbf{p})\cap \mathcal{Y}_0=\emptyset$ iff $\forall y\in \mathcal{Y}_0, p_i^{(y)}\leq \alpha$, which means $\sup_{y\in \mathcal{Y}_0}p_i^{(y)}\leq \alpha$. \\
   \item  Another expression for the latter is given as follows:
    \begin{align*}
        q_i&= \frac{1}{n+1}\sup_{y\in \mathcal{Y}_0}\Big(1+\sum_{j=1}^n \ind{S_{Y_j}(X_j)\geq S_{y}(X_{n+i})} \Big)\\
        &= \frac{1}{n+1}\Big(1+\sum_{j=1}^n \ind{S_{Y_j}(X_j)\geq \inf_{y\in \mathcal{Y}_0} S_{y}(X_{n+i})} \Big).
   \end{align*}
    \item If $\mathcal{Y}_0=[a,b]$        and $S_y(x)=|y-\mu(x)|/\sigma(x)$ the above expression gives 
    $$
     q_i=p_i^{(a)}\ind{\mu(X_{n+i})<a} + p_i^{(b)}\ind{\mu(X_{n+i})>b}+ \ind{a \leq \mu(X_{n+i})\leq b}. 
    $$
    This is because $\inf_{y\in [a,b]} S_y(x)$ is $0$ if $\mu(x)\in [a,b]$, $S_a(x)$ if $\mu(x)<a$, and $S_b(x)$ if $\mu(x)>b$.
    \item If $\mathcal{Y}_0=[a,b]$        and $S_y(x)=\max(q_{\beta_0}(x)-y,y-q_{\beta_1}(x))$, we have 
    $$
q_i=p_i^{(a)}\ind{\mu(X_{n+i})<a} + p_i^{(b)}\ind{\mu(X_{n+i})>b}+ p_i^{(\mu(X_{n+i}))} \ind{a \leq \mu(X_{n+i})\leq b},    $$    where $\mu(x)=(q_{\beta_0}(x)+q_{\beta_1}(x))/2$.
This is because $\inf_{y\in [a,b]} S_y(x)$ is $S_{\mu(x)}(x)$ if $\mu(x)\in [a,b]$, $S_a(x)$ if $\mu(x)<a$, and $S_b(x)$ if $\mu(x)>b$.
\end{itemize}
}
\paragraph{\add{Regression and length-restriction.}}

\add{
The informative set collection is given by
$\mathcal{I}=\{\mbox{$I$ interval of } \R\::\: |I|\leq 2\lambda_0\}$, $\lambda_0>0$.
\begin{itemize}
    \item By definition, $\mathcal{C}^{\alpha}_{n+i}(\mathbf{p})\in \mathcal{I}$ iff  $|\{y\in \R\::\: p_i^{(y)}> \alpha\}|\leq 2\lambda_0$, that is, $|\{y\in \R\::\: S_y(X_{n+i})\leq S_{(\lceil (1-\alpha)(n+1)\rceil)} \}|\leq 2\lambda_0$. We then obtain the expression of $q_i$ by inverting the condition  $S_y(X_{n+i})\leq S_{(\lceil (1-\alpha)(n+1)\rceil)}$ with respect to $y$. This expression depends on the considered score function.
    \item For $S_y(x)=|y-\mu(x)|/\sigma(x)$, we have $S_y(X_{n+i})\leq S_{(\lceil (1-\alpha)(n+1)\rceil)}$ iff $|y-\mu(X_{n+i})|\leq \sigma(X_{n+i}) S_{(\lceil (1-\alpha)(n+1)\rceil)}$ which gives $q_i\leq \alpha$ iff $\sigma(X_{n+i}) S_{(\lceil (1-\alpha)(n+1)\rceil)}\leq \lambda_0$ (see Remark~\ref{rem:caldepth}). Inverting the latter with respect to $\alpha$ gives
    $$
    q_i= (n+1)^{-1}\big(1+\sum_{j=1}^n \ind{S_{Y_j}(X_j)> \lambda_0/\sigma(X_{n+i})}\big).
    $$
    \item For $S_y(x)=\max(q_{\beta_0}(x)-y,y-q_{\beta_1}(x))$, we have $S_y(X_{n+i})\leq S_{(\lceil (1-\alpha)(n+1)\rceil)}$ iff $q_{\beta_0}(x)- S_{(\lceil (1-\alpha)(n+1)\rceil)}\leq y\leq S_{(\lceil (1-\alpha)(n+1)\rceil)} + q_{\beta_1}(x)$. Hence, $q_i\leq \alpha$ iff $S_{(\lceil (1-\alpha)(n+1)\rceil)}+(q_{\beta_1}(X_{n+i})-q_{\beta_0}(X_{n+i}) )/2\leq \lambda_0$, which gives
    $$
    q_i= (n+1)^{-1}\big(1+\sum_{j=1}^n \ind{S_{Y_j}(X_j)> \lambda_0-(q_{\beta_1}(X_{n+i})-q_{\beta_0}(X_{n+i}) )/2}\big).
    $$
\end{itemize}
}

\paragraph{\add{Classification.}}

\add{The following informative collections are considered:
\begin{itemize}
    \item For excluding $\mathcal{Y}_0\subset \range{K}$,  
$\mathcal{I}=\{C\subset \range{K}\::\: C\cap \mathcal{Y}_0=\emptyset \}$. The general expression is $q_i=\max_{y\in \mathcal{Y}_0}p_i^{(y)}$, with the same reasoning as in the regression case. 
    \item For non-trivial classification,  $\mathcal{I}=\{C\subset \range{K}\::\: |C|\leq K-1\}$ and $q_i=\min_{y\in \range{K}} p^{(y)}_i$. Indeed, $\mathcal{C}^{\alpha}_{n+i}(\mathbf{p})\in \mathcal{I}$ iff  
    $|\{y\in \range{K}\::\: p_i^{(y)}>\alpha\}|\leq K-1$ iff there exists $y\in \range{K}$ such that $p_i^{(y)}\leq \alpha$, that is, $\min_{y\in \range{K}}p_i^{(y)}\leq \alpha$.
    \item For at most $k_0$-sized classification, $\mathcal{I}=\{C\subset \range{K}\::\: |C|\leq k_0\}$. A similar reasoning leads to $q_i$ being the $(K-k_0)$-th smallest element in the set $\{p^{(y)}_i,y\in \range{K}\}$. 
\end{itemize}
}

\section{\add{Alternative proof of Theorem~\ref{thm-gen-basic} using PRDS}}\label{sec:alternativeproof}

\add{Proving Theorem~\ref{thm-gen-basic} does not go through the PRDS property \citep{BY2001} but rather uses the monotonicity introduced in Assumption~\ref{as:generalpvalues},  see Section~\ref{sec:thm-gen-basic}. We propose in this section an alternative proof that explicitly relies on the PRDS property, and which shares similarities with the techniques developed in \cite{Jin2023selection}.}

\subsection{\add{Identifying a PRDS property}}

\add{Recall that a random variable family $(T_i)_{i\in \range{m}}$ (valued in $[0,1]$) is said to be PRDS on $\cH\subset \range{m}$ if for any measurable nondecreasing\footnote{A subset $D\subset [0,1]^m$ is said to be nondecreasing if for all $(x_i)_{i\in \range{m}}\in D$ and all $(y_i)_{i\in \range{m}}\in [0,1]^m$, $(\forall i\in \range{m}, x_i\leq y_i)$ implies $(y_i)_{i\in \range{m}}\in D$. } set $D\subset [0,1]^m$, for all $i\in \cH$, the function $u\in [0,1]\mapsto \P((T_j)_{j\in \range{m}}\in D\:|\: T_i=u)$ is nondecreasing.}

\add{
The following results can be seen as an extension of Lemma~5 in  \cite{Jin2023selection} to our $\mathcal{I}$-informative context.
\begin{theorem}\label{th:PRDS}
Consider the setting and assumptions of Theorem~\ref{thm-gen-basic} and $(q_i)_{i\in \range{m}}$ the $\mathcal{I}$-adjusted $p$-value vector defined by \eqref{qformula}. Then the family $(q_1,\dots, q_{i-1},p_i^{(Y_{n+i})}, q_{i+1},\dots,q_m)$ is PRDS on $\{i\}$  for any $i\in \range{m}$.
\end{theorem}
}

\add{
\begin{proof}   
Let us denote $\mathbf{p}_{-i}=(p^{(y)}_{j})_{j\neq i,y\in \ZZ}$ and $\mathbf{q}_{-i}:=(q_j)_{j\neq i}$. By Section~\ref{sec:as:generalpvalues}, Assumption~\ref{as:generalpvalues} is satisfied with a random variable $W_i$ which is independent of $p_i^{(Y_{n+i})}$. 
In particular, $\mathbf{p}_{-i}=\Psi_i(p^{(Y_{n+i})}_{i},W_i)$ where $u\in [0,1] \mapsto \Psi_i(u,W_i)\in \R^{\range{m-1}\times \ZZ}$ is a nondecreasing function.
In addition, by Section~\ref{sec:as:nonincreasingSelect}, $\mathbf{q}_{-i}$ is a coordinate-wise nondecreasing function of $\mathbf{p}_{-i}$. Combining these facts, we obtain that 
$$
\mathbf{q}_{-i}=\phi_i(p^{(Y_{n+i})}_{i},W_i),
$$
where $u\in [0,1] \mapsto \phi_i(u,W_i)\in \R^{\range{m-1}}$ is a nondecreasing function. As a result, for any measurable nondecreasing set $D\subset [0,1]^m$,
  \begin{align*}
  \P((p_i^{(Y_{n+i})},\mathbf{q}_{-i})\in D\:|\: p_i^{(Y_{n+i})}=u) &= \P((p_i^{(Y_{n+i})},\phi_i(p^{(Y_{n+i})}_{i},W_i))\in D\:|\: p_i^{(Y_{n+i})}=u)    \\
  &=\P((u,\phi_i(u,W_i))\in D) ,
  \end{align*}
where we denoted  $(p_i^{(Y_{n+i})},\mathbf{q}_{-i})=(q_1,\dots, q_{i-1},p_i^{(Y_{n+i})}, q_{i+1},\dots,q_m)$ with some notation abuse and where we  used independence between $W_i$ and $p_i^{(Y_{n+i})}$. Since for a fixed $W_i$, $u\mapsto \ind{(u,\phi_i(u,W_i))\in D}$ is nondecreasing (since $D$ is nondecreasing), the last display is nondecreasing in $u$.
\end{proof}
}

\subsection{\add{Alternative proof of Theorem~\ref{thm-gen-basic}}}
\add{
By Definition~\ref{def:basic}, we have
\begin{align*}
\FCP(\RinfoSP(\mathbf{p}),Y)
&=\frac{\sum_{i\in \BH(\mathbf{q})} \ind{Y_{n+i} \notin \mathcal{C}_{n+i}^{\alpha |\BH(\mathbf{q})|/m}(\mathbf{p})} }{1\vee |\BH(\mathbf{q})|}\\
&=\frac{\sum_{i\in \BH(\mathbf{q})} \ind{Y_{n+i} \notin \mathcal{C}_{n+i}^{\alpha |\BH(\mathbf{q}^{0,i})|/m}(\mathbf{p})} }{1\vee |\BH(\mathbf{q}^{0,i})|}\\
&\leq \sum_{i\in \range{m}}  \frac{ \ind{  p_i^{(Y_{n+i})}\leq  \alpha |\BH(\mathbf{q}^{0,i})|/m}}{1\vee |\BH(\mathbf{q}^{0,i})|},
\end{align*}
where the second equality follows  from $\BH(\mathbf{q}^{0,i})=\BH(\mathbf{q})$ when $i \in \BH(\mathbf{q})$, by denoting $\mathbf{q}^{0,i}$ the vector $\mathbf{q}$ where the $i$-th coordinate has been replaced by $0$, see for instance \cite{ramdas2019unified}. This entails
\begin{align*}
\E\left (\FCP(\RinfoSP(\mathbf{p}),Y)\right )
&\leq \sum_{i\in \range{m}}  \E\left( \frac{ \ind{  p_i^{(Y_{n+i})}\leq  \alpha |\BH(\mathbf{q}^{0,i})|/m}}{1\vee |\BH(\mathbf{q}^{0,i})|}\right).
\end{align*}
Now, by Theorem~\ref{th:PRDS}, for any $r>0$, the function
$
u\mapsto \P(|\BH(\mathbf{q}^{0,i})|< r\:|\: p_i^{(Y_{n+i})}= u)$
is nondecreasing and thus so is $
u\mapsto \P(|\BH(\mathbf{q}^{0,i})|< r\:|\: p_i^{(Y_{n+i})}\leq  u)
$. Applying Lemma~\ref{lem:BR2} gives the result.
}

\section{Auxiliary results}

\begin{lemma}[Lemma~3.2 (i) in \cite{BR2008}]\label{lem:BR}
    Let $g:[0,1]\rightarrow (0,\infty)$ be a nonincreasing function and $U$ be a random variable which is super-uniform, that is, $\forall u\in [0,1],$ $\P(U\leq u)\leq u$. Then, for any $c>0$, we have
    $$
\E\left[\frac{\ind{U\leq c g(U)}}{g(U)}\right]\leq c.
    $$
\end{lemma}

\add{
\begin{lemma}[Lemma~3.2 (ii) in \cite{BR2008}]\label{lem:BR2}
    Let $U,V$ be nonnegative random variables such that $U$ is super-uniform, that is, $\forall u\in [0,1],$ $\P(U\leq u)\leq u$ and $V$ is such that for any $r>0$, the function $u\mapsto \P(V<r\:|\: U\leq u)$ is nondecreasing. Then, for any $c>0$, we have
    $$
\E\left[\frac{\ind{U\leq c V}}{V}\right]\leq c.
    $$
\end{lemma}
}

\begin{lemma}[Lemma 1 of \cite{BKY2006}]\label{lem:BY}
    If $T$ is a Binomial variable with parameter $N-1\geq 0$ and $t\in(0,1]$, we have
  $$
  \E[1/(T+1)]=(1-(1-t)^{N})/(Nt)\leq 1/(Nt). 
  $$  
\end{lemma}

\begin{lemma}[Lemma D.6 of \cite{marandon2024adaptive}]\label{BHsmallerp} 
Write $\wh{\ell}=\wh{\ell}(\mathbf{q})=|\BH(\mathbf{q})|$ for the number of rejections of $\BH(\mathbf{q})$ \eqref{equBH}. Fix any $i\in \{1,\dots,m\}$ and consider two collections $\mathbf{q}=(q_j,1\leq j\leq m)$ and $\mathbf{q}'=(q'_j,1\leq j\leq m)$ which satisfy almost surely that
\begin{align}\label{propppprime}
\forall j\in \{1,\dots,m\}, \left\{\begin{array}{cc} q'_j\leq q_j& \mbox{ if } q_j\leq q_i;\\q'_j = q_j& \mbox{ if } q_j> q_i.\end{array}\right.  
\end{align}
Then 
$
\{q_i\leq \alpha \wh{\ell}(\mathbf{q})/m \}=\{ q_i\leq \alpha \wh{\ell}(\mathbf{q}')/m \}\subset \{ \wh{\ell}(\mathbf{q})=\wh{\ell}(\mathbf{q}')\}.
$
\end{lemma}

The calibration splitting trick can be seen as a way to obtain statistical guarantees in conformal inference when  making a data-driven choice regarding the inference
. In \cite{marandon2024adaptive,gazin2023transductive}, the choices are about  adaptive score functions. 
The next lemma presents this trick in the case that the choice is about which examples to select for (potentially more efficient and powerful) further inference
\footnote{This idea was sketched in version 3 of the arXiv version \cite{bao2024selective} of the work \cite{bao2024selective}}. 
\begin{lemma}[Calibration splitting trick for selective conformal prediction sets]\label{lemselective}
Assume that for any training sample $\dtrain$, calibration sample $\dcal=((X_j,Y_j))_{j\in \range{\ncal}}$ and test sample $\dtest=(X_j,Y_j)_{j\in \range{\ncal+1,\ncal+\ntest}}$, such that $(\dtrain,\dcal,\dtest)$ has some distribution $Q$, the  procedure (that is, a prediction set collection on a selection) and  initial selection rule are as follows:
\begin{itemize}
    \item[(i)] the procedure $\mathcal{R}=(\mathcal{C}_{\ncal+i})_{i\in \mathcal{S}}$, $\mathcal{C}_{\ncal+i}\subset \ZZ$, $\mathcal{S}\subset \range{\ntest}$, with $\mathcal{R}=\mathcal{R}(\dcal,\dtestX; \dtrain)$ built upon $\dtrain,\dcal$ and $\dtestX:=(X_j)_{j\in \range{\ncal+1,\ncal+\ntest}}$ such that it controls a criterion $\E_Q( \mathcal{E}(\mathcal{R},Y))\leq \alpha$ 
    if the entries of $\dcal\cup \dtest$ are exchangeable conditionally on $\dtrain$;
    \item[(ii)] an initial selection rule $\mathcal{S}^{(0)}=\mathcal{S}^{(0)}(\dcal,\dtestX;\dtrain)\subset \range{\ncal+1,\ncal+\ntest}$ built upon $\dtrain,\dcal$ and $\dtestX$ which is permutation preserving in the latter, that is, for any permutation $\sigma$ of $\range{\ncal+1,\ncal+\ntest}$, we have $\sigma(\mathcal{S}^{(0)}(\dcal,\dtestX;\dtrain))=\mathcal{S}^{(0)}(\dcal,\sigma(\dtestX);\dtrain)$.
\end{itemize}
Let us consider samples $(\dtrain,\dcal,\dtest)$ as above such that the entries of $\dcal\cup\dtest$ are exchangeable conditionally on $\dtrain$ and split the calibration set $\dcal$ into two samples $\dcal^{(1)}$ and $\dcal^{(2)}$ of respective sizes $\ncal^{(1)}$ and $\ncal^{(2)}$. Let $\dcal^{(2),X}=(X_j)_{j\in \range{\ncal^{(1)}+1,\ncal}}$.
Then for $\mathcal{S}^{(0)}=\mathcal{S}^{(0)}(\dcal^{(1)},\dcal^{(2),X}\cup \dtestX ; \dtrain)\subset \range{\ncal^{(1)}+1,\ncal+\ntest}$, the procedure $\mathcal{R}^{split}=\mathcal{R}((\dcal^{(2)})_{\mathcal{S}^{(0)}},(\dtestX)_{\mathcal{S}^{(0)}} ; \dtrain)$ is a procedure achieving $\E_Q( \mathcal{E}(\mathcal{R}^{split},Y) \:|\: \mathcal{S}^{(0)})\leq \alpha$ and thus also $\E_Q( \mathcal{E}(\mathcal{R}^{split},Y))\leq \alpha$.
\end{lemma}
For instance, Lemma~\ref{lemselective} can be used for the criterion $\mathcal{E}(\mathcal{R},Y)=\FCP(\mathcal{R},Y)$ in which case it provides FCR control.
{For instance, since classical marginal prediction intervals always satisfy (i) with $S=\range{m}$ and the FCR criterion, Lemma~\ref{lemselective} shows that sample splitting can offer FCR control for any (permutation preserving) selection rule. This is an alternative method to the swapping approach of \cite{jin2024confidence}.}
\begin{proof}
For convenience, let us write $\dcal\cup \dtest= ((X_j,Y_j))_{j\in \range{\ncal+\ntest}}= (Z_j)_{j\in \range{\ncal+\ntest}}$. 
By (i), it is enough to prove that the entries of $(Z_j)_{j\in S}$ are exchangeable conditionally on $\mathcal{S}^{(0)}=S$, $\dcal^{(1)}$ and $\dtrain$, for any possible realisation $S$ of $\mathcal{S}^{(0)}$. 
For any $\sigma$ permutation of $\range{\ncal^{(1)}+1,\ncal+\ntest}$ which only affects the indexes of $S$, we have
\begin{align*}
    &\mathcal{D}( (Z_{\sigma(j)})_{j \in S, j\geq \ncal^{(1)}+1} \:|\: \mathcal{S}^{(0)}=S, \dcal^{(1)},\dtrain)\\
    &=\mathcal{D}( (Z_{\sigma(j)})_{j \in S, j\geq \ncal^{(1)}+1}\:|\: \mathcal{S}(\dcal^{(1)},\dcal^{(2),X}\cup \dtestX ; \dtrain)=S, \dcal^{(1)},\dtrain)\\
    &=\mathcal{D}( (Z_{\sigma(j)})_{j \in S, j\geq \ncal^{(1)}+1} \:|\: \mathcal{S}(\dcal^{(1)},(Z_{\sigma(j)})_{ j\in \range{\ncal^{(1)}+1, \ncal+\ntest}} ; \dtrain)=S, \dcal^{(1)},\dtrain),
\end{align*}
by using the permutation preserving property (ii) and because $\sigma(S)=S$ by definition of $\sigma$. Now using that the distribution of $(Z_{\sigma(j)})_{ j\in \range{\ncal^{(1)}+1, \ncal+\ntest}} $ is the same as $(Z_{j})_{ j\in \range{\ncal^{(1)}+1, \ncal+\ntest}} $ conditionally on $\dtrain$ and $\dcal^{(1)}=(Z_j)_{ j\in \range{\ncal^{(1)}}}$,  the last display is equal to 
\begin{align*}
\mathcal{D}( (Z_{j})_{j \in S, j\geq \ncal^{(1)}+1} \:|\: \mathcal{S}^{(0)}=S, \dcal^{(1)},\dtrain),
\end{align*}
which provides the desired exchangeability for valid inference  on $(\dtestX)_{\mathcal{S}^{(0)}}$ using $(\dcal^{(2)})_{\mathcal{S}^{(0)}}$.
\end{proof}

\section{Additional illustrations in the regression case}\label{sec:moreregressionillustration}

  \subsection{Excluding a null value $y_0\in \R$} \label{sec:y0regression}

    We consider here $\mathcal{I}=\{ \mbox{$I$ interval of } \R\::\:  y_0\notin I \}$, which is useful in situation where the value $y_0$ corresponds to some ``normal'' value and the user wants to report only prediction intervals for ``abnormal'' individuals, that is, when the outcome value deviates from this reference value. The score considered here is $S_y(x)=|y-\mu(x)|/\sigma(x)$, where $\mu$, $\sigma$ are predictors of the conditional mean and variance, respectively.  
    
    An illustration is provided in Figure~\ref{fig-regressexnonull} in the non-parametric regression Gaussian model for different situations: 
    Rows 1,2 correspond to an homoscedastic model with perfect prediction of the variance  and a mean predictor with various accuracy (less accurate in the middle for row 1, more accuracy in the middle for row 2). Rows 3,4 correspond to an heteroscedastic model with perfect prediction of the mean and a the variance predictor under-estimating the variance in the middle for row 3, and over-estimating the variance in the middle for row 4). The marginal prediction interval (no selection) is displayed in light blue while the prediction interval after selection is displayed in dark red.  The quantities reported at the top of each panel are the FCR and adjusted power averaged over $1000$ repetitions (while each panel only displays the last experiment as a typical realisation of the sample).

\begin{figure}
\begin{tabular}{ccc}
\hspace{-11mm}Classical conformal &\hspace{-9mm} \texttt{InfoSP}  &\hspace{-9mm} \texttt{InfoSCOP}
\\
\hspace{-11mm}\includegraphics[width=0.38\textwidth, height = 0.24\textheight]{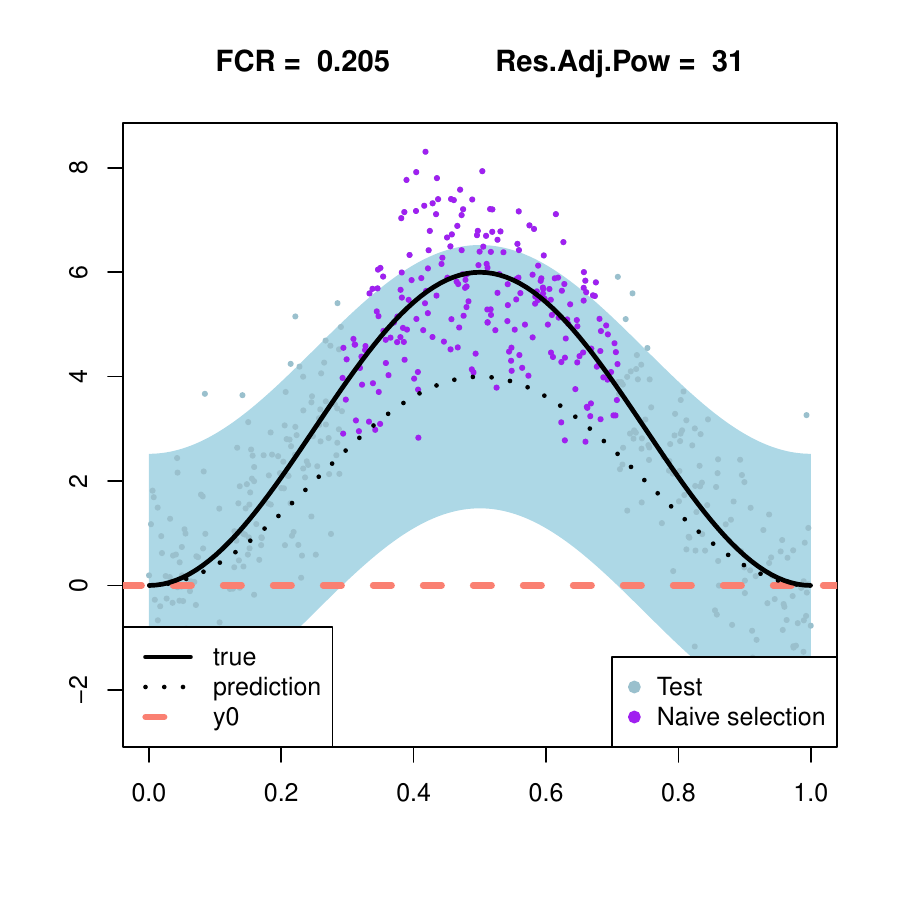}&\hspace{-11mm}
\includegraphics[width=0.38\textwidth,  height = 0.24\textheight]{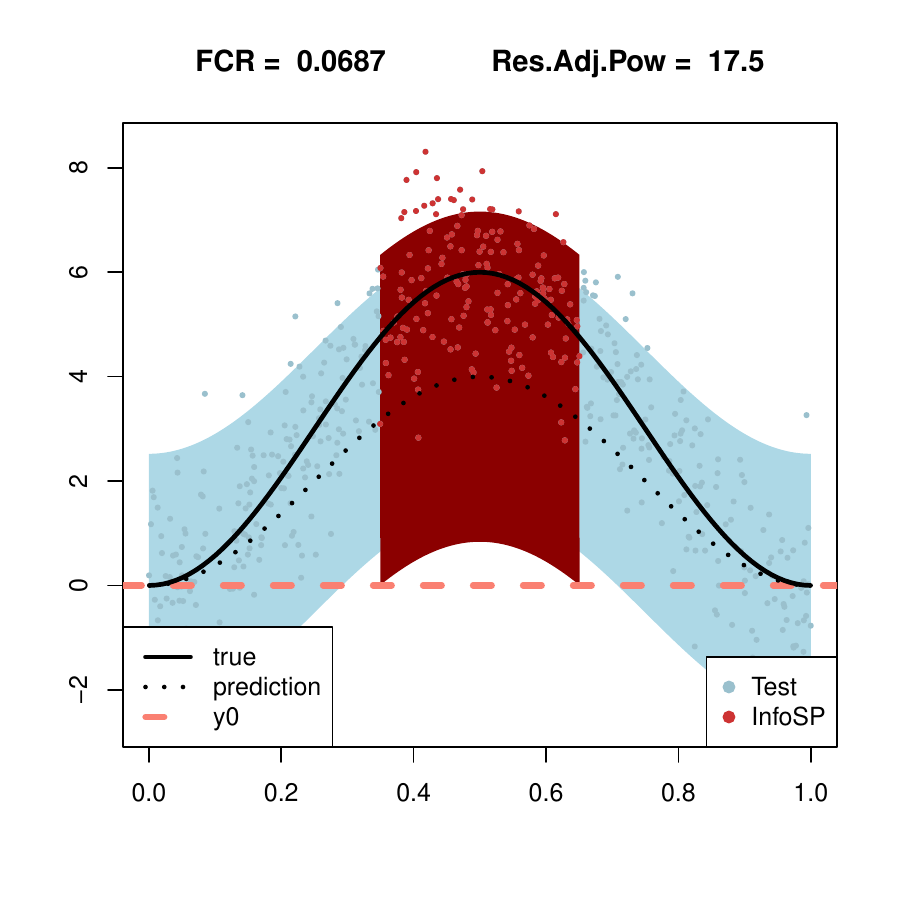}&\hspace{-11mm}
\includegraphics[width=0.38\textwidth,  height = 0.24\textheight]{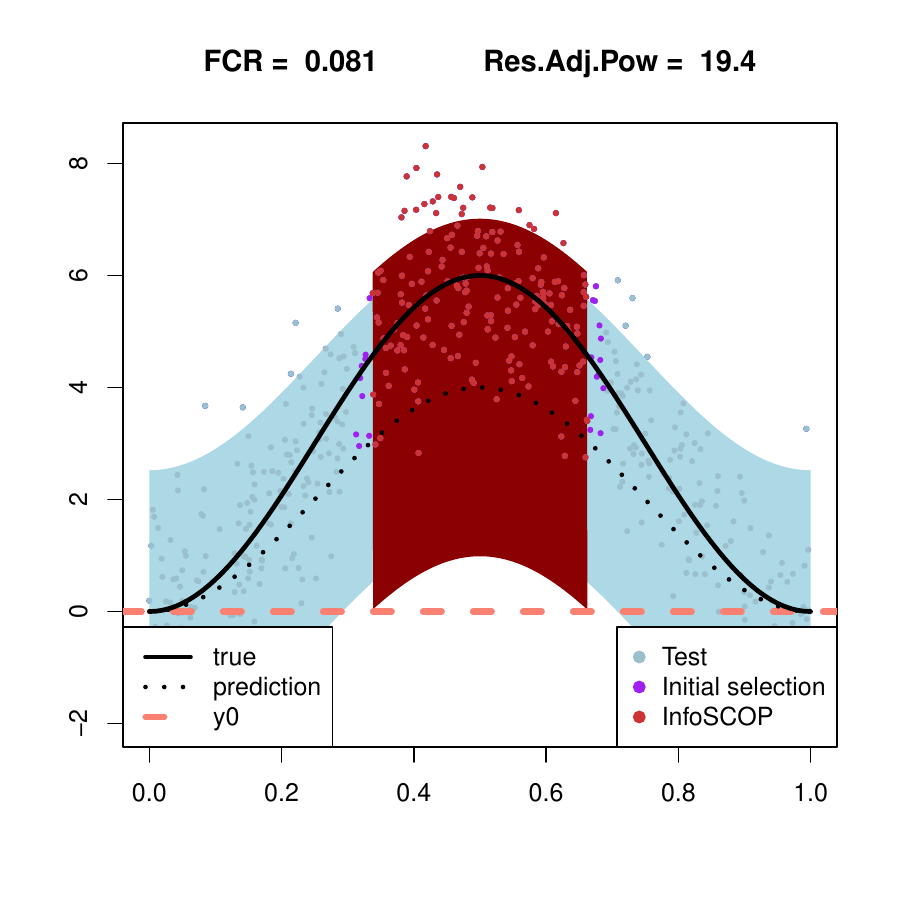}
\vspace{-6mm}\\
\hspace{-11mm}\includegraphics[width=0.38\textwidth, height = 0.24\textheight]{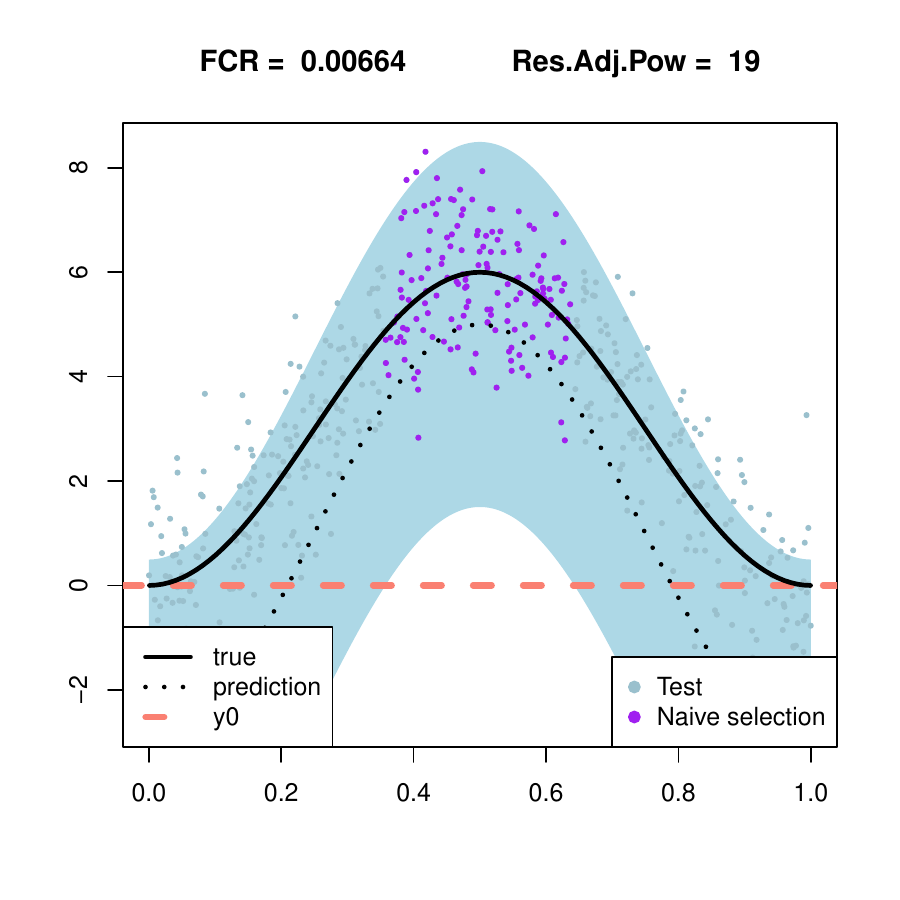}&\hspace{-11mm}
\includegraphics[width=0.38\textwidth,  height = 0.24\textheight]{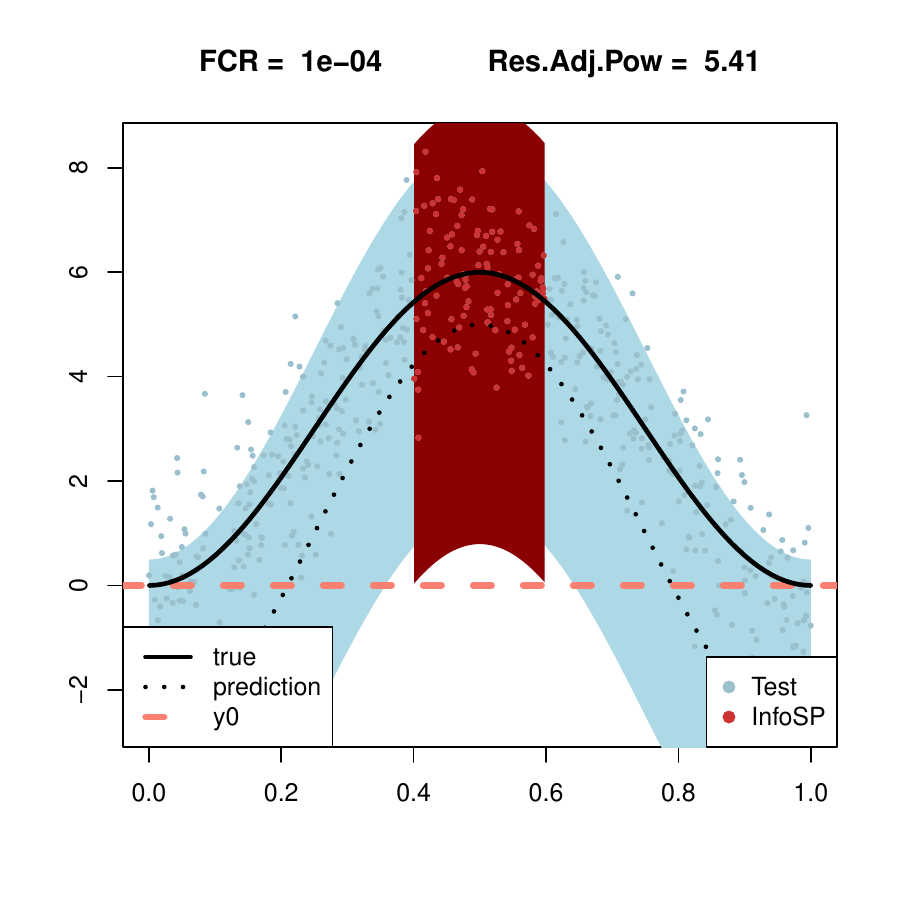}&\hspace{-11mm}
\includegraphics[width=0.38\textwidth,  height = 0.24\textheight]{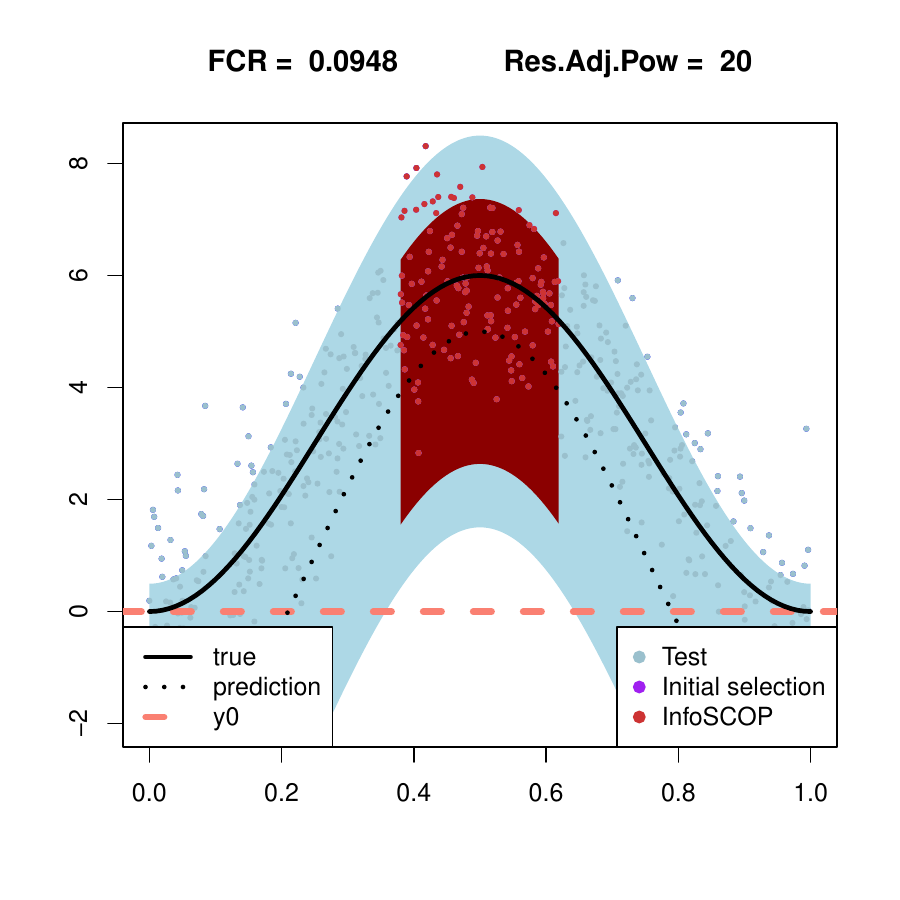}
\vspace{-6mm}\\
\hspace{-11mm}\includegraphics[width=0.38\textwidth, height = 0.24\textheight]{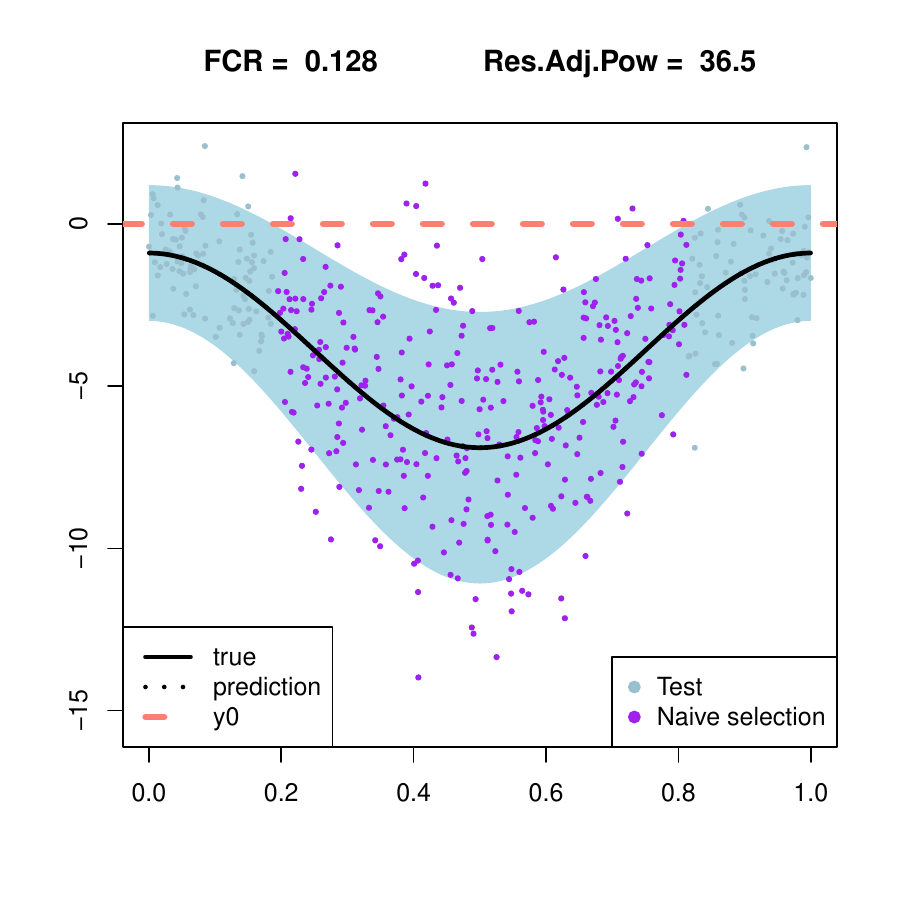}&\hspace{-11mm}
\includegraphics[width=0.38\textwidth,  height = 0.24\textheight]{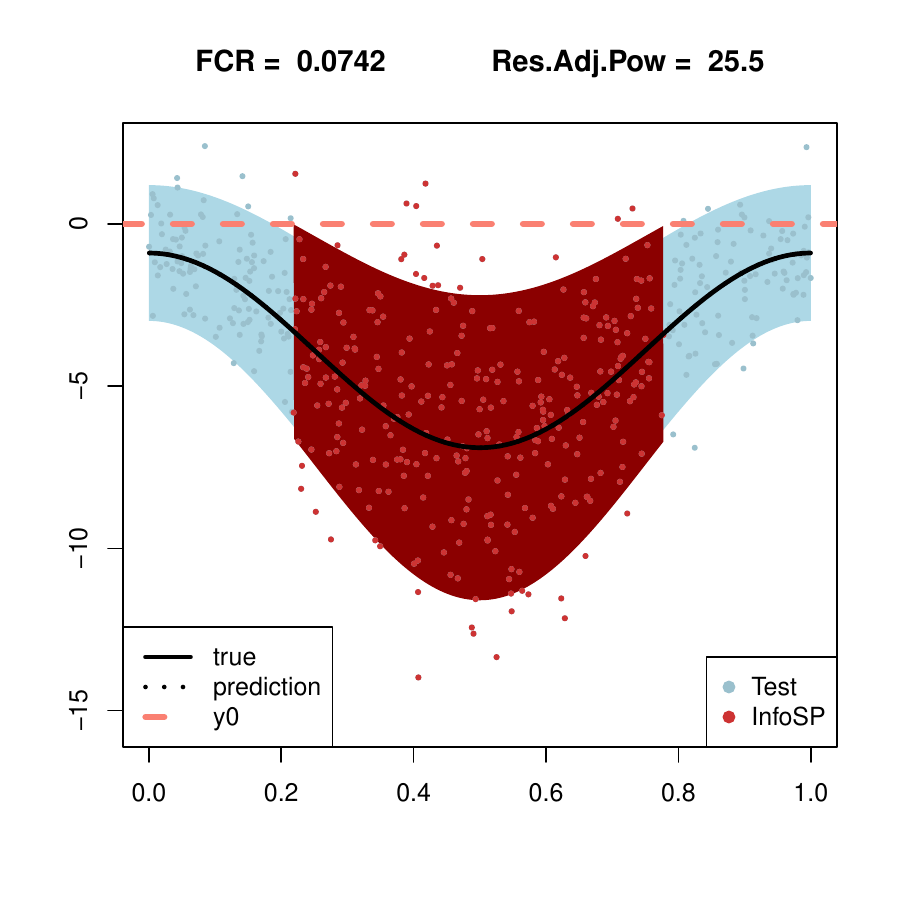}&\hspace{-11mm}
\includegraphics[width=0.38\textwidth,  height = 0.24\textheight]{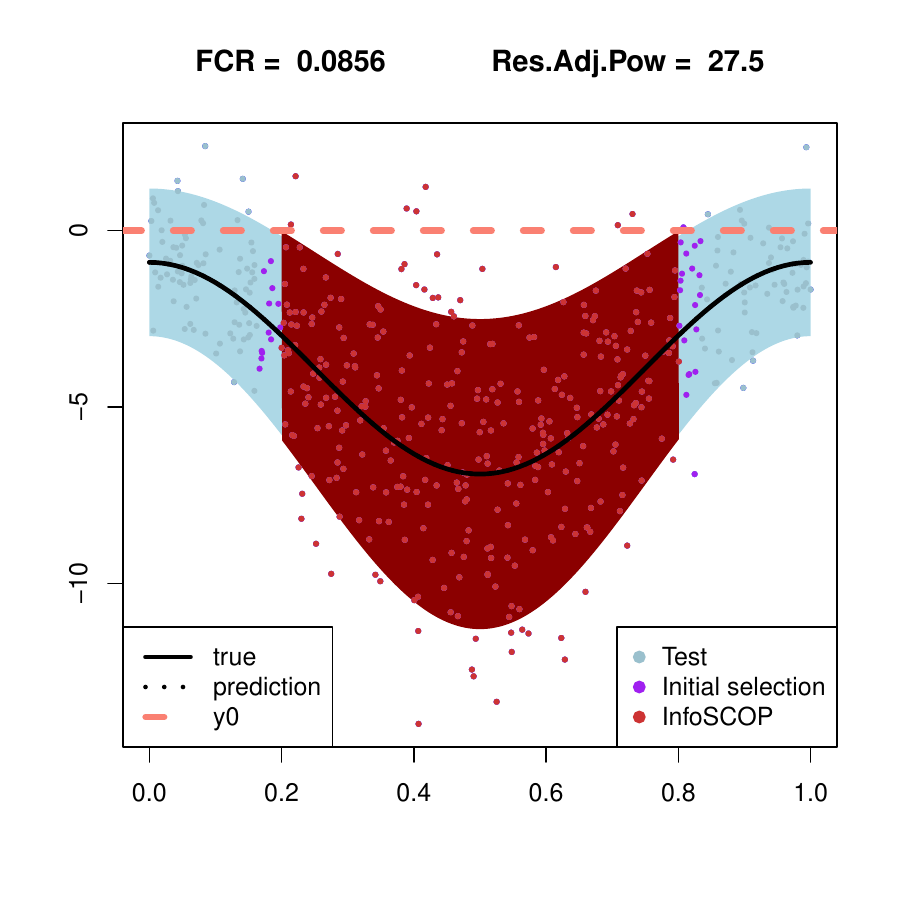}
\vspace{-6mm}\\
\hspace{-11mm}\includegraphics[width=0.38\textwidth, height = 0.24\textheight]{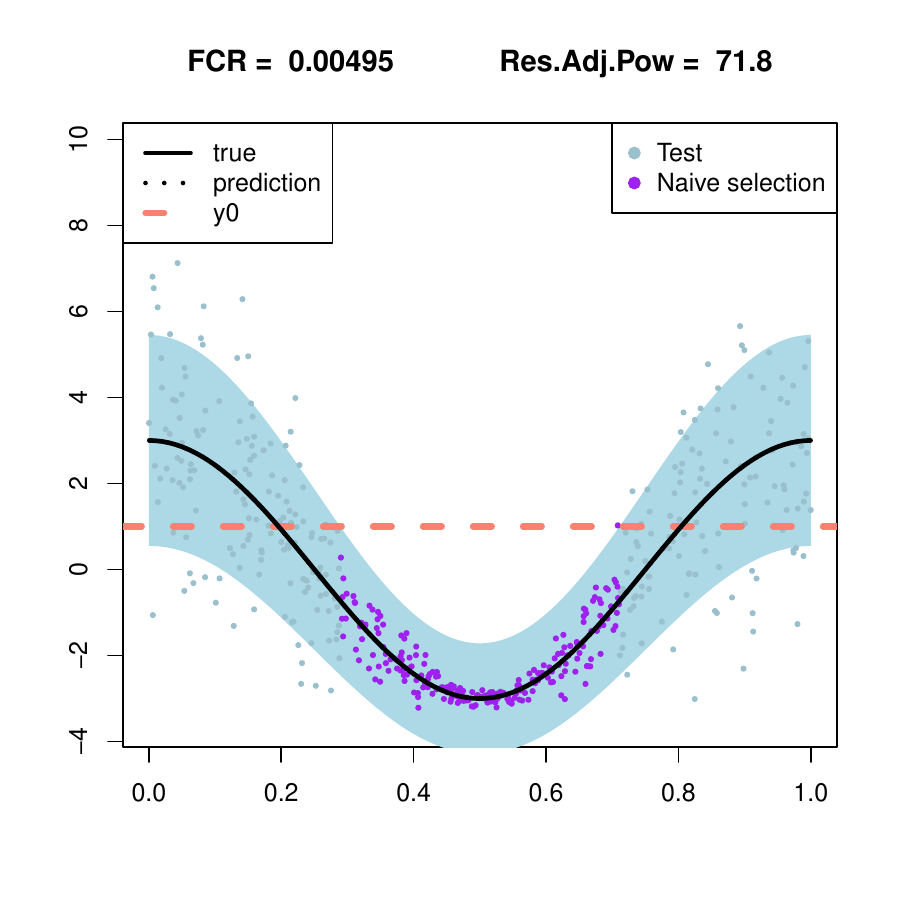}&\hspace{-11mm}
\includegraphics[width=0.38\textwidth,  height = 0.24\textheight]{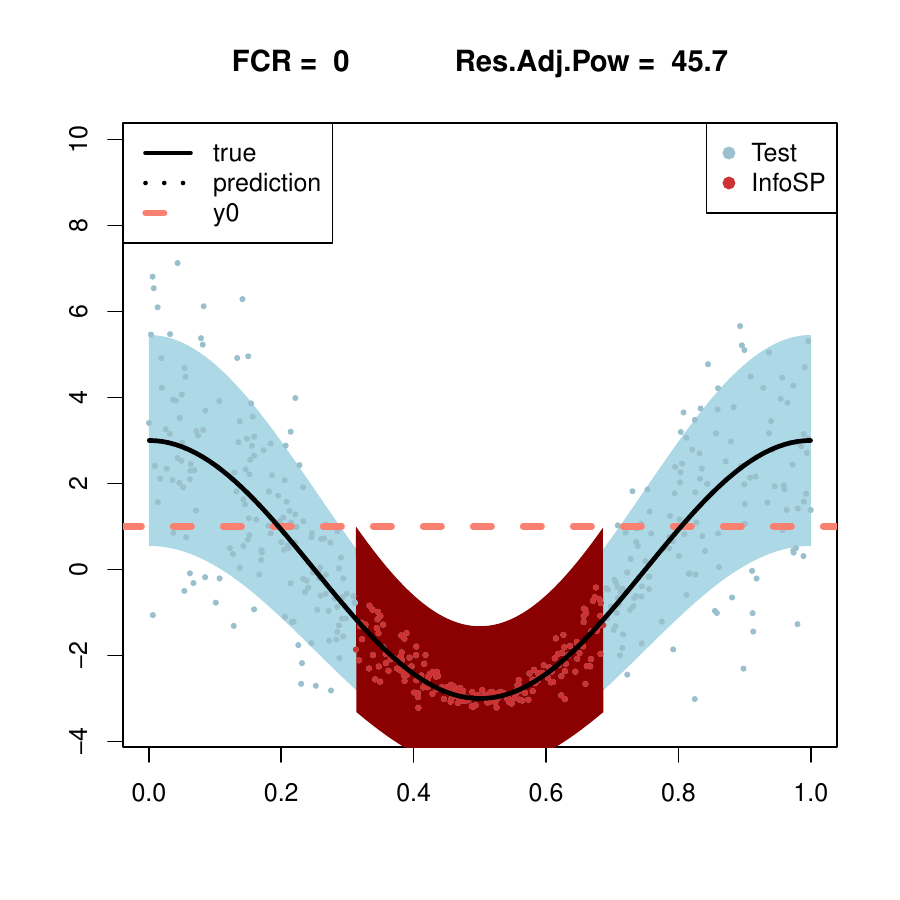}&\hspace{-11mm}
\includegraphics[width=0.38\textwidth,  height = 0.24\textheight]{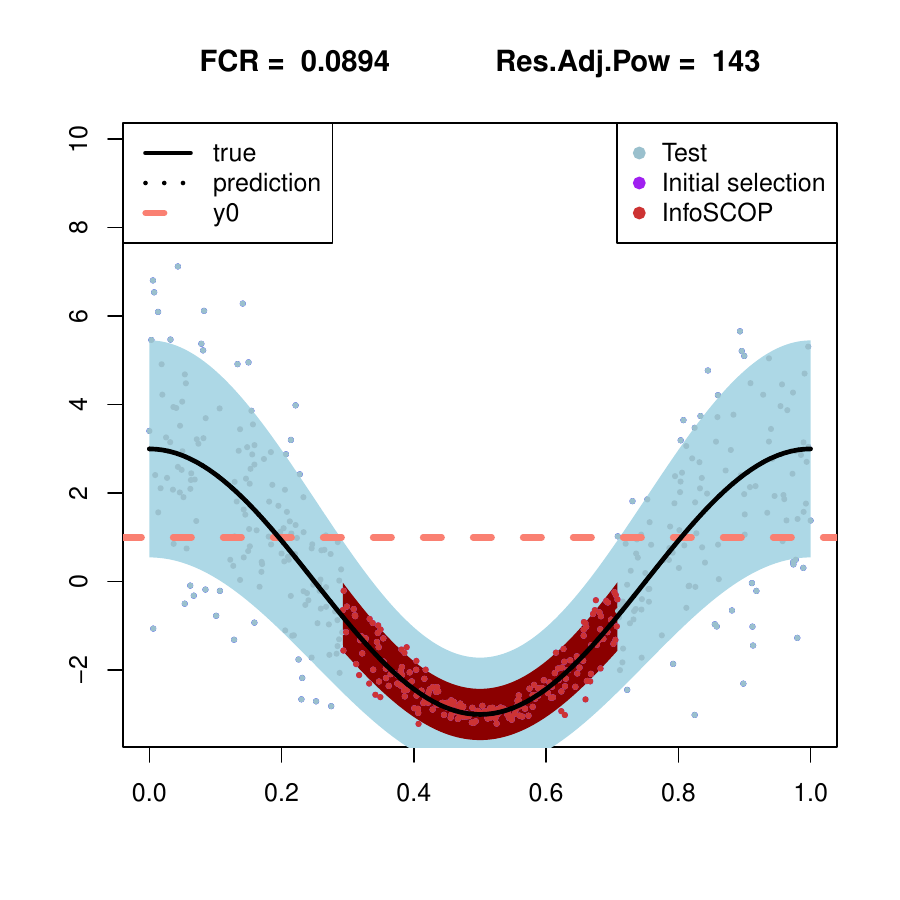}
\end{tabular}
\vspace{-6mm}
\caption{
Informative selection by excluding the null value $y_0$ in the (homoscedastic or heteroscedastic) regression case, see \S~\ref{sec:y0regression}.
\label{fig-regressexnonull}} 
\end{figure}

    First, as expected, the naive classical conformal selection does not provide FCR control in all cases, while \texttt{InfoSP} and \texttt{InfoSCOP} do control the FCR in any case. 
    Second, \texttt{InfoSCOP} always improves \texttt{InfoSP} in terms of power, and the range of improvement depends on where non-covering errors are likely to happen: 
    when errors are more likely to arise on the selection (rows 1,3), the behavior of \texttt{InfoSP} and \texttt{InfoSCOP} are similar; when errors are less likely to arise on the selection (rows 2,4), \texttt{InfoSCOP} improves over \texttt{InfoSP} in a striking manner (and is even better than classical conformal). This is both because of the reduction of the selection effect and because the pre-processed $p$-values are calibrated with much lower residuals and thus are much more efficient than the original $p$-values.

\subsection{Prediction intervals for \cite{Jin2023selection} selection}\label{sec:JCselect}

 We focus here on the case where the user wants to build prediction sets only for outcomes $Y_{n+i}> y_0$, which corresponds to excluding the set $\mathcal{Y}_0=(-\infty,y_0]$ and is related to the selection proposed in \cite{Jin2023selection}. 

We consider here the aim of finding one-sided prediction intervals on the selected. 
Let us assume that the score function is monotone in the following sense \citep{Jin2023selection}: for all $x\in \R^d$, for $y\leq y'$, $S_y(x)\geq S_{y'}(x)$. A classical example of monotonic score function is  given by $S_y(x)=(\mu(x)-y)/\sigma(x)$. The following result summarizes our finding in this case.

\begin{corollary}\label{cor:JConesided}
Consider the iid model in the regression case and assume that the score function is monotone (see above) and such that Assumption~\ref{assI} (ii) (iii) and Assumption~\ref{as:noties} hold. Then the following holds for \texttt{InfoSP} with informative collection $\mathcal{I}=\{I\mbox{ interval of } \R\::\: I\cap (-\infty,y_0]=\emptyset \}$ and $p$-value collection $\pfullbf$ \eqref{standardpvalue}: 
\begin{itemize}
    \item[(i)] \texttt{InfoSP} selects $\mathcal{S}=\BH(\mathbf{q})$ the rejected set of BH procedure at level $\alpha$ applied with the $p$-values
    $$
    q_i=\pfull{y_0}{i}= \frac{1}{n+1}\Big(1+\sum_{j=1}^n \ind{S_{Y_j}(X_j)\geq S_{y_0}(X_{n+i})} \Big),
    $$
    which coincides with the rejection set of   the procedure proposed in \cite{Jin2023selection}.
    \item[(ii)] The selection $\mathcal{S}=\BH(\mathbf{q})$ of \texttt{InfoSP}  controls the FDR at level $\alpha$ in the following sense:
    $$
    \sup_{P_{XY}} \E_{(X,Y)\sim P_{XY}}\left[\frac{\sum_{i\in \mathcal{S}} \ind{Y_{n+i} \leq y_0}}{1\vee |\mathcal{S}|}\right]\leq \alpha.
    $$
    \item[(iii)] The prediction intervals are of the form
    $
    \mathcal{C}_{n+i}= \{y>y_0\::\: S_{y}(X_{n+i}) \leq  S_{(n_\alpha)}\}$, $i\in \mathcal{S},$
    where $S_{(1)}\leq \dots\leq S_{(n)}$ are the ordered calibration scores $S_{Y_j}(X_j)$, $1\leq j\leq n$ (with $S_{(n+1)}=+\infty$), and $n_\alpha=\lceil (1- \alpha|\mathcal S(\mathbf{p})|/m)(n+1)\rceil$.
    \item[(iv)] These prediction intervals control the FCR at level $\alpha$ in the sense of \eqref{iidcontrol}.
\end{itemize}
\end{corollary}

In other words, the above result complements the multiple testing procedure of \cite{Jin2023selection}, by providing in addition FCR controlling informative prediction sets on the selected outcomes.

\begin{proof}
The proof is direct with Example~\ref{ex:regression2} and monotonicity (for (i)), Lemma~\ref{lem:FDRsmallerFCR} (for (ii)), Remark~\ref{rem:caldepth} (for (iii)), Theorem~\ref{thm-gen-basic} (for (iv)).
\end{proof}

Obviously, a similar result holds for \texttt{InfoSCOP}, for any initial selection step $\mathcal{S}^{(0)}\subset \range{r+1,n+m}$ that satisfies the permutation preserving Assumption~\ref{as:S0}. 
Corollary~\ref{cor:JConesided} is illustrated on Figure~\ref{fig-regressexnonullJinCandes} when $\mathcal{S}^{(0)}$ is taken here has $\BH(\mathbf{q})$ at level $2\alpha$ with the score $S_y(x)=(\mu(x)-y)/\sigma(x)$. 
The comments are qualitatively similar to those of Figure~\ref{fig-regressexnonull}: when covering errors are less likely on the selection, the improvement of \texttt{InfoSCOP} over \texttt{InfoSP} is substantial.

\begin{figure}[h!]
\begin{tabular}{ccc}
Classical conformal & \texttt{InfoSP}  & \texttt{InfoSCOP}
\\\hspace{-9mm}\includegraphics[width=0.38\textwidth, height = 0.24\textheight]{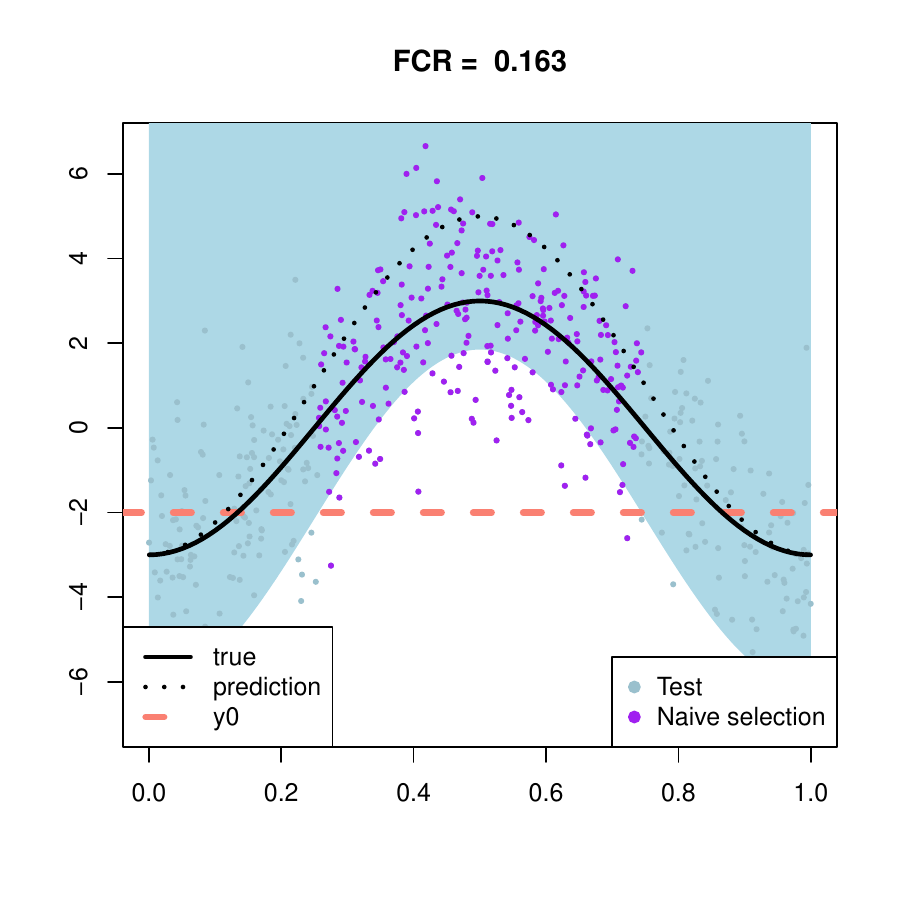}&\hspace{-9mm}
\includegraphics[width=0.38\textwidth,  height = 0.24\textheight]{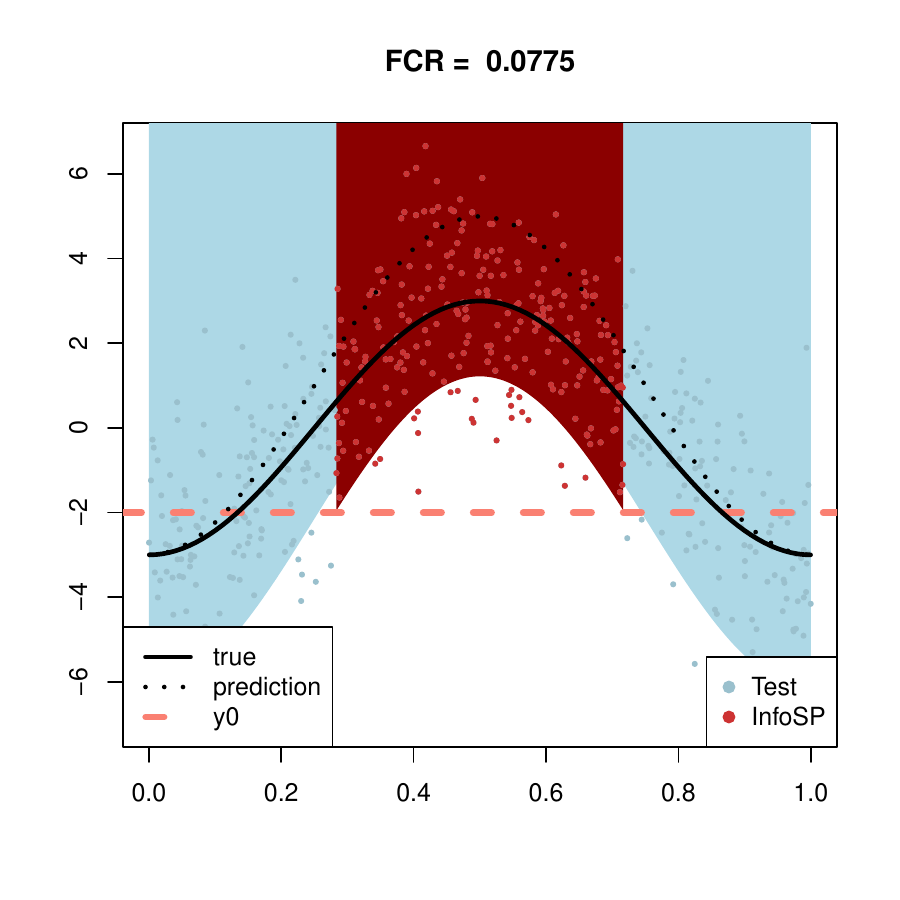}&\hspace{-9mm}
\includegraphics[width=0.38\textwidth,  height = 0.24\textheight]{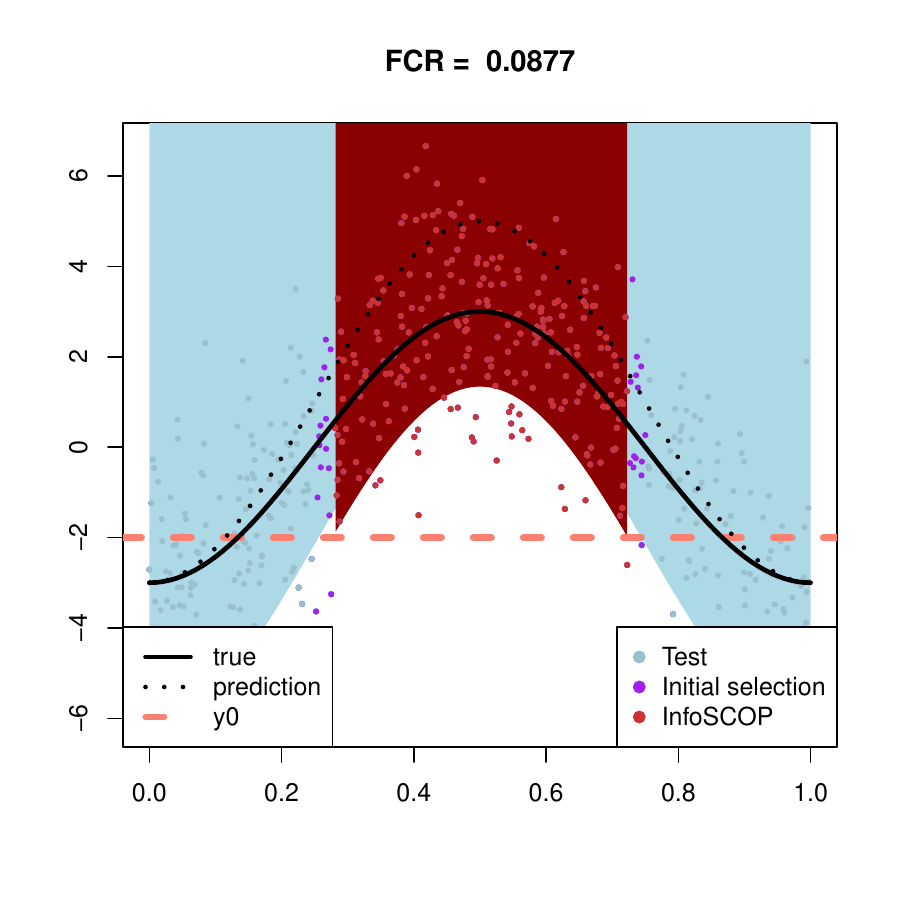}\vspace{-6mm}\\
\hspace{-9mm}\includegraphics[width=0.38\textwidth, height = 0.24\textheight]{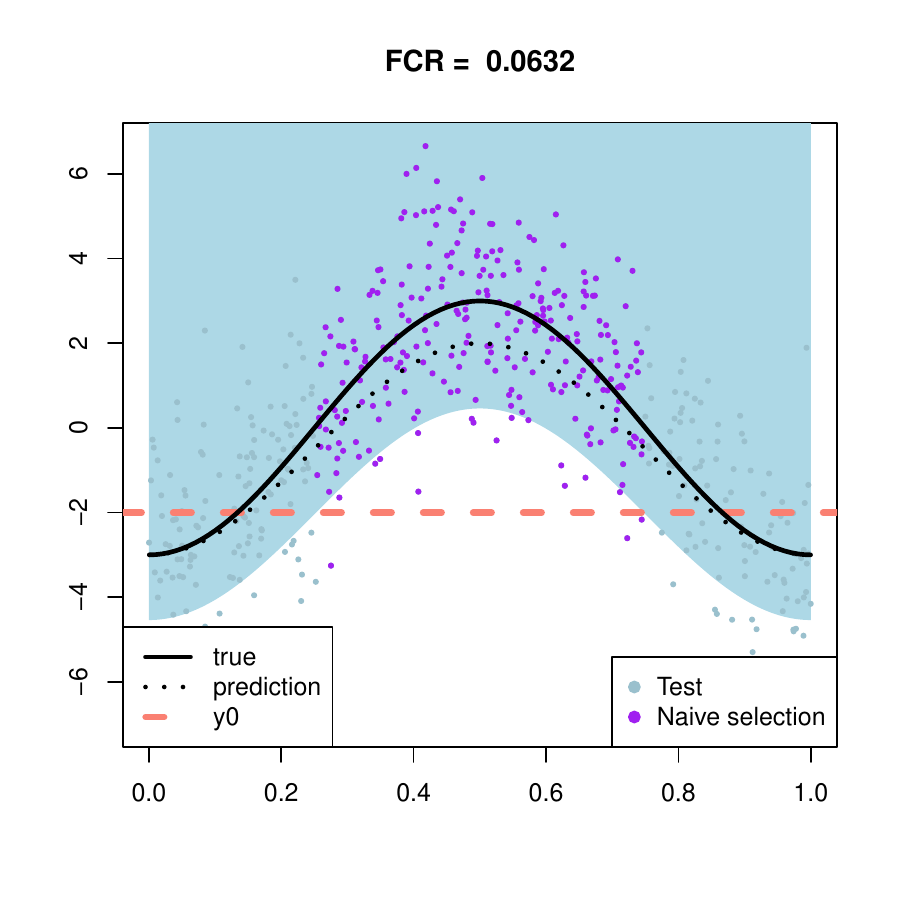}&\hspace{-9mm}
\includegraphics[width=0.38\textwidth,  height = 0.24\textheight]{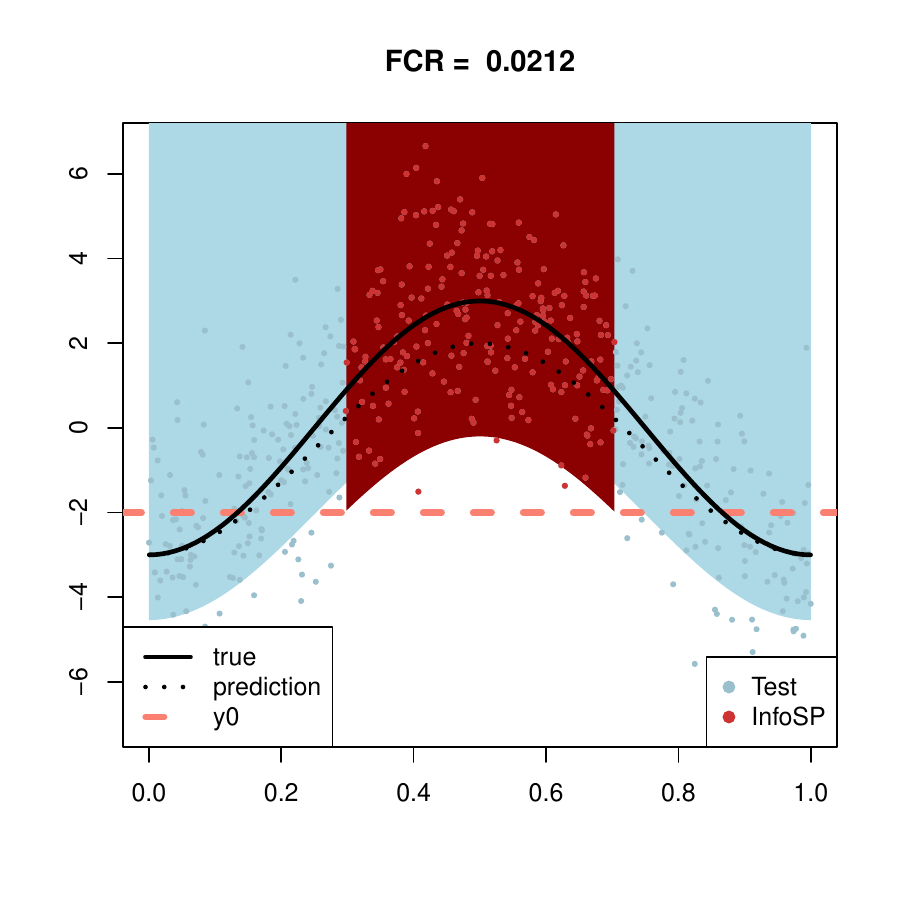}&\hspace{-9mm}
\includegraphics[width=0.38\textwidth,  height = 0.24\textheight]{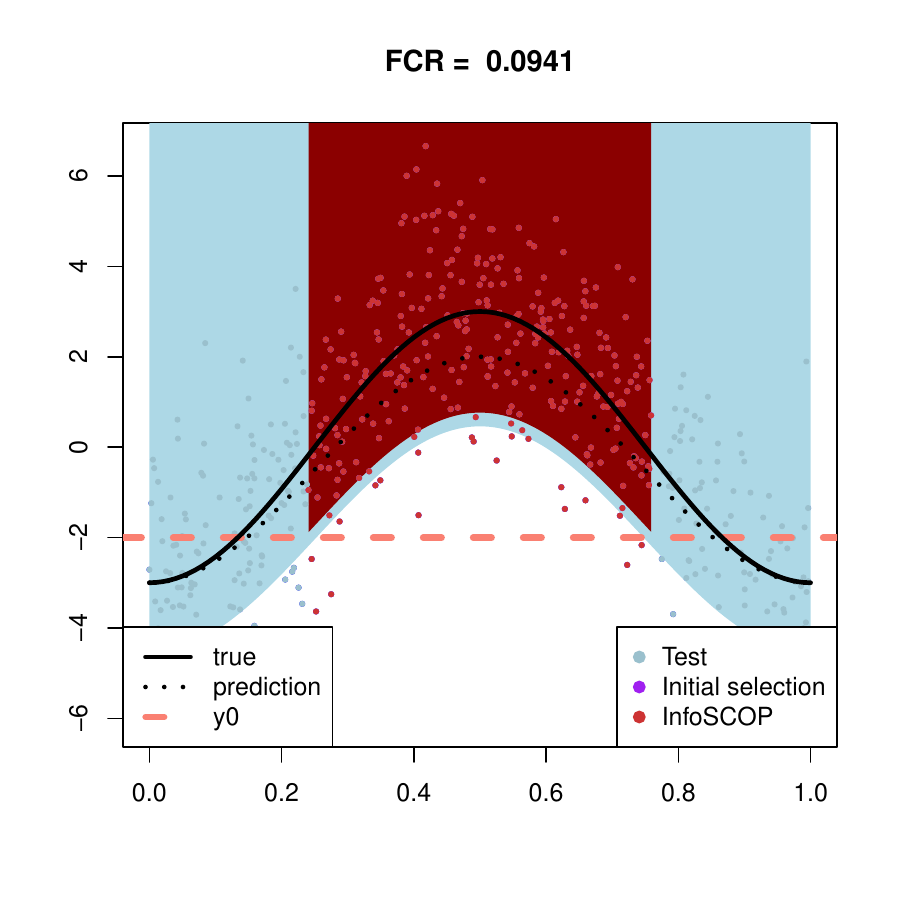}
\end{tabular}
\vspace{-9mm}
 \caption{Same as Figure~\ref{fig-regressexnonull} with \cite{Jin2023selection} type selection and one-sided prediction intervals, see \S~\ref{sec:JCselect}.
\label{fig-regressexnonullJinCandes}} 
\end{figure}

\subsection{{    
 Predicting gene expression from promoter sequences
}}\label{subsec-yeast}

\add{
\cite{Vaishnav2022} fit a convolutional neural network model on tens of millions of random promoter sequences with the goal of predicting the expression
level of a particular gene induced by a promoter sequence. They then used the model's
predictions to study the effects of promoters.  They evaluated the performance on 61150 labeled sequences that  were not used in training the model. The 61150 pairs of predicted and measured expressions will be used as a basis for a simulation to evaluate the performance of classic conformal, \texttt{InfoSP}, and \texttt{InfoSCOP}.  
}

\add{
Suppose an analyst considers $m=50$ new random promoter sequences. From these, the analyst is interested in the prediction intervals for promoter sequences that are under-expressed.  Importantly, the range of plausible values (i.e., a two-sided prediction interval) are of interest to the analyst, since it is not enough to establish with confidence that the expression is below a certain value,  but also to know that the expression is above zero  and evaluate the plausible extent to which it exceeds zero. 
}

\add{
Suppose the analyst has access to the expression and predicted expression of $n=200$ random promoter sequences. These examples will serve as the calibration sample.  
}

\add{
We consider the score function $S_y(x) = |\mu(x)-y|$, where $\mu(x)$ is the predicted expression (more sophisticated score functions, such as the locally weighted or quantile based functions discussed in \S~\ref{sec:appliRegression}, cannot be used since we only have access to $\mu(x)$ from the prediction machine). We define an example as informative if its prediction interval is entirely below the threshold $a=9$. Therefore, the $p$-value for testing the (random) null hypothesis that the example is not informative is $$q_i = \max_{y\geq a} \pfull{y}{i} = \pfull{a}{i}\times \ind{\mu(X_{n+i})<a}+
\ind{\mu(X_{n+i})\geq a},$$
see Corollary~\ref{cor:abexcluded} with $b=+\infty$.
}

\add{
In order to assess our methods using the  61150 labeled native yeast promoter sequences, we randomly sample  labeled and
unlabeled data sets as follows. For each of $B=10000$ trials, we randomly sample $n+m$ promoters, out of which  $n$ points serve as the labeled calibration data set and $m$  points as the unlabeled test data set. 
The results are summarized in Table \ref{tab:yeast}. Due to selection of informative examples, the classic conformal procedure that constructs each prediction interval at level $1-\alpha = 0.9$ fails to control the FCR at the $\alpha = 0.1$ nominal level. \texttt{InfoSP} provides the required control, but at the cost of some non-negligible power loss. \texttt{InfoSCOP} applies an initial selection step in which examples from the test sample and the second half the calibration sample are forwarded to \texttt{InfoSP}, only if their conformal $p$-value, computed using the first half of the calibration sample,  for testing the (random) null hypothesis that the example is not informative, is at most $\alpha$. \texttt{InfoSCOP} controls the FCR at level $\alpha$ and has only slightly reduced power compared to classic conformal. 
}

\begin{table}[t!]

\centering
\begin{tabular}{|c|c|c|c|}
\hline
& \textbf{CC} & \textbf{\texttt{InfoSP}} & \textbf{\texttt{InfoSCOP}} \\ \hline
\textbf{FCR} & 0.150 & 0.030 & 0.081 \\ \hline
\textbf{Power} & 16.9 & 12.6 & 15.8 \\ \hline
\textbf{Res-adjusted Power} & 2.81 & 1.73 & 2.37 \\ \hline
\end{tabular}
\caption{\label{tab:yeast}\add{The FCR, power, and resolution-adjusted power, for the three procedures considered for the yeast application (see text). The resolution adjusted power is the expected number of discoveries divided by the length of the prediction interval (since the length is fixed in this analysis). Based on $B=10000$ repetitions.}}
\end{table}

\add{
Figure \ref{fig:yeast} shows the prediction intervals for a single data generation. \texttt{InfoSCOP} selects almost the same number of informative examples as classic conformal, but all prediction intervals constructed with \texttt{InfoSCOP} cover the true value. Moreover, the lower bounds are all above zero. In practice, following such an analysis, the analyst can use the list of prediction intervals from \texttt{InfoSCOP} to the next stage: knowing with confidence the prediction intervals for informative promoters for this set can help the future design of $n+m$ examples with even lower expressions (see \cite{Vaishnav2022} for uses for the design of promoters for low gene expression).  There will be a  lab cost of culturing the yeast of $n$ examples  to actually measure the expression level for $n$ promoters to serve as the calibration data, but $m$ can be quite large. 
}

\begin{figure}[h!]
\begin{center}
    \includegraphics[width=0.5\textwidth, page=19]{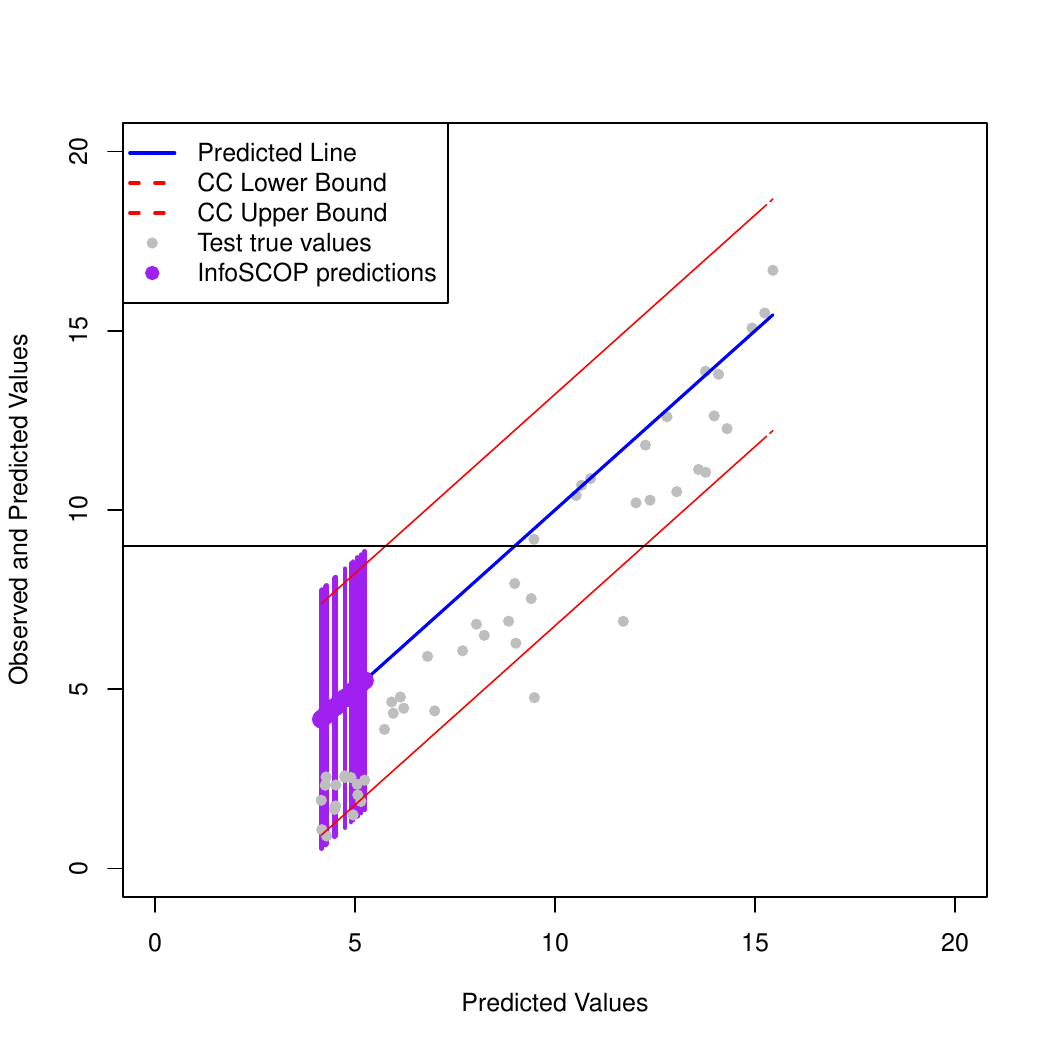}
\end{center}
\vspace{-1cm}
\caption{\label{fig:yeast} \add{The observed versus predicted values for yeast expression, as well as the classic conformal prediction intervals upper and lower bounds (red lines), and the \texttt{InfoSCOP} prediction intervals (purple vertical lines), for a single data generation (see text in \S~\ref{subsec-yeast} for details). The false coverage proportion for classic conformal is 0.176 and 0 for \texttt{InfoSCOP}. Classic conformal reports 17 informative examples, but three examples are not covered by their prediction interval. \texttt{InfoSCOP} reports 16 informative examples (a subset of the 17 reported by classic conformal, including the three examples with non-covering classic conformal prediction intervals).  }
} 
\end{figure}

\add{
\begin{remark}
The selection in the above analysis fits within the \cite{Jin2023selection} selection framework outlined in the previous section, but with a different score function than the scores suggested  \cite{Jin2023selection}, since our score function is selected to provide two-sided prediction intervals.      
\end{remark}
}

\subsection{Comparison with existing selective prediction sets} \label{sec:comparison}

\begin{figure}[h!]
\begin{center}
    \includegraphics[width=0.5\textwidth]{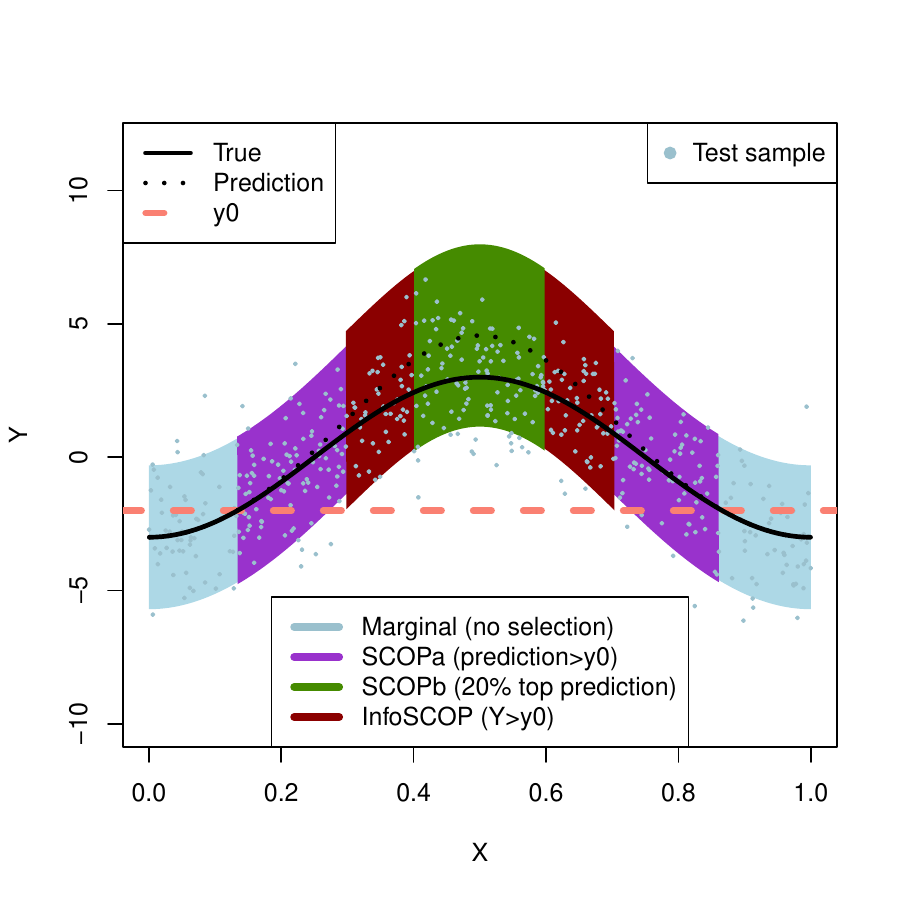} \hspace{-1cm}\includegraphics[width=0.5\textwidth]{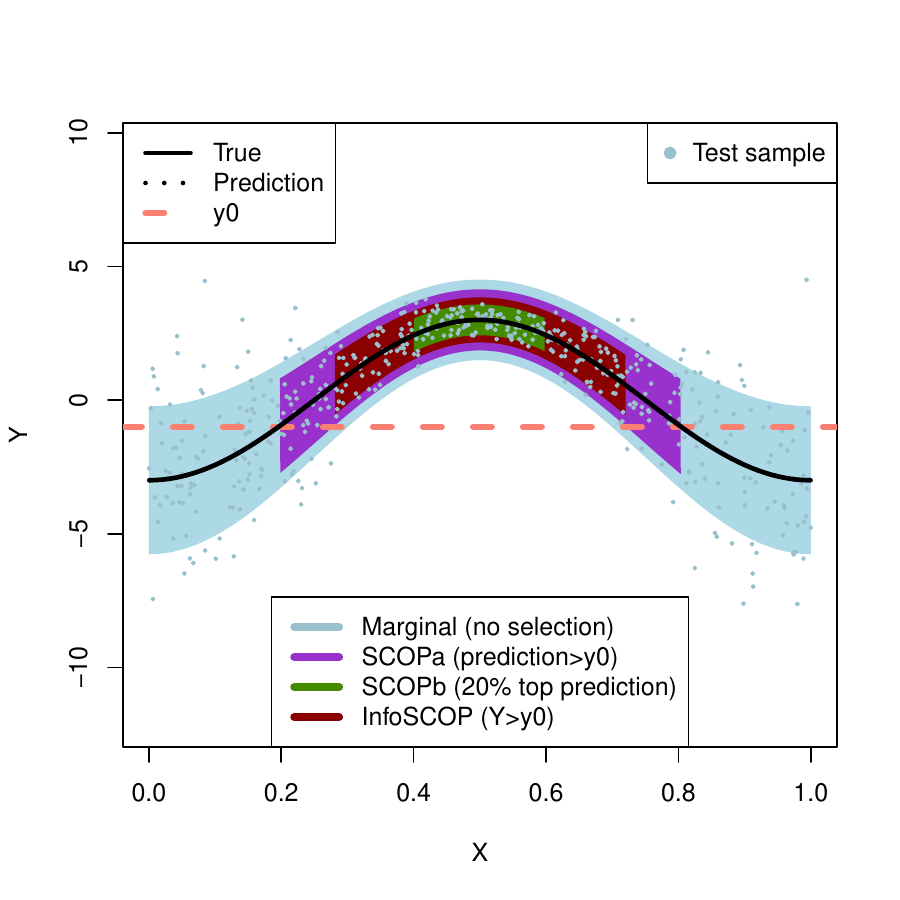}
\end{center}
\vspace{-1cm}
\caption{\label{fig:intro} Comparison to previous conformal selective inferences in the regression case: new informative prediction intervals (\texttt{InfoSCOP}) versus non-informative prediction intervals (SCOPab) (left: homoscedastic Gaussian regression with perfect variance prediction and errors in the mean prediction; right: heteroscedastic Gaussian regression with perfect mean prediction and errors in the variance prediction). 
Informative means here prediction intervals that does not contain $y_0$ (dashed line), see \S~\ref{sec:comparison}.
\label{fig:comparison}
} 
\end{figure}

Figure~\ref{fig:comparison} is useful to visualize the difference between the selection proposed by \cite{bao2024selective} and our informative selection. We display selective prediction intervals for two FCR controlling selections proposed in \cite{bao2024selective}: \texttt{SCOPa} uses a thresholding rule $\mathcal{S}=\{i\in \range{m}\::\: \mu(X_{n+i})\geq y_0\}$, while \texttt{SCOPb} selects the largest $\mu(X_{n+i})$'s. Each time, a suitable conditional conformal prediction set is built in \cite{bao2024selective} and reported in Figure~\ref{fig:comparison} in green and purple. As one can see, either the selection is too conservative, or the prediction interval could include the nominal $y_0$. This is not the case of \texttt{InfoSCOP} which always exclude $y_0$ by essence.

\section{Additional illustrations in the classification case}\label{sm_simulations}

\subsection{Three classes, each bivariate normal}
For the iid model,  as described in \S~\ref{subsec-simul-bivariatenormal}, we demonstrate the initial selection step in \texttt{infoSCOP} for excluding a null class   in one realization of the data generation in Figure \ref{fig-exchang-snapshot}.

\begin{figure}[h!]
\begin{center}
   \includegraphics[width=0.28\textwidth, height = 0.28\textheight, page=2]{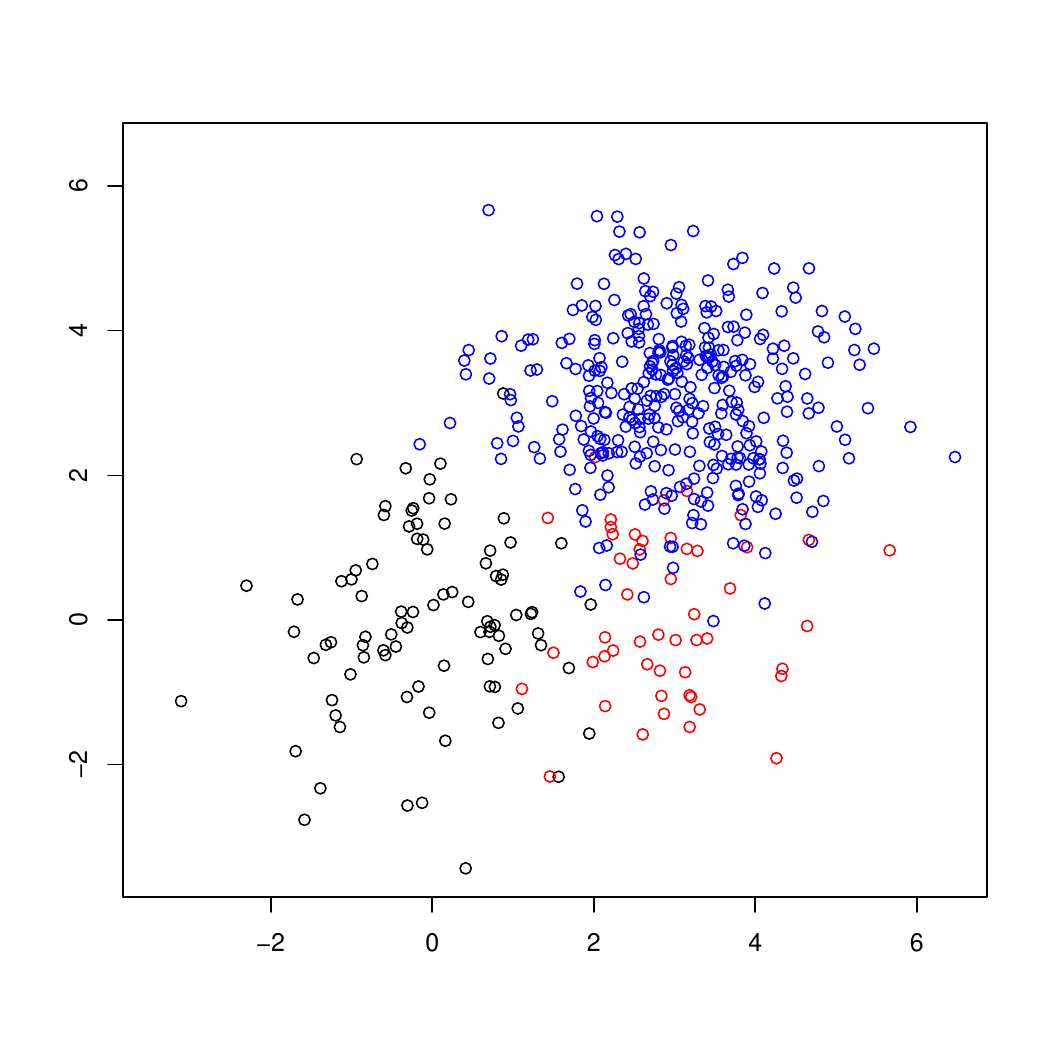}
    \includegraphics[width=0.28\textwidth, height = 0.28\textheight, page=3]{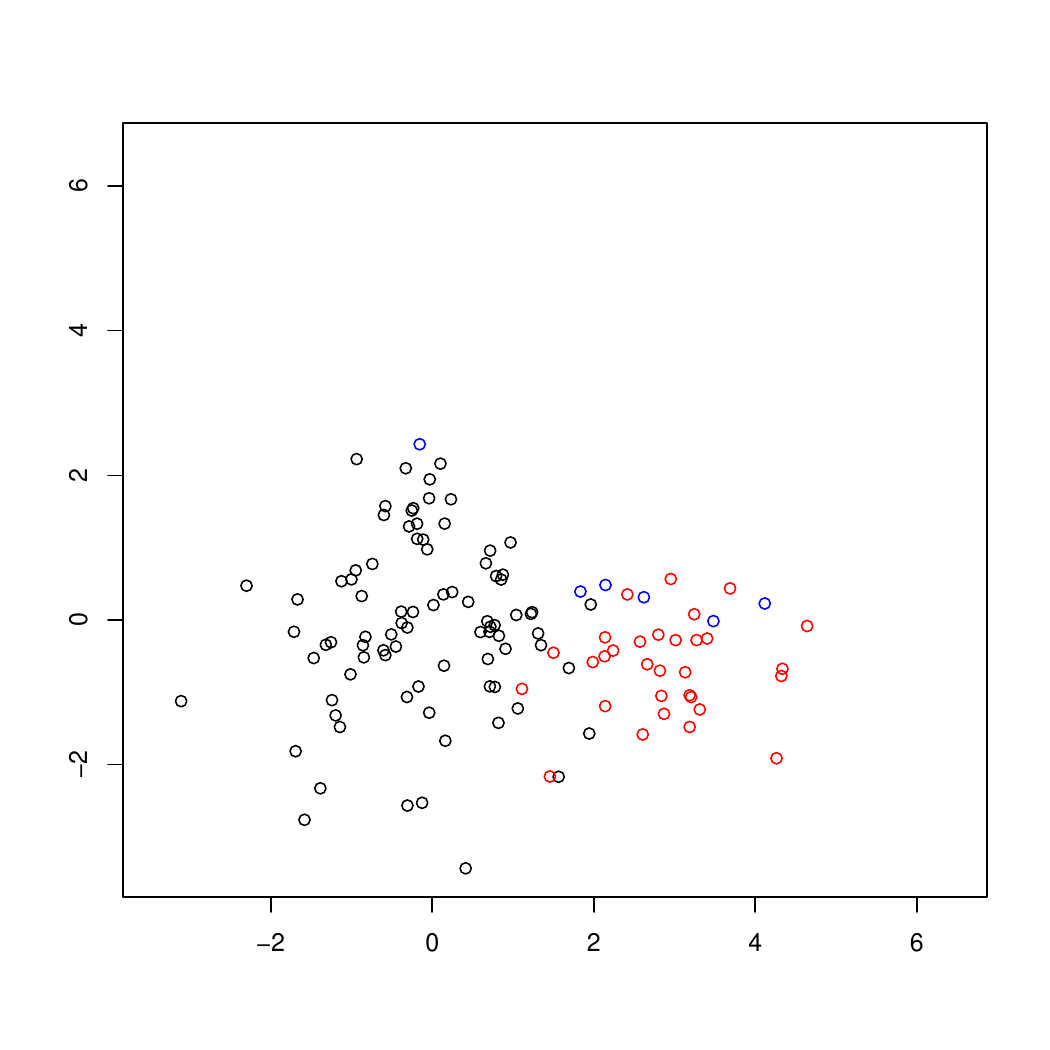} 
\end{center}
    \vspace{-10mm}
    \caption{\label{fig-exchang-snapshot} The test sample in a single data generation for selecting informative prediction sets that exclude a null class. The setting is that of unbalanced classes,  and the SNR is 3. The data points from the test sample of each of the three classes, where the null group is in blue:  left panel for the entire test sample, right panel the remaining examples from the test set after the initial selection step.   }
\end{figure}

For non-trivial classification, as described in \S~\ref{subsec:min-iid}, we provide numerical results in Figure \ref{fig-nontrivial-classification}. 

\begin{figure}[h!]
\begin{center}
 {\includegraphics[width=0.36\textwidth, height = 0.27\textheight, page=2]{legends.pdf}}

\vspace{-3.5cm}
 \begin{tabular}{ccc}
    \hspace{-5mm} \includegraphics[width=0.34\textwidth, height = 0.2\textheight, page=3]{subscriptMinP1.pdf} & 
\hspace{-5mm}\includegraphics[width=0.34\textwidth,  height = 0.2\textheight,page=1]{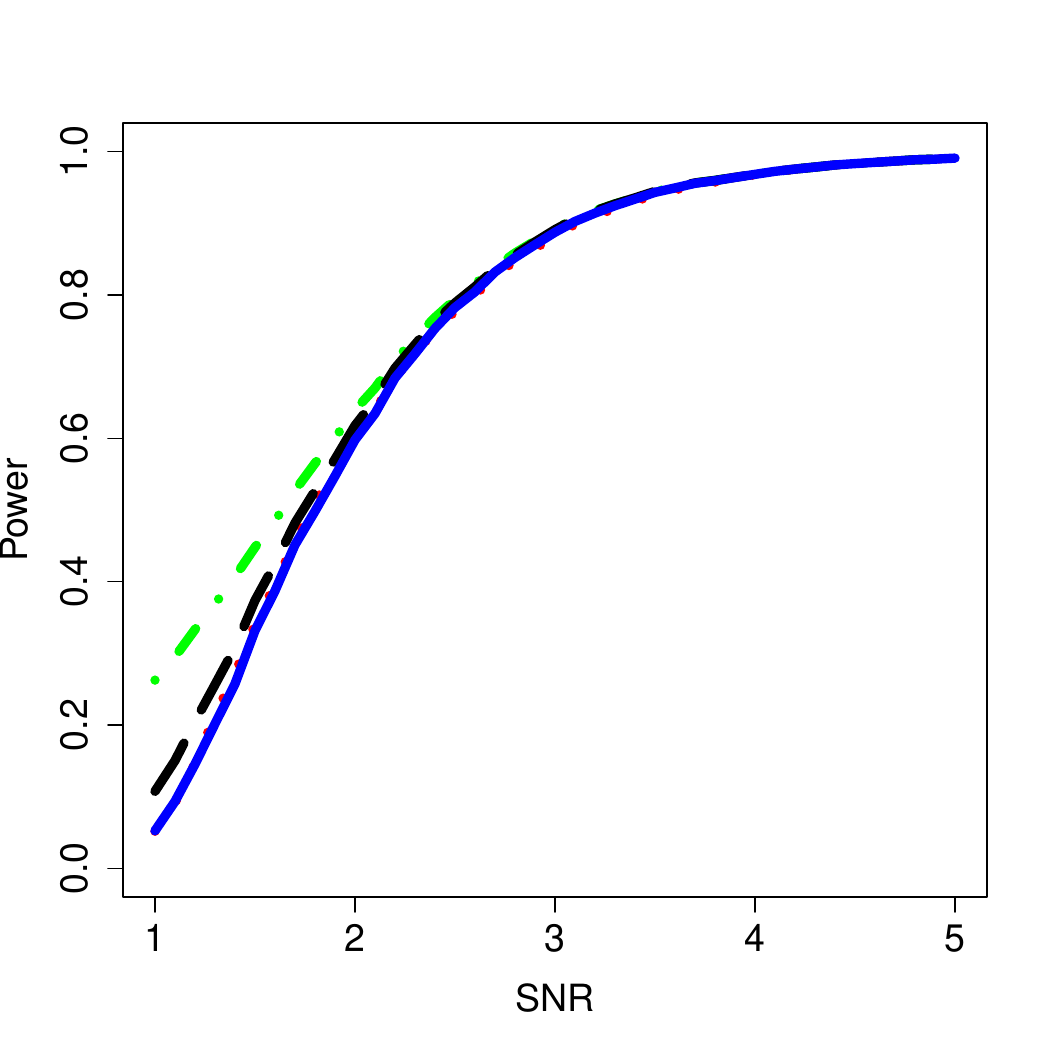} &\hspace{-5mm}\includegraphics[width=0.34\textwidth,  height = 0.2\textheight,page=2]{subscriptMinP1.pdf}\vspace{-5mm}\\
\hspace{-5mm}\includegraphics[width=0.34\textwidth, height = 0.2\textheight, page=3]{subscriptMinP2.pdf}&
\hspace{-5mm}\includegraphics[width=0.34\textwidth,  height = 0.2\textheight,page=1]{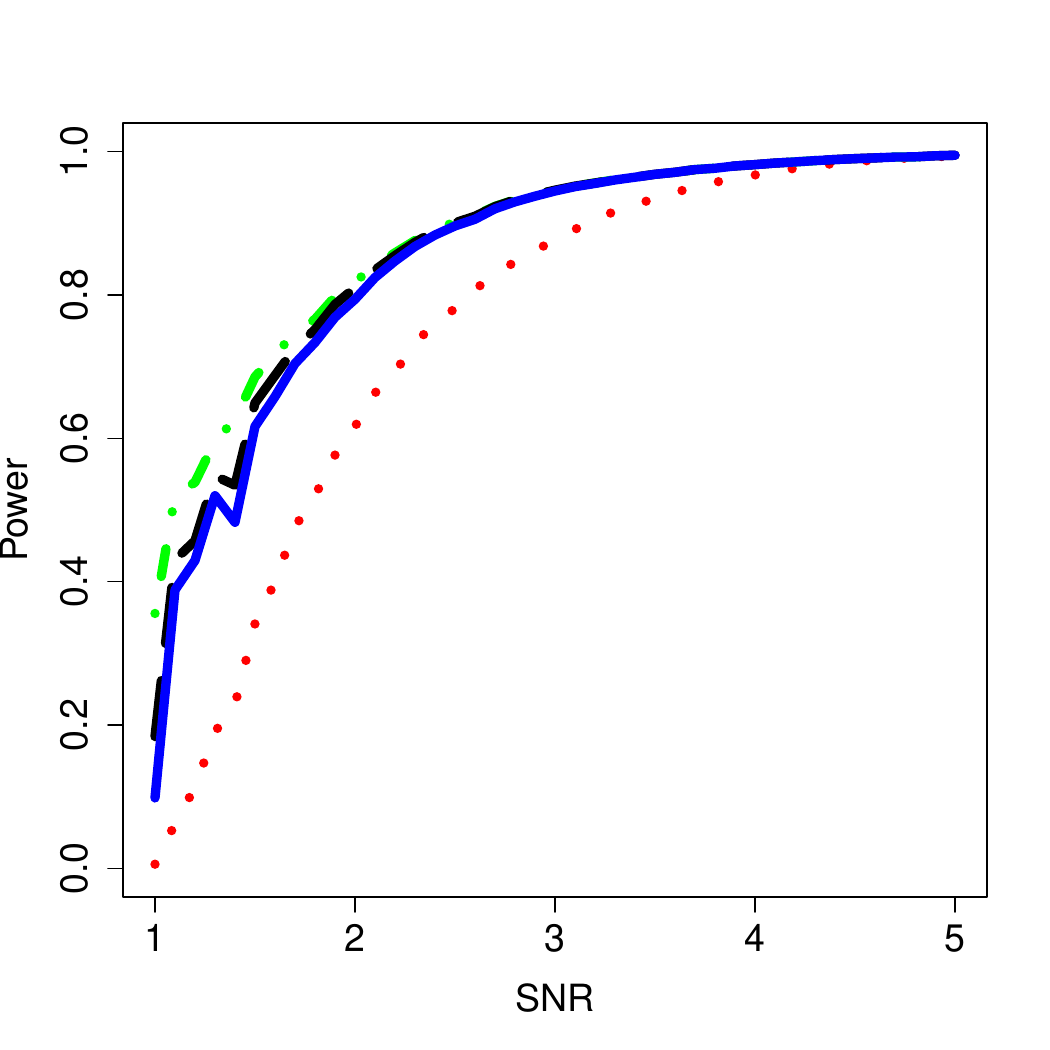}&
\hspace{-5mm}\includegraphics[width=0.34\textwidth,  height = 0.2\textheight,page=2]{subscriptMinP2.pdf}\\
 \vspace{-9mm}\end{tabular}
 \end{center}
\caption{\label{fig-nontrivial-classification} 
Selecting non-trivial prediction sets in the classification iid model, in the case of balanced classes (top row) and unbalanced classes   (bottom row), for the Gaussian mixture model with $K=3$ classes described in \S~\ref{subsec-simul-bivariatenormal}.  FCR (first column), resolution-adjusted power (second column) and the expected fraction of covering prediction sets (third column) versus SNR. 
 The number of data generations was 2000, 1000 data points were used for training, and $n=m=500$. }
\end{figure}

\subsection{Three classes, each a mixture of bivariate normals}\label{sm-large-overlap}
For the iid model, as in \S~\ref{subsec-simul-bivariatenormal}, we consider the balanced and unbalanced cases of 3 classes. However, here we consider that each class comes from a mixture of two bivariate normals, where one component of the mixture is identical in all three classes,  and given the class there is probability half of coming from that common mixture component. Thus the overlap between the classes is much greater than in the settings considered in \S~\ref{subsec-simul-bivariatenormal}. 

In Figure \ref{SM-fig-exchang-hard} we show that the qualitative conclusions with regard to the respective procedures are as in \S~\ref{subsec-simul-bivariatenormal}. Interestingly, in the hardest setting at the bottom row we see that the power of \texttt{InfoSCOP} is even larger than classic conformal. This is because the initial selection step not only prevents paying too much for selection, but it can also improve the calibration set for computing the conformal $p$-values after selection.  In this hard setting, there is a large improvement of the tail probability that matters for inference. We demonstrate this in one realization of the data generation in Figure \ref{SM-fig-exchang-snapshot}, where we can see that the 0.95 quantile of the original calibration set is much larger than the 0.95 quantile of examples that remain in the calibration set after the initial selection step. This is because the data generation is such that examples in the central cloud tend to have larger nonconformity scores $S_{Y_i}(X_i)$ than the examples that are not in the central cloud, and most of the examples from the central cloud  were eliminated from the calibration sample with the initial selection step. So the remaining nonconformity scores after the initial selection step tend to be smaller. 
Since the nonconformity scores of the examples from the test sample are unchanged, the improved calibration set after the initial selection step  results in $p$-values that tend to be smaller when testing the null group, and in particular there are many more $p$-values that are below the 0.05 level.

\begin{figure}[h!]
\begin{center}
\vspace{-2cm}
 {\includegraphics[width=0.36\textwidth, height = 0.27\textheight, page=2]{legends.pdf}}

\vspace{-3cm}
  \begin{tabular}{ccc}

\hspace{-5mm}\includegraphics[width=0.34\textwidth, height = 0.2\textheight, page=3]{subscriptMinP1b.pdf} &
\hspace{-5mm}\includegraphics[width=0.34\textwidth,  height = 0.2\textheight,page=1]{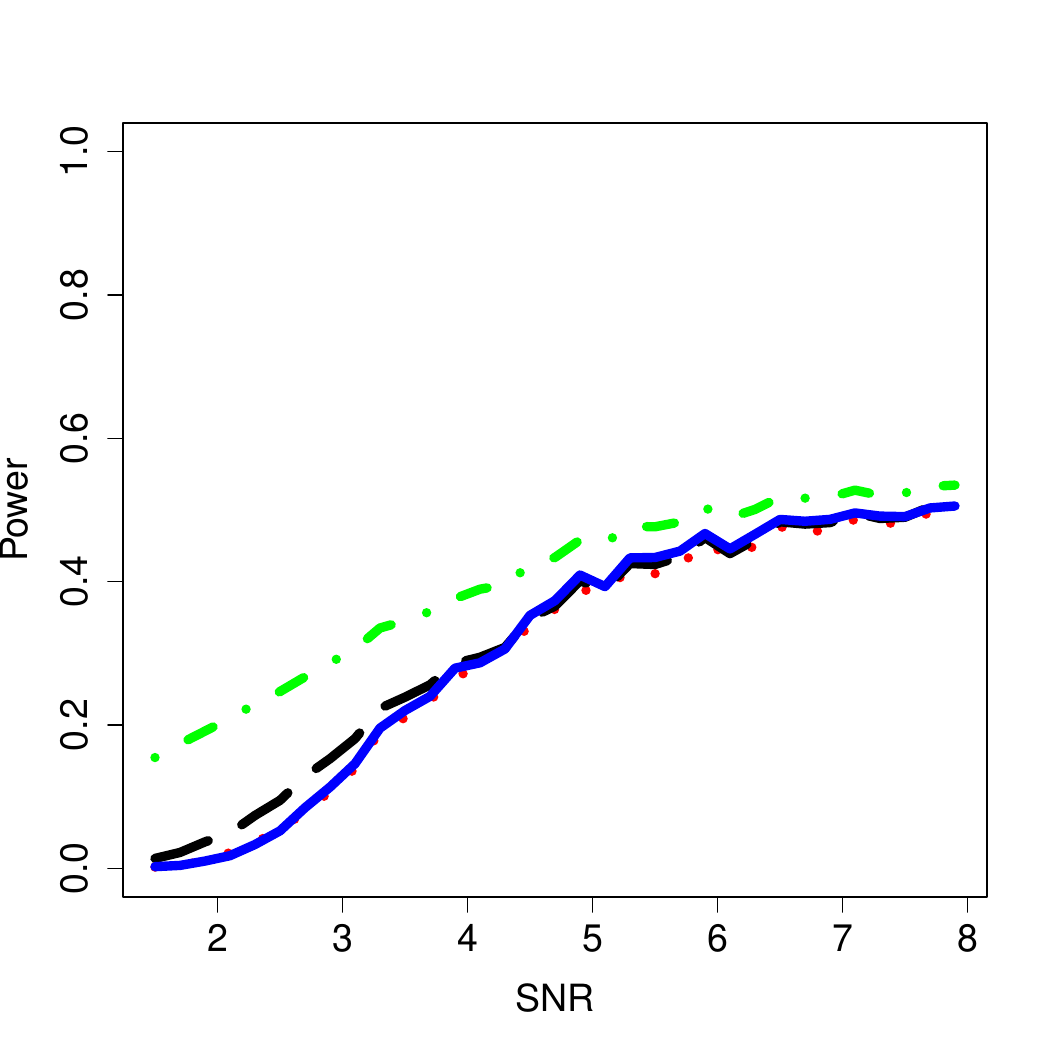} &
\hspace{-5mm}\includegraphics[width=0.34\textwidth,  height = 0.2\textheight,page=2]{subscriptMinP1b.pdf}\vspace{-5mm}\\

\hspace{-5mm}\includegraphics[width=0.34\textwidth, height = 0.2\textheight, page=3]{subscriptnonnullSelection1b.pdf} &
\hspace{-5mm}\includegraphics[width=0.34\textwidth,  height = 0.2\textheight,page=1]{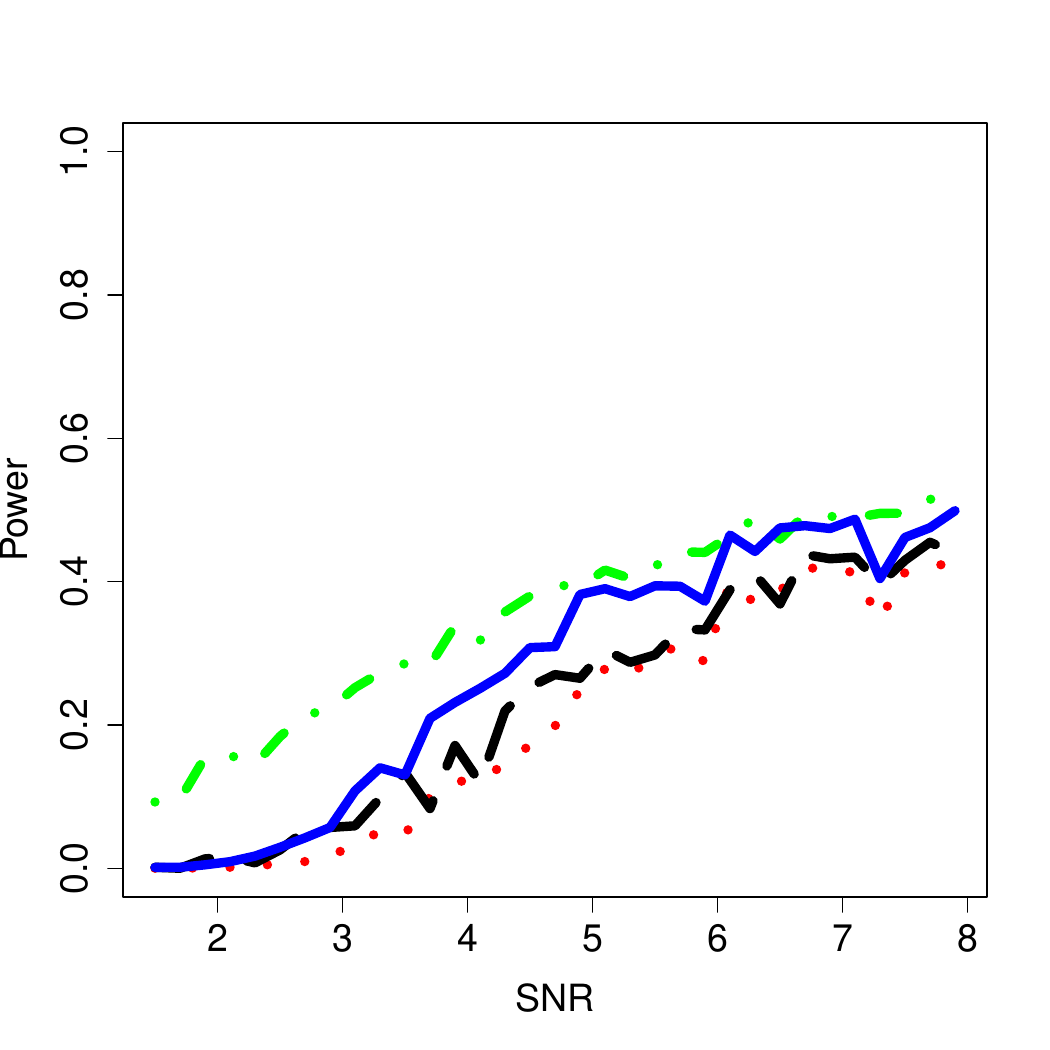}& 
\hspace{-5mm}\includegraphics[width=0.34\textwidth,  height = 0.2\textheight,page=2]{subscriptnonnullSelection1b.pdf}\vspace{-5mm}\\

\hspace{-5mm}\includegraphics[width=0.34\textwidth, height = 0.2\textheight, page=3]{subscriptMinP2b.pdf}& 
\hspace{-5mm}\includegraphics[width=0.34\textwidth,  height = 0.2\textheight,page=1]{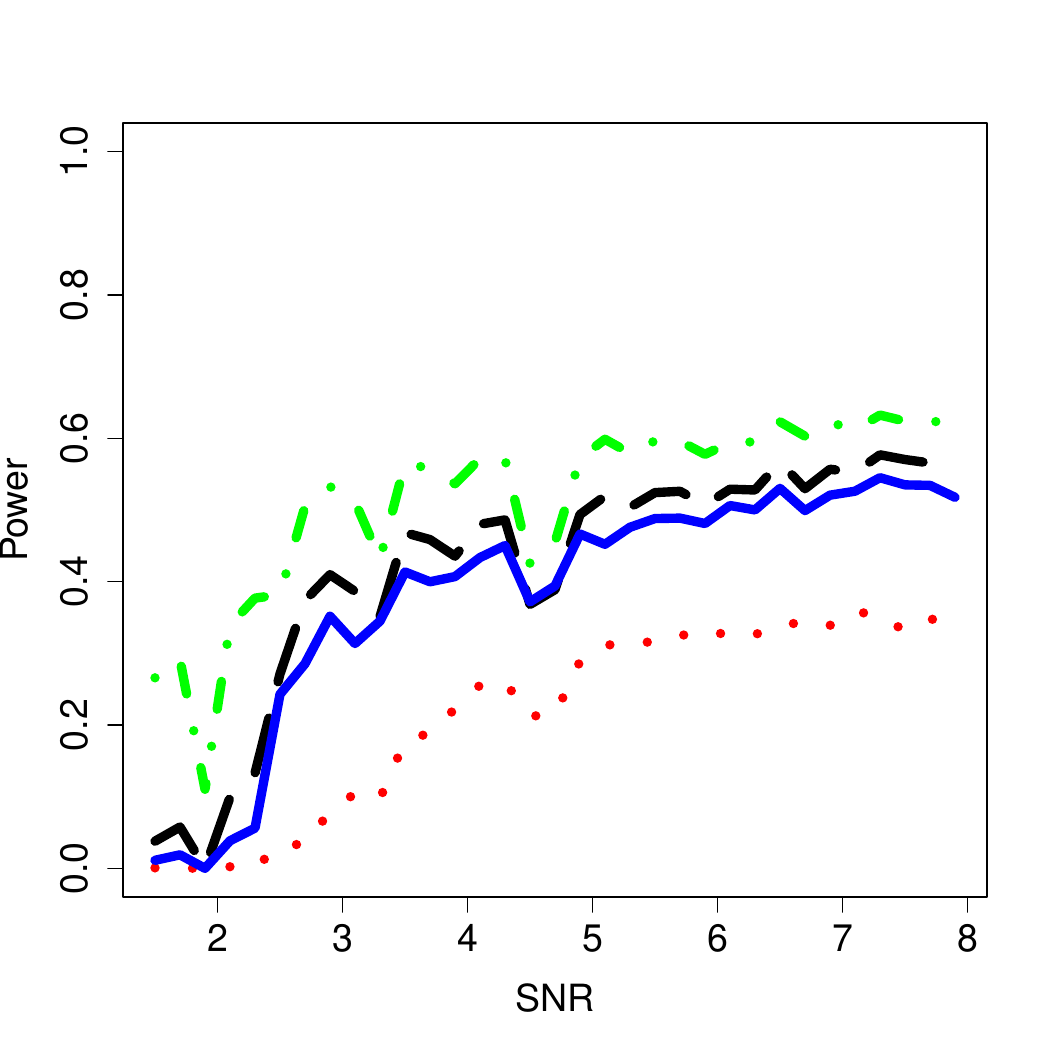}&
\hspace{-5mm}\includegraphics[width=0.34\textwidth,  height = 0.2\textheight,page=2]{subscriptMinP2b.pdf}\vspace{-5mm}\\

\hspace{-5mm}\includegraphics[width=0.34\textwidth, height = 0.2\textheight, page=3]{subscriptnonnullSelection2b.pdf}&
\hspace{-5mm}\includegraphics[width=0.34\textwidth,  height = 0.2\textheight,page=1]{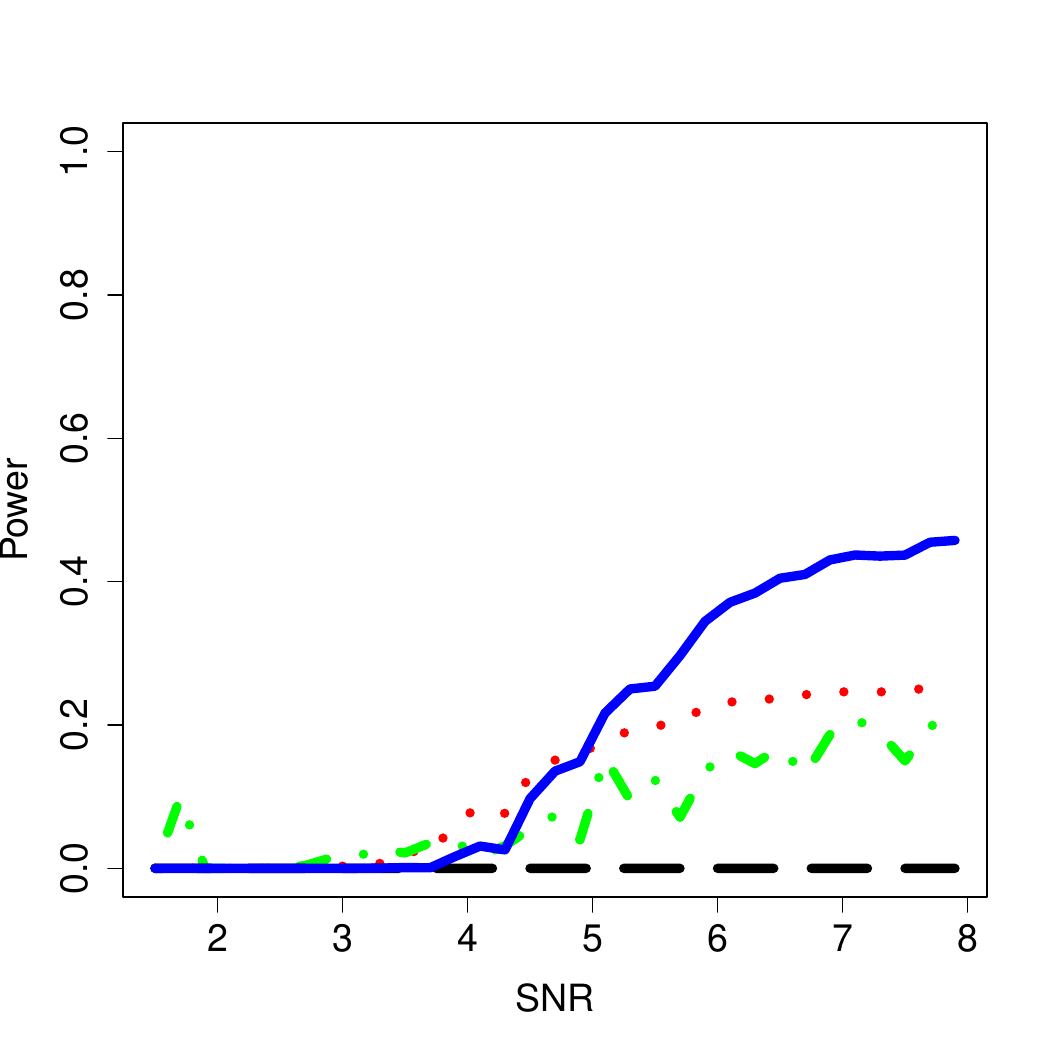}&
\hspace{-5mm}\includegraphics[width=0.34\textwidth,  height = 0.2\textheight,page=2]{subscriptnonnullSelection2b.pdf}\vspace{-5mm}\\
\end{tabular}
\end{center}
\caption{\label{SM-fig-exchang-hard} Selecting informative prediction sets in the iid setting with a large overlap between three classes. FCR (left column), resolution-adjusted power (middle column) and the expected fraction of covering prediction sets (right column) versus SNR for: 
 balanced classes  for minimally informative prediction sets (first row) and for prediction sets excluding a null class (second row); unbalanced classes  for minimally informative prediction sets (third row) and for prediction sets excluding the largest class (fourth row). Each class is a mixture of a common component and a unique component, see \S~\ref{sm-large-overlap} for details.  Based on 2000 data generations, 1000 data points were used for training, and  $n = m = 500$.  } 
\end{figure}

\begin{figure}[h!]
    \centering
    \includegraphics[width=0.38\textwidth, height = 0.38\textheight, page=2]{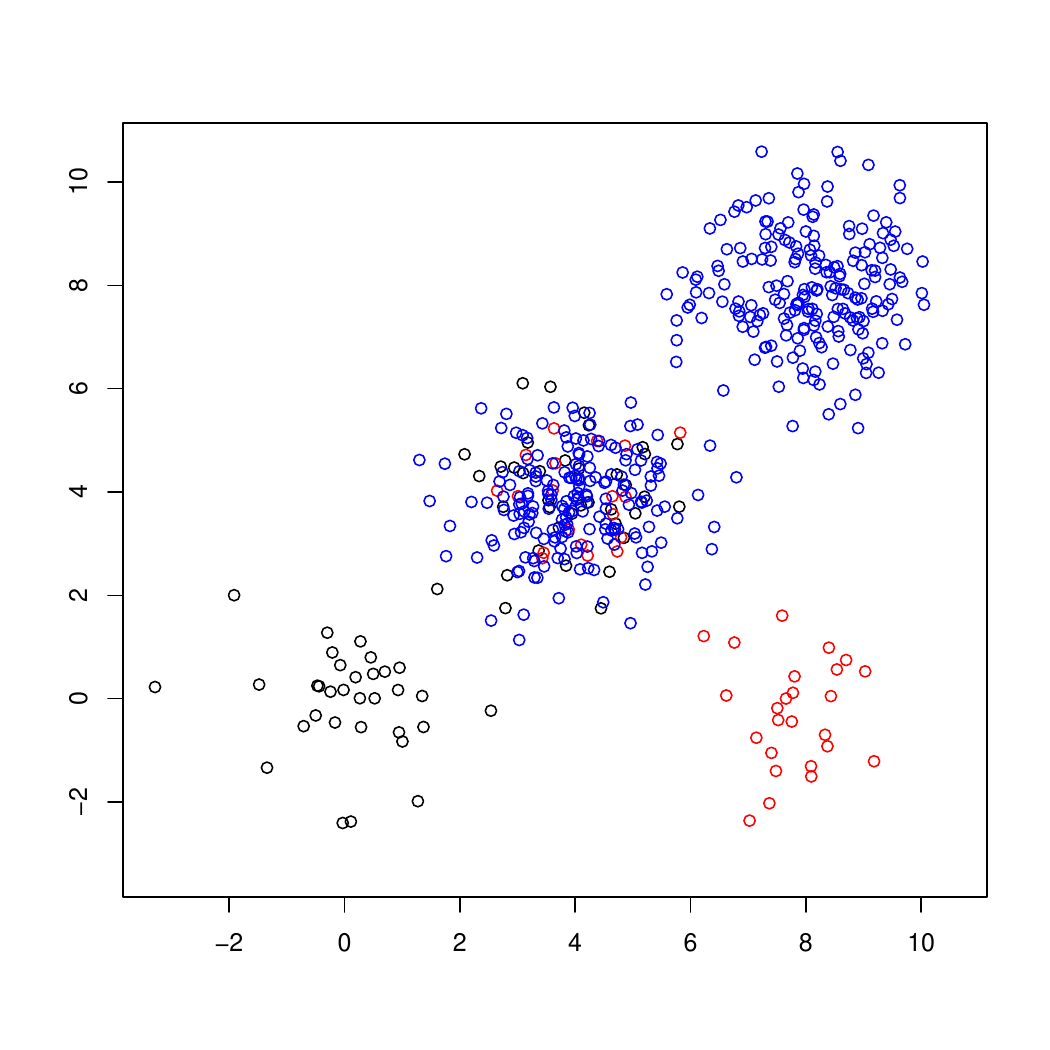}
    \includegraphics[width=0.38\textwidth, height = 0.38\textheight, page=3]{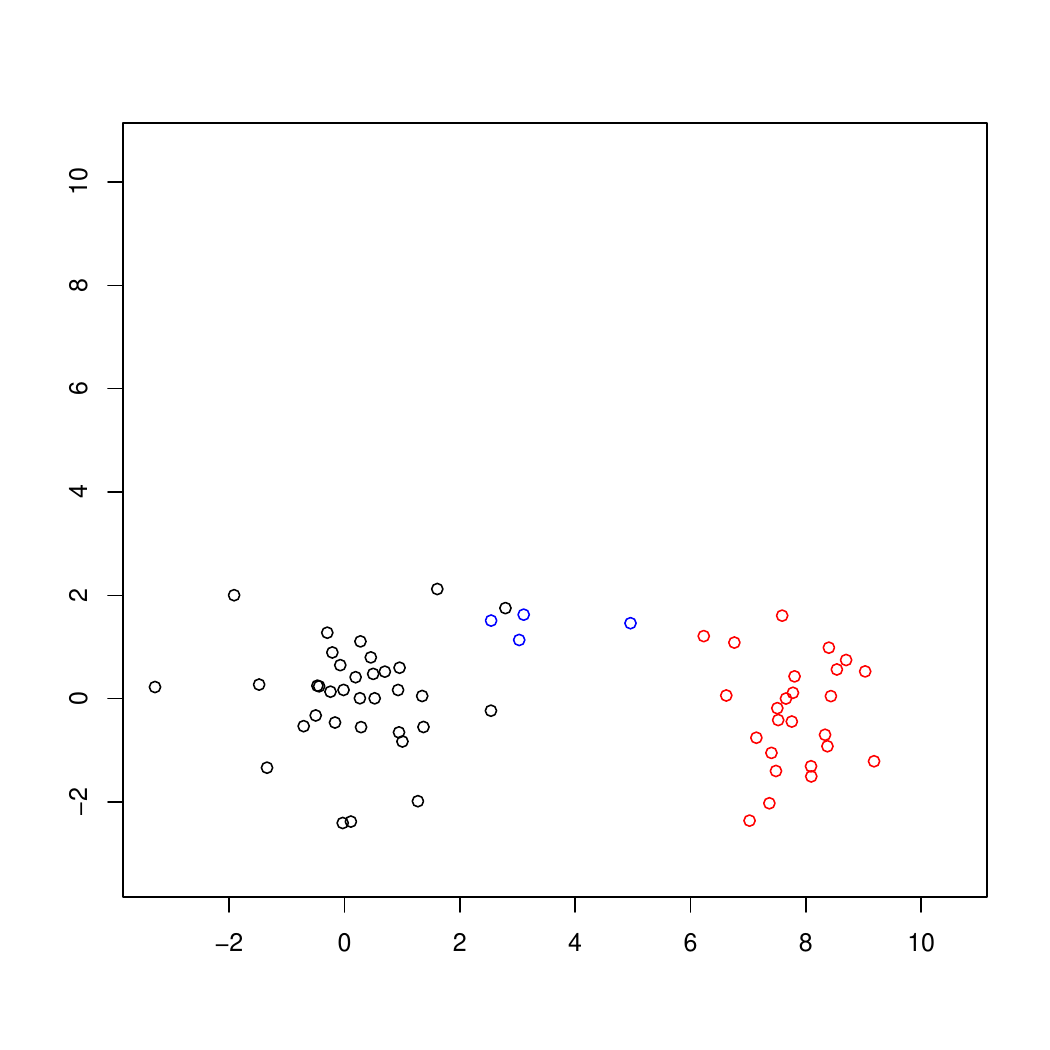}\vspace{-5mm}\\
    \includegraphics[width=0.38\textwidth, height = 0.38\textheight, page=5]{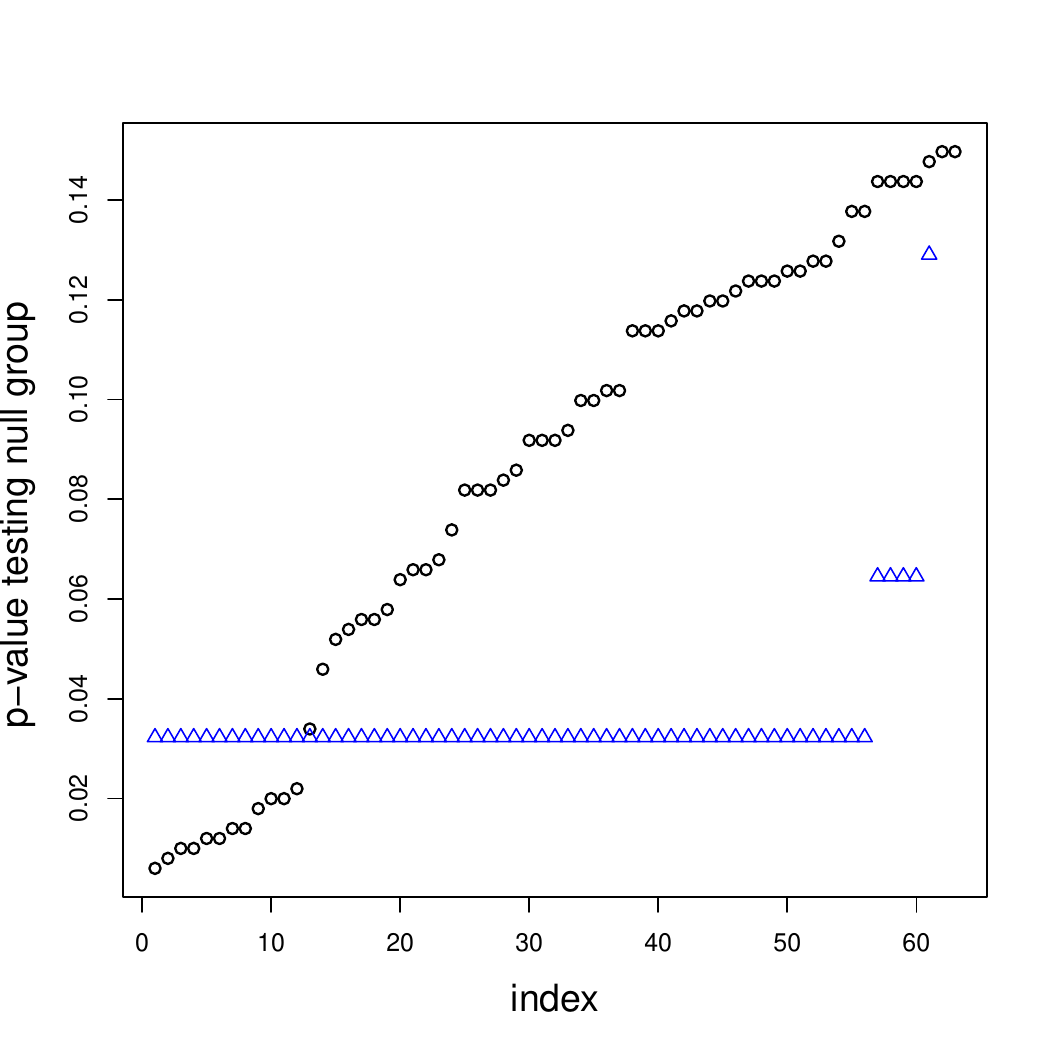}
    \includegraphics[width=0.38\textwidth, height = 0.38\textheight, page=6]{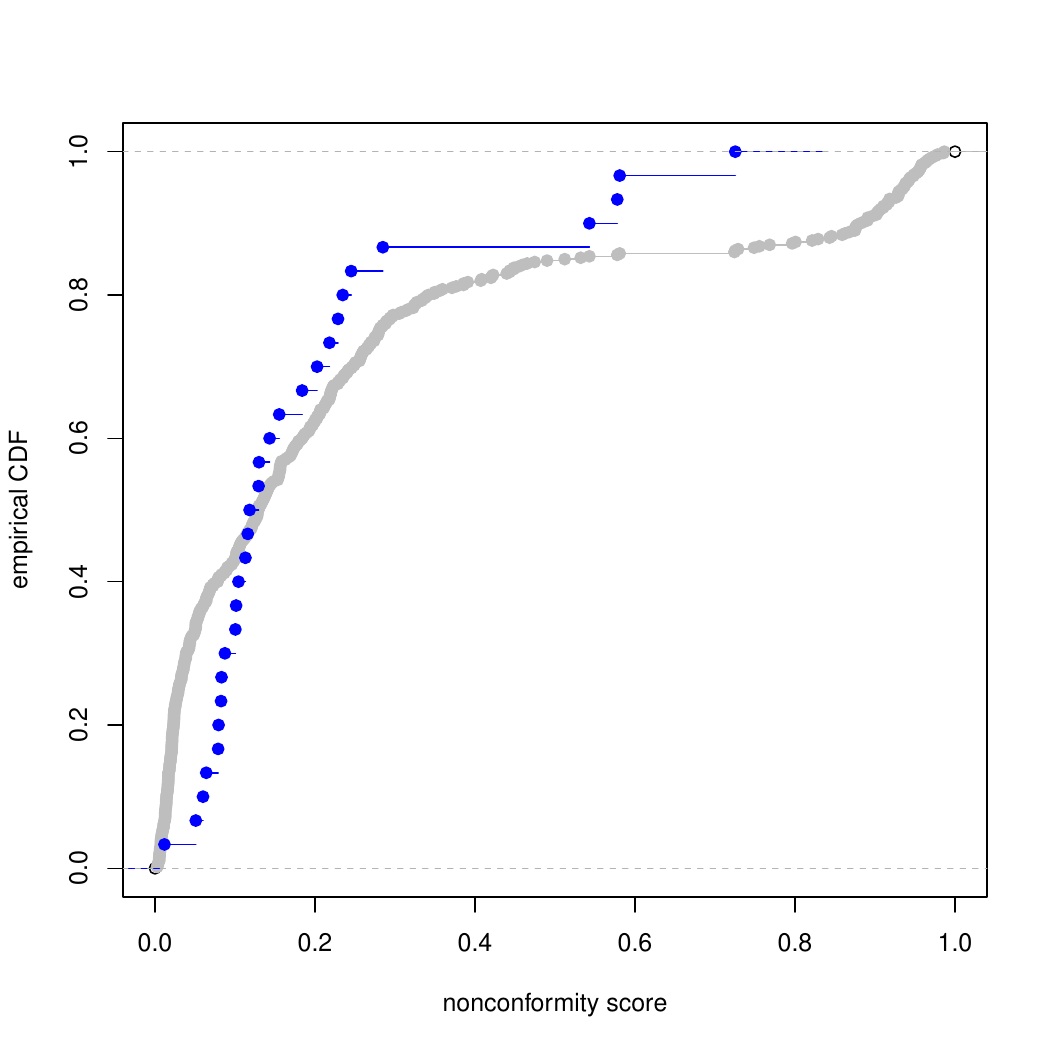}
    \vspace{-5mm}
    \caption{\label{SM-fig-exchang-snapshot}The setting is that of unbalanced classes, with each class having probability half of being in the same mixture component, and the SNR is 8. In the first row, the data points from each of the three classes, where the null group is in blue:  top left panel for the entire test sample, top right panel the remaining examples after the initial selection step. In the second row, left panel, their $p$-values using the entire calibration set (black circles) and using the examples from the calibration set remaining after the initial selection step (blue triangles). In the second row, right panel, the empirical CDF of the nonconformity scores for true classes in the entire calibration set (gray) and using the examples from the calibration set remaining after the initial selection step for the of the   (blue).  }
    \label{fig:enter-label}
\end{figure}

\subsection{{Three classes of animals}}

For the class-conditional model with label shift described in \S~\ref{sec:appli}, Figure \ref{fig:app2} provides the resulting FCR and power measures. 

\begin{figure}
    c) Non-trivial classification with label shift
    \begin{center}
    \includegraphics[width=0.9\linewidth]{"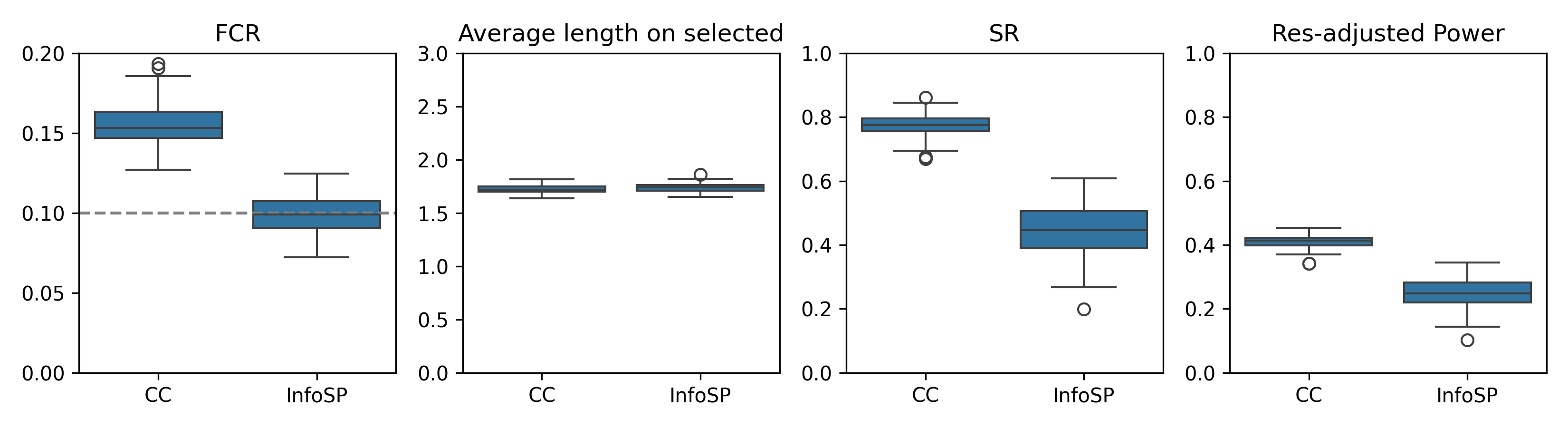"}\\
    \end{center}
    d) Non-null classification with label shift
    \begin{center}
    \includegraphics[width=0.9\linewidth]{"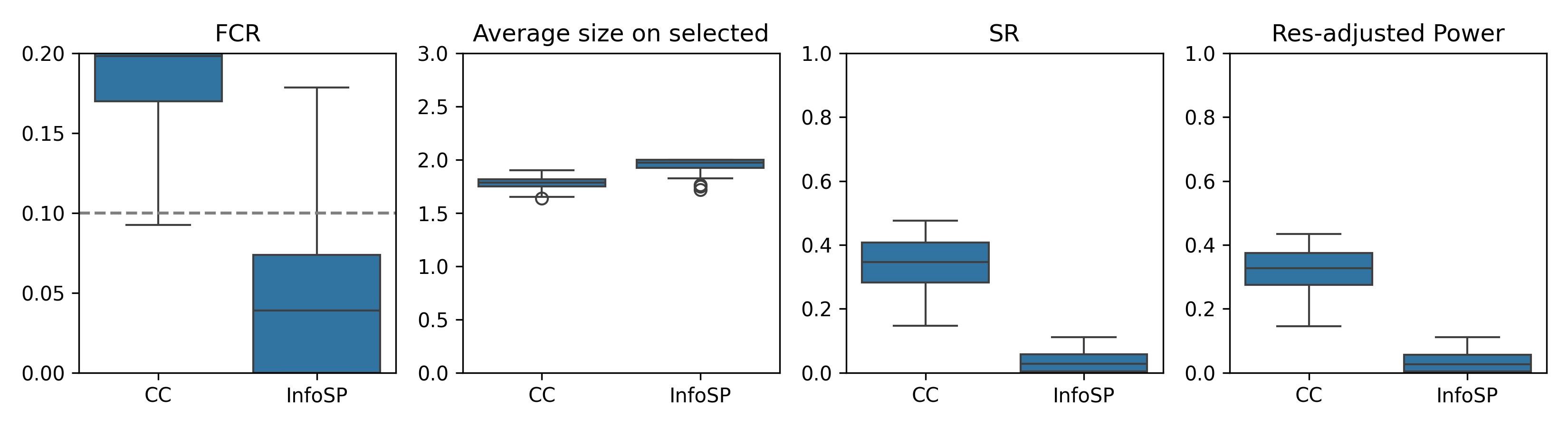"}\\
    \end{center}
    \caption{For the class-conditional model, the FCR, average size of the selected, SR, and resolution-adjusted power for the methods, for $\alpha=0.1$. }
    \label{fig:app2}
\end{figure}

\end{document}